\documentclass[10pt]{article}
\usepackage{amssymb,amsmath,amsthm,epsfig}
\usepackage{latexsym, enumerate}
\usepackage{eepic}
\usepackage{epic}
\usepackage{graphicx}
\usepackage{color}
\usepackage{ifpdf}
\usepackage{subfigure}
\usepackage{tikz}
\usepackage{dsfont}
\usepackage{multirow}
\usepackage{makecell}
\usepackage{algorithm}
\usepackage{bm}
\usepackage{multirow}
\usepackage{color}
\usepackage[colorlinks, linkcolor=blue,anchorcolor=blue,citecolor=blue,urlcolor=blue]{hyperref}
\usepackage{appendix}
\usepackage{comment}
\graphicspath{{figs/}}

\theoremstyle{plain}
\newtheorem{lem}{Lemma}[section]
\newtheorem{thm}[lem]{Theorem}
\newtheorem{cor}[lem]{Corollary}

\theoremstyle{definition}
\newtheorem{defn}{Definition}[section]

\theoremstyle{remark}
\newtheorem{rem}{Remark}[section]

\newcommand{\p}{\partial}
\newcommand{\ds}{\displaystyle}

\newcommand{\ms}{\medskip}
\newcommand{\R}{ \mathbb{R}}

\def \e{\ensuremath{\mathrm{e}}}
\def \i{\ensuremath{\mathrm{i}}}
\def \d{\ensuremath{\mathrm{d}}}
\begin{document}

\title{ \large\bf A mathematical theory of  microscale hydrodynamic cloaking and shielding by electro-osmosis}
\author{
Hongyu Liu\thanks{Department of Mathematics, City University of Hong Kong, Kowloon, Hong Kong, China.\ \ Email: hongyu.liuip@gmail.com; hongyliu@cityu.edu.hk}
\and
Zhi-Qiang Miao\thanks{College of Mathematics, Hunan University, Changsha 410082, Hunan Province, China. Email: zhiqiang\_miao@hnu.edu.cn}
\and
Guang-Hui Zheng\thanks{Corresponding author. College of Mathematics, Hunan University, Changsha 410082, Hunan Province, China. Email: zhenggh2012@hnu.edu.cn; zhgh1980@163.com}
}
\date{}%
\maketitle
\begin{abstract}
In this paper, we develop a general mathematical framework for perfect and approximate hydrodynamic cloaking and shielding of electro-osmotic flow, which is governed by a coupled PDE system via the field-effect electro-osmosis. We first establish the representation formula of the solution of the coupled system using the layer potential techniques. Based on Fourier series, the perfect hydrodynamic cloaking and shielding conditions are derived for the control region with the cross-sectional shape being annulus or confocal ellipses. Then we further propose an optimization scheme for the design of approximate cloaks and shields within general geometries. The well-posedness of the optimization problem is proved. In particular, the condition that can ensure the occurrence of approximate cloaks and shields for general geometries are also established. Our theoretical findings are validated and supplemented by a variety of numerical results. The results in this paper also provide a mathematical foundation for more complex hydrodynamic cloaking and shielding.
\end{abstract}
\section{Introduction}
During the past 16 years, there has been rapid progress in rendering objects invisible by cloaking them with metamaterials in physics and mathematics subjects. One influential methodology to design metamaterials is the transformation-based approach. Since the pioneering works \cite{GLU,Pendry2006, Leonhardt2006}, many studies about cloaking based on the transformation theory appear in succession in various scientific areas, including acoustic waves \cite{ammari2013, deng2017,Kohn2010, Liu2009}, conductive heat flux  \cite{Craster2017, Narayana2012}, electromagnetic waves\cite{bao2014, deng2017(1), greenleaf2009}, stresses in an elastic medium  \cite{li2018, Stenger2012, li2016}, dc electric currents \cite{Yang2012}, and quantum mechanical matter waves \cite{Zhang2008}.
Effectively, these cloaks are shells composed of metamaterials with a tailor-made distribution of the effective material parameters, determined by the invariance of governing equations under a transformation of the spatial coordinates. However, the requirement for inhomogeneous and anisotropic parameters makes it difficult to fabricate devices designed by transformation optics. Consequently, the feasibility of fabricating cloaks using composite natural materials has attracted great interest. For this purpose, scattering-cancellation technology has been developed and successfully used in electromagnetism \cite{Alu2005} and other fields \cite{Chen2012}.
Generally speaking, this method can realize a similar function to transformation optics, while it only needs bilayer or monolayer structures and homogeneous isotropic bulk materials.
In addition to the bulk materials, the metasurface cloak has also received much attention with the advantages of low weight and thinness. Instead of employing a shell, the metasurface cloak uses an ultrathin frequency-selective surface which is designed so that the induced currents along the surface cancel the scattering from the object to a cloak  \cite{Alu2009, Alu2011, Chen2011, Padooru2012}. Such a strategy is particularly relevant at microwave frequencies at which frequency-selective surfaces are readily available and easy to fabricate \cite{Munk2005, Tretyakov2003}.

In fluid, hydrodynamic cloaking has been a subject of intense research all the time. There are reports in the physics literature. The first systematic analysis of hydrodynamic cloaking is due to Urzhumov and Smith \cite{Urzhumov2011}, based on transformation theory. Later, they further extended the analysis to the case of a two-dimensional flow around a cylinder in a medium with mixed positive and negative permeability \cite{Urzhumov2012}.  For the high Reynolds numbers, it was demonstrated how to confine the water waves in a certain area for cloaking regions \cite{Zou2019, Zhang2020}.
More recently, another hydrodynamic model has been used to control fluid motion, i.e., the creeping flow or Stokes flow (Reynolds number $Re \ll 1$ ) inside two parallel plates, and a series of experimental works have been reported \cite{Park2019a, Park2019b, Park2021, Boyko2021}. The gap between two plates is much smaller than the characteristic length of the other two spatial dimensions, so the model is also called the Hele-Shaw flow or Hele-Shaw cell \cite{Hele1898}.  By using these microfluidic structures, Park et al  \cite{Park2019a, Park2019b} have demonstrated by simulation that such anisotropic fluid media can be mimicked within the cloak, thereby producing the desired hydrodynamic cloaking effect. We should mention that a transformed fluid medium is not yet physically realizable; proof of concepts is so far by numerical simulation. The implementation of such a transformation-based fluid-flow cloak relied on 10 layers of metamaterial microstructures, as well as a fluid background filled with microcylinders to avoid impedance mismatch. Hence, there has been a growing interest in realizing metamaterial-less hydrodynamic cloaks.

In particular,  in \cite{Boyko2021} Boyko et al present a new theoretical approach and an experimental demonstration of hydrodynamic cloaking and shielding in a Hele-Shaw cell that does not rely on metamaterials. In analogy to optics, “hydrodynamic cloaking” refers to a state wherein the (flow) field external to some region around an object is unaffected by the presence of the object. “Hydrodynamic cloaking” refers to the elimination of hydrodynamic forces on the object.  The mechanism is based on the fact that flow fields on small scales are completely governed by momentum sources at boundaries and can be linearly superposed. An effective way to create such momentum sources is electro-osmotic flow (EOF)---an electrokinetic phenomenon that generates flow due to the interaction of an applied electric field with the native or induced net charge at liquid-solid interfaces \cite{Hunter2001}. The fluid pressure in the Hele-Shaw cell also satisfies a Laplacian-type equation, which has the analytical solutions for circular cylinder objects in the polar coordinate. With analytical solutions, Boyko et al first achieve both cloaking and shielding conditions for circular cylinder objects. Further, they establish approximate cloaking and shielding for the more complex shape of a slight deform from a perfectly cylindrical shape using shape perturbation theory.

In this paper, we greatly extend the results of Boyko et al \cite{Boyko2021} and establish a more general mathematical framework. For details, the contributions of this work are fourfold:
\begin{itemize}
\item{Based on the physics literature \cite{Boyko2021}, we give the rigorous mathematical definition of hydrodynamic shielding and cloaking.}
\item{The representation formula of solution of the coupled system is obtained firstly, which gives a quantitative description of the hydrodynamic model.}
\item{By using the uniform approach--layer potential theory, we establish sharp conditions that can ensure the occurrence of the hydrodynamic cloaking and shielding for annulus (radial case) and confocal ellipses (non-radial case). Especially, for the confocal ellipses case which is not considered in \cite{Boyko2021}, we introduce additional elliptic coordinates technique to overcome the difficulty caused by non-radial geometry.}
\item{For more general geometry, we further propose an optimization method to design the hydrodynamic shielding and cloaking, and prove the well-posedness of optimization problem. More important, the condition that can ensure the occurrence of approximate cloaks and shields for arbitrary-geometry are also established. A large amount of numerical experiments, which contains smooth objects, non-smooth objects and multiple objects, indicates that the optimized zeta potential can achieve the hydrodynamic shielding and cloaking effectively.}
\end{itemize}

The paper is organized in the following way. We begin with the mathematical setting of the problem and statement of the main results in Section \ref{sec-problem}. In Section \ref{sec-layer-potentials}, we first present some preliminary knowledge on boundary layer potentials and then establish the representation formula of solution of the governing equations.
Section \ref{sec-cloaking-shilding} is devoted to the study of the perfect and approximate cloaking and shielding conditions by the analytical method and the optimization method, respectively.
In Section \ref{sec-NumSim}, we conduct numerical experiments to corroborate our theoretical findings. The paper is concluded in Section~\ref{sect:6} with some relevant discussions. 

\section{Mathematical setting of the problem and statement of the main results}\label{sec-problem}
To begin, we consider the creeping flow of a viscous fluid of density $\rho$, viscosity $\mu$, and dielectric permittivity $\varepsilon$ within a
narrow gap $\tilde{h}$ between two parallel plates of length $\tilde{L}$ and width $\tilde{W}$ forming a Hele-Shaw configuration in $\R^3$, as shown
in Figure \ref{fig-schematic}(a). We employ a Cartesian coordinate system $\tilde{x} =(\tilde{x}_1, \tilde{x}_2, \tilde{x}_3) \in \R^3$, where the $\tilde{x}_1$ and $\tilde{x}_2$ axes lie at the lower plane and the $\tilde{x}_3$ axis is perpendicular thereto. Let $\tilde{\zeta}^L = \tilde{\zeta}^L(\tilde{x})$ and $\tilde{\zeta}^U = \tilde{\zeta}^U(\tilde{x})$ be an arbitrary zeta-potential distribution in the lower and upper plates. Let $\tilde{\bm{v}} = (\tilde{\bm{u}}(\tilde{x}, \tilde{t}),\tilde{w}(\tilde{x}, \tilde{t}))$ and $\tilde{p} = \tilde{p}(\tilde{x})$ denote the velocity field and  the pressure,  respectively, of the fluid. Here $\tilde{\bm{u}}(\tilde{x}, \tilde{t})$ is  the in-plane velocity field. When an electrostatic in-plane electric field $\tilde{\bm E}$ is applied
parallel to the plates, the fluid motion is then governed by the continuity and momentum equations
\begin{equation*}
  \tilde{\nabla}\cdot \tilde{\bm v} = 0, \qquad \rho\Big( \frac{\p \tilde{\bm v}}{\p \tilde{t}}+ \tilde{\bm v}\cdot  \tilde{\nabla}\tilde{\bm v}\Big)
 =  -  \tilde{\nabla} \tilde{p} + \mu \tilde{\Delta}  \tilde{\bm v},
\end{equation*}
with  the Helmholtz--Smoluchowski slip boundary conditions
\begin{equation*}
  \tilde{\bm{u}}|_{\tilde{x}_3=0}= -\frac{\varepsilon\tilde{\zeta}^L \tilde{\bm E}}{\mu}, \qquad   \tilde{\bm{u}}|_{\tilde{x}_3=\tilde{h}}= -\frac{\varepsilon\tilde{\zeta}^U \tilde{\bm E}}{\mu}.
\end{equation*}
These equations contain three independent dimensional parameters, the density $\rho$,  the viscosity $\mu$ and the dielectric permittivity $\varepsilon$, which are both constant and assumed to be known. For convenience in what follows, we mark the dimensional variables that appear in these equations with a tilde, for instance, $\tilde{\bm u}$, $\tilde{p}$, and so on. In addition, we also denote the dimensional gradient and Laplace operator as $\tilde{\nabla}$ and $\tilde{\Delta}$.

By making use of the Helmholtz--Smoluchowski slip condition, we have implicitly assumed that surface conduction
is negligible and thus consider asymptotically small Dukhin numbers. Under this assumption, the fluid in the bulk is electrically neutral and the electric field $\tilde{\bm E}$ is solenoidal, i.e., $\tilde{\nabla} \cdot \tilde{\bm E}  = 0$.  The electric field can be expressed through an electrostatic potential $\tilde{\varphi}$, $\tilde{\bm E} = -\tilde{\nabla} \tilde{\varphi} $, that is governed by the Laplace equation $\Delta \tilde{\varphi} =0$, which satisfies the insulation boundary conditions on the walls, $\frac{\p \tilde{\varphi} }{\p \tilde{x}_3}|_{\tilde{x}_3=0, \tilde{h}} = 0$.

In microscale flows, fluidic inertia is commonly negligible compared to viscous stresses. Further, the gap is narrow. Therefore, we restrict our analysis to shallow geometries ($\tilde{h}\ll \tilde{L}, \tilde{W}$) and neglect fluidic inertia, represented by a small reduced Reynolds number.
Applying the lubrication approximation, we average over the depth of the cell and reduce the analysis to a two-dimensional problem in $\R^2$. The governing equations for the depth-averaged velocity $\bm \tilde{\bm u}_{aver}$, the pressure $\tilde{p}$, and the electrostatic potential $\tilde{\varphi}$ in $\tilde{x}_1$--$\tilde{x}_2$ plane are (see \cite{Boyko2021} )
\begin{equation}\label{dimension-equations}
 \tilde{\bm u}_{aver} = - \frac{\tilde{h}^2}{12 \mu} \tilde{\nabla} \tilde{p} + \tilde{\bm u}_{slip},  \quad \tilde{\Delta}\tilde{p} = \frac{12 \varepsilon}{\tilde{h}^2}\tilde{\nabla}\tilde{\varphi}\cdot \tilde{\nabla}\tilde{\zeta}_{mean} \quad \mbox{and} \quad \tilde{\Delta} \tilde{\varphi} = 0,
\end{equation}
where $\bm \tilde{\bm u}_{aver}$ is the average value of $\tilde{\bm u}$ over $\tilde{x}_3$ by integration, $\tilde{\bm u}_{slip}=\frac{\tilde{\varepsilon}\langle \tilde{\zeta}\rangle}{\tilde{\mu}}\tilde{\nabla}\tilde{\varphi}$ is the depth-averaged Helmholtz-Smoluchowski slip velocity, ${\zeta}_{mean}$ is the arithmetic mean value of the zeta potential on the lower and upper plates, i.e., ${\zeta}_{mean}=(\tilde{\zeta}^L+\tilde{\zeta}^U)/2$. In the later analysis, we assume the zeta potential on the lower and upper plates is same.
Scaling by the characteristic dimensions, we introduce the following dimensionless equations:
\begin{equation}\label{dimensionless-equations}
 \bm u_{aver}  =- \frac{1}{12} \nabla p - \zeta_{mean}  \nabla \varphi,\quad \Delta p = -12 \nabla \varphi \cdot \nabla \zeta_{mean}  \quad \mbox{and} \quad \Delta \varphi = 0,
\end{equation}
where $\bm u_{aver}$, $p$, $\varphi$ and $\zeta_{mean}$ are non-dimensional normalized variables.
\begin{figure}[H]
	\centering  
	\subfigcapskip=-10pt 
	\subfigure[]{
		\includegraphics[width=0.5\linewidth]{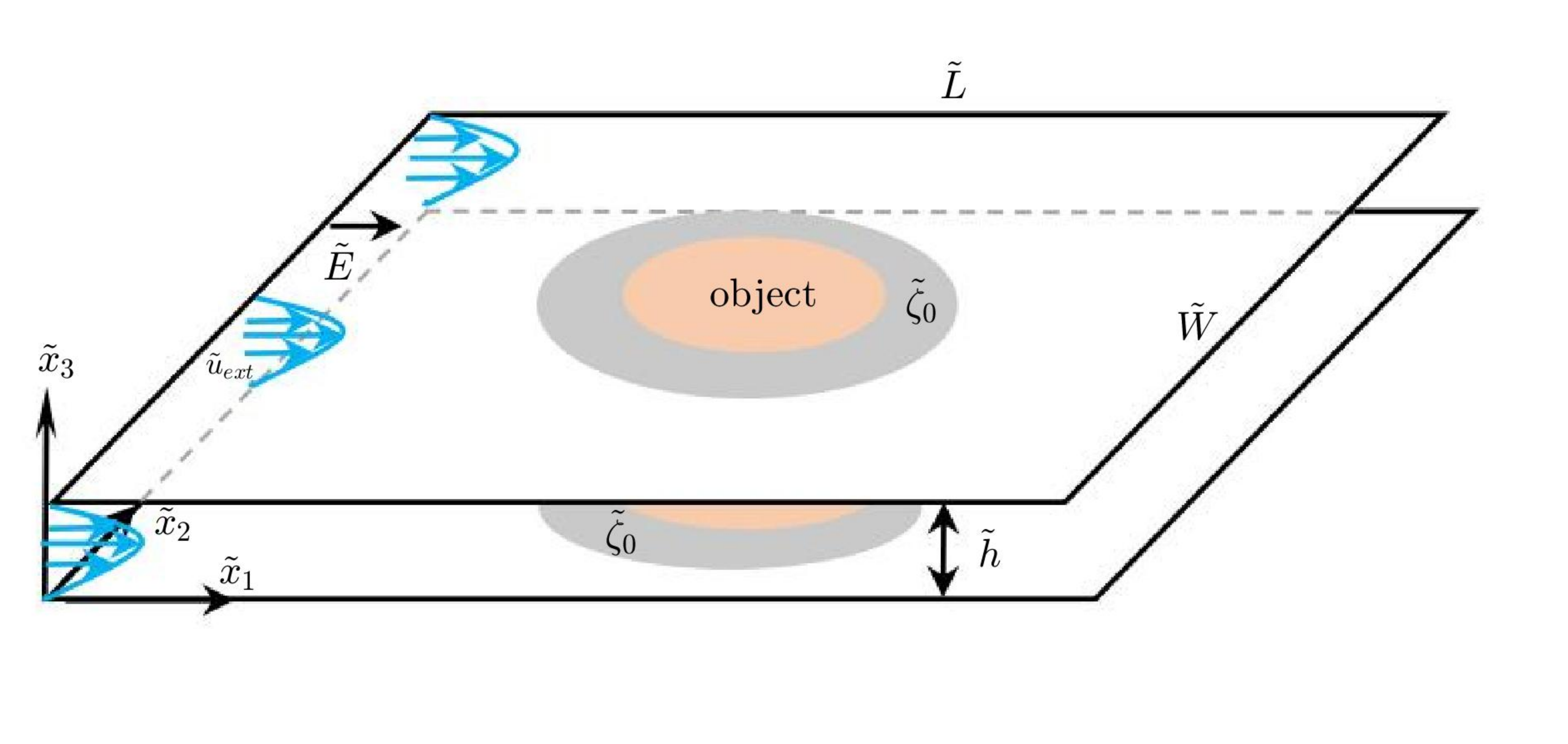}}
	\subfigure[]{
		\includegraphics[width=0.35\linewidth]{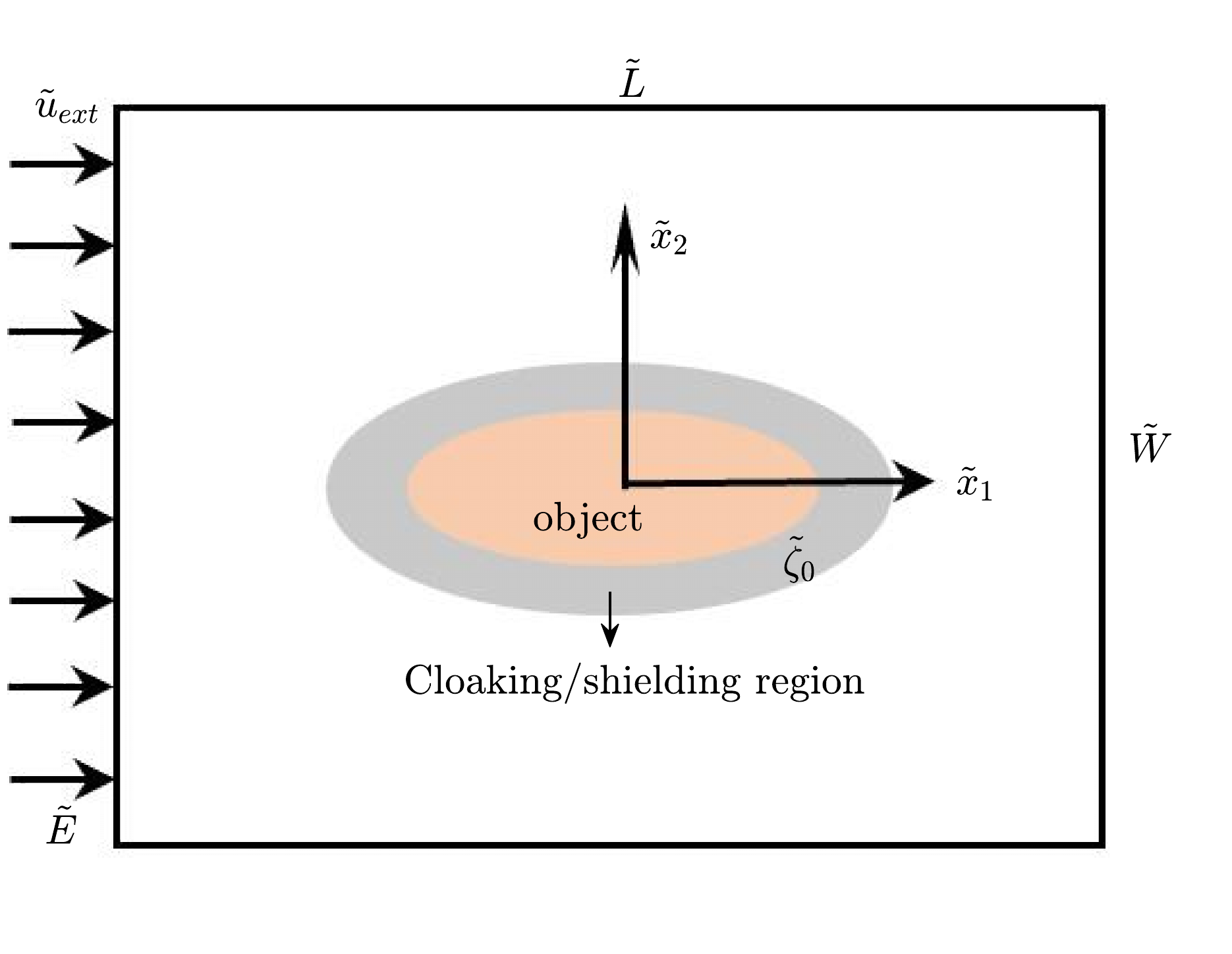}}
	\caption{Schematic illustration of the Hele-Shaw configuration. (a) the three-dimensional electro-osmosis model. (b) the reduced two-dimensional problem.
 }\label{fig-schematic}
\end{figure}

In addition to these equations and the geometry of the flow domain, we also require the boundary conditions that apply at the boundaries of the domain to completely characterize the flow. These boundary conditions include the form and magnitudes of the velocity on the left inlet and right outlet to the domain.
In this paper, we mainly consider a pillar-shaped object with arbitrary cross-sectional shape confined between the walls of a Hele-Shaw cell and subjected to a pressure-driven flow with an externally imposed mean velocity $\tilde{u}_{ext}$ and electric field $\tilde{E}$ along the $\tilde{x}_1$--axis, as shown in Figure \ref{fig-schematic}. The reduced two-dimensional problem is shown in Figure \ref{fig-schematic}(b).
To cloak and shield the object hydrodynamically, we are concerned with the scattering by the object surrounded by a region, i.e., with exterior boundary value problems for the Laplace equation. We solve the problem assuming an unbounded domain, enabled by the fact that the boundaries of the chamber are located far from the cloaking and shielding region.

To mathematically state the problem, let $\Omega$ be a bounded domain in $ \mathbb{R}^2$ and let $D$ (object) be a domain whose closure is contained in $\Omega$. Throughout this paper, we assume that $\Omega$ and $D$ are of class $C^{1,\alpha}$ for some $0<\alpha<1$.  Let $H(x)$ and $P(x)$ be the harmonic function  in $\mathbb{R}^2$, denoting the background electrostatic potential and pressure field. For a given constant parameter $\zeta_0 \in \R$, the  zeta potential distribution in
$\mathbb{R}^2\setminus\overline{D}$ is given by
\begin{align*}
          \zeta_{mean}   =
            \begin{cases}
            \ds \zeta_0 &  \mbox{in } \Omega \setminus\overline{D},\ms\\
            \ds 0 & \mbox{in } \mathbb{R}^2\setminus\overline{\Omega}.
            \end{cases}
\end{align*}
We may consider the configuration as an insulation and no-penetration core coated by the shell (control region) $\Omega \setminus\overline{D}$ with zeta potential $ \zeta_0$.  From the equations \eqref{dimensionless-equations} and the assumption of unbounded domain, the governing equations for non-uniform electro-osmotic flow via a Hele-Shaw configuration is modeled as follows:
\begin{align}\label{electro-osmotic equation}
\begin{cases}
\ds \Delta \varphi = 0  & \mbox{in } \mathbb{R}^2\setminus\overline{D}, \ms \\
\ds \frac{\partial \varphi}{\partial \nu} = 0  & \mbox{on } \partial D, \ms \\
\ds  \varphi = H(x) + O(|x|^{-1}) & \mbox{as } |x|\rightarrow + \infty,\ms\\
\ds \Delta p= 0 & \mbox{in }  \mathbb{R}^2\setminus\overline{D}, \ms\\
\ds \frac{\partial p}{\partial \nu} = 0  & \mbox{on } \partial D, \ms\\
\ds p|_{+}=p|_{-} & \mbox{on }  \partial \Omega,\ms\\
\ds \frac{\partial p}{\partial \nu} \Big|_{+} - \frac{\partial p}{\partial \nu} \Big|_{-} = 12 \zeta_0\frac{\partial \varphi}{\partial \nu} & \mbox{on } \partial \Omega,\ms \\
\ds  p = P(x) + O(|x|^{-1}) & \mbox{as } |x|\rightarrow +\infty,
\end{cases}
\end{align}
where $ \frac{\partial }{\partial \nu}$ denote the outward normal derivative and we use the notation $\frac{\partial p}{\partial \nu}\big|_{\pm}$ indicating
$$
\frac{\partial p}{\partial \nu}\bigg|_{\pm}(x):=\lim_{t\rightarrow 0^+}\langle \nabla p(x\pm t\nu(x)),\nu(x) \rangle, \ \ x\in \p \Omega,
$$
where $\nu$ is the outward unit normal vector to $\p \Omega$.

We are now in a position to introduce the definition of cloaking and shielding which play a central role in this paper.
\begin{defn}\label{def-cloaking}
The triples $\{D, \Omega; \zeta_0\}$ is said to be a perfect hydrodynamic cloaking if
\begin{equation}\label{cond-cloaking}
\bm u_{aver}  = - \nabla P / 12 \quad  \mbox{in }  \mathbb{R}^2\setminus\overline{\Omega},
\end{equation}
where $\bm u_{aver}$ is the dimensionless version of  $\tilde{\bm u}_{aver}$ defined in \eqref{dimension-equations}.
If the notation $"="$ is replaced by $"\approx"$, then it is called a near/approximate hydrodynamic cloaking.
\end{defn}
Let the depth-averaged hydrodynamic force acting on the object be given by
\begin{equation}\label{force}
  \bm F =\int_{\p D}\bm \sigma \cdot \nu \d s(y),
\end{equation}
where $\bm \sigma = -p I$ is  the normalized stress tensor.
\begin{defn}\label{def-shielding}
The triples $\{D, \Omega; \zeta_0\}$ is said to be a perfect  hydrodynamic shielding if
\begin{equation}\label{cond-shielding}
  \bm F = 0 \quad  \mbox{on }  \p D.
\end{equation}
If the notation $"="$ is replaced by $"\approx"$, then it is called a near/approximate hydrodynamic shielding.
\end{defn}

Outside the cloaking region, the pressure is related to the velocity field through $\langle \bm{u} \rangle = -\nabla p/12 $ subjected to the
boundary condition $p(x)=P(x)$ as $|x|\rightarrow\infty$, and therefore, according to the Definition \ref{def-cloaking}, the condition (\ref{cond-cloaking}) can be expressed in terms of the
pressure as
\begin{align}\label{cond-cloaking-p}
p(x)=P(x), \quad x\in \mathbb{R}^2\setminus\overline{\Omega}.
\end{align}
Similarly, from (\ref{force}) it follows that the depth-averaged hydrodynamic force $\bm F$ vanishes provided
\begin{equation}\label{cond-shileding-p}
  p(x) = 0, \quad x\in \Omega\setminus\overline{D}.
\end{equation}
 According to Definition \ref{def-shielding}, hydrodynamic shielding occurs. In this paper, we assume $D$ and $\Omega$ are known, and want to find appropriate zeta potential $\zeta_0$ to achieve the hydrodynamic shielding and cloaking effectively.

Our main results in this paper are given in the following theorems. The proofs are given in Section \ref{subsec-cloaking-annulus} \ref{subsec-cloaking-ellipses} and \ref{subsec-optimal}, respectively.
\begin{thm}\label{main-thm-annulus}
Let the domains $D$ and $\Omega$ be concentric disks of radii $r_i$ and $r_e$, where $r_e>r_i$. Let $H(x) = r^n\e^{\i n \theta}$ and $P(x) = 12 r^n\e^{\i n \theta}$ for $n\geq 1$. If
\begin{equation}\label{annulus-cloaking-zeta}
    \zeta_0=\frac{2r_i^{2n}r_e^{2n}}{r_e^{4n}-r_i^{4n}},
\end{equation}
then the perfect hydrodynamic cloaking condition  (\ref{cond-cloaking}) is satisfied. And if
\begin{equation}\label{annulus-shielding-zeta}
\zeta_0=\frac{2 r_e^{2n}}{r_e^{2n}-r_i^{2n}},
\end{equation}
then the perfect hydrodynamic shielding condition  (\ref{cond-shielding}) is satisfied.
\end{thm}
\begin{rem}
In \cite{Boyko2021}, the authors consider the special case where the background electrostatic potential and pressure field are given by $H(x) = r \cos(\theta)$ and $P(x) = 12r\cos(\theta)$. In fact, the special case is included in Theorem \ref{main-thm-annulus} when $n=1$. Compared with the linear background fields in \cite{Boyko2021}, we extend the background electrostatic potential and pressure field to a more general harmonic function in Theorem \ref{main-thm-annulus}.
\end{rem}

\begin{thm}\label{main-thm-ellipses}
Let the domains $D$ and $\Omega$ be confocal ellipses of elliptic radii $\xi_i$ and $\xi_e$, where $\xi_e>\xi_i$.
\begin{itemize}
    \item Let $H(x) = \cosh (n\xi) \cos (n\eta)$ and $P(x) = 12  \cosh (n\xi) \cos (n\eta)$ for $n\geq 1$. If
    \begin{equation}\label{ellipse-cloaking-zeta-x}
    \zeta_0=\frac{ \sinh (n\xi_i)}{\big(\sinh (n\xi_e) - \e^{n(\xi_i-\xi_e)}\sinh (n\xi_i)\big)\cosh (n(\xi_e -\xi_i))},
    \end{equation}
    then the perfect hydrodynamic cloaking condition  (\ref{cond-cloaking}) is satisfied. If
    \begin{equation}\label{ellipse-shielding-zeta-x}
    \zeta_0=\frac{\e^{n\xi_e}}{\sinh (n\xi_e) - \e^{n(\xi_i-\xi_e)}\sinh (n\xi_i)},
    \end{equation}
    then the perfect hydrodynamic shielding condition  (\ref{cond-shielding}) is satisfied.
    \item Let $H(x) = \sinh (n\xi) \sin (n\eta)$ and $P(x) = 12  \sinh (n\xi) \sin (n\eta)$ for $n\geq 1$. If
    \begin{equation}\label{ellipse-cloaking-zeta-y}
    \zeta_0=\frac{ \cosh (n\xi_i)}{\big(\cosh (n\xi_e) - \e^{n(\xi_i-\xi_e)}\cosh (n\xi_i)\big)\cosh (n(\xi_e -\xi_i))},
    \end{equation}
    then the perfect hydrodynamic cloaking condition  (\ref{cond-cloaking}) is satisfied. If
    \begin{equation}\label{ellipse-shielding-zeta-y}
    \zeta_0=\frac{ \e^{n\xi_e}}{\cosh (n\xi_e) - \e^{n(\xi_i-\xi_e)}\cosh (n\xi_i)},
    \end{equation}
    then the perfect hydrodynamic shielding condition  (\ref{cond-shielding}) is satisfied.
  \end{itemize}
\end{thm}
\begin{rem}
If we change simultaneously the sign of the background field $H(x)$ and $P(x)$, then the cloaking and shielding structures will not be changed. However, if we only change the sign of one of them, then we need to change the sign of $\zeta_0$ such that the cloaking and shielding occur.
Further, we notice that
\begin{equation*}
  \sinh (n\xi_e) - \e^{n(\xi_i-\xi_e)}\sinh (n\xi_i) =\cosh (n\xi_e) - \e^{n(\xi_i-\xi_e)}\cosh (n\xi_i),
\end{equation*}
which means the shielding structure is not changed in these cases. In Section \ref{sec-NumSim}, the observation is verified numerically when $n=1$.
\end{rem}

For the arbitrary-geometry for $D$ and $\Omega$, by using the optimization approach, we shall establish a sufficient condition for the occurrence of approximate hydrodynamic cloaking (or shielding) in the following theorem.
\begin{thm}\label{thm-estimation}
Let $p$ be the solution to (\ref{electro-osmotic equation}) with $p|_{+}=P$ on $\p \Omega$ (or $p|_{-}=0$ on $\p \Omega$), and the optimization functional $\mathcal{F}(\zeta_{0,opt})$, $\mathcal{G}(\zeta_{0,opt})$ are defined by
(\ref{cost-functional-cloaking}), (\ref{cost-functional-shielding}) respectively.
If there exists an optimal zeta potential  $\zeta_{0,opt}$ such that $\mathcal{F}(\zeta_{0,opt})<\epsilon^2$ (or $\mathcal{G}(\zeta_{0,opt})<\epsilon^2$) where $\epsilon\ll 1$, then the approximate hydrodynamic cloaking (or shielding) occurs, that is, $|p - P|<\epsilon$ in $\R^2 \setminus \overline{\Omega}$ (or $|p|<\epsilon$ in $\Omega \setminus \overline{D}$).
\end{thm}
\begin{rem}
It's worth noting that, for general geometry, the perfect hydrodynamic cloaking and shielding are usually difficult to achieve. It is more closely related to the corresponding overdetermined boundary value problem (see Section \ref{subsec-optimal} for details).
\end{rem}

\section{Layer potentials formulation}\label{sec-layer-potentials}
In this section, we first collect some preliminary knowledge on boundary layer potentials and then establish the representation formula of solution of the governing equations.
Let $\Gamma_i := \p D$ and $\Gamma_e := \p \Omega$. For $\Gamma = \Gamma_i$  or $\Gamma_e$, let us now introduce the single-layer potential by
\begin{align*}
\mathcal{S}_\Gamma[\vartheta](x) :=\int_{\Gamma}G(x,y)\vartheta(y)\d s(y), \quad  x\in \mathbb{R}^2,
\end{align*}
where $\vartheta\in L^2(\Gamma)$ is the density function, and the Green function $G(x,y)$ to the Laplace in  $\mathbb{R}^2$ is given by
\begin{align*}
G(x,y)=\frac{1}{2\pi}\ln|x-y|.
\end{align*}
Then the following jump relation holds :
\begin{align}
\label{jump-relation}
\frac{\partial\mathcal{S}_\Gamma[\vartheta]}{\partial \nu}\bigg|_{\pm}(x)&=\Big(\pm\frac{1}{2}I+\mathcal{K}^*_\Gamma\Big)[\vartheta](x), \ \ x\in \Gamma,
\end{align}
where $\mathcal{K}^*_{\Gamma}$ denote Neumann-Poincar$\acute{e}$ (NP) operator defined by
\begin{align*}
\mathcal{K}_\Gamma^*[\vartheta](x)=\int_{\Gamma}\frac{\partial G(x,y)}{\partial \nu (x)}\vartheta(y)\d s(y).
\end{align*}

To establish the representation formula of the solution, we make use of the following lemma.
\begin{lem}\label{I+K-invertible}\cite{Ammari2007}
The operator $\frac{1}{2} I+\mathcal{K}^*_{\Gamma_i}: L_0^2(\Gamma_i)\rightarrow L_0^2(\Gamma_i)$ is invertible. Here $L_0^2 := \{f\in L^2(\Gamma_i); \int_{\Gamma_i} f \d s=0\}$.
\end{lem}

By the layer potential theory, we can obtain the following theorem using Lemma \ref{I+K-invertible}.
\begin{thm}\label{well-posedness}
Let $\varphi, p\in C^2(\mathbb{R}^2\setminus\overline{D})\bigcap C(\mathbb{R}^2\setminus D)$ be the classical solution to (\ref{electro-osmotic equation}).
Then $\varphi$ can be represented as
\begin{equation}\label{sol-ep}
\varphi =
H(x) + \mathcal{S}_{\Gamma_i}[\phi](x),\quad x\in\mathbb{R}^2\setminus\overline{D},
\end{equation}
where  density function $ \phi \in L_0^2(\Gamma_i )$ satisfies
\begin{align}\label{e-potential-density}
\Big(\frac{1}{2} I+\mathcal{K}^*_{\Gamma_i}\Big)[\phi]=-\frac{\partial H}{\partial \nu}\Big|_{\Gamma_i}.
\end{align}
And $p$ can be represented using the single-layer potentials $S_{\Gamma_i}$ and $S_{\Gamma_e}$ as follows:
 \begin{equation}\label{sol-p}
p = P(x) +\mathcal{S}_{\Gamma_i}[\psi_i](x) + \mathcal{S}_{\Gamma_e}[\psi_e](x),\quad  x\in  \mathbb{R}^2\setminus\overline{D},
\end{equation}
where the pair $(\psi_i, \psi_e)\in L_0^2(\Gamma_i)\times L_0^2(\Gamma_e)$ satisfy
\begin{align}
\label{density-equation}
\begin{cases}
\ds \Big(\frac{1}{2} I+\mathcal{K}^*_{\Gamma_i}\Big)[\psi_i]
+\frac{\partial\mathcal{S}_{\Gamma_e}[\psi_e]}{\partial\nu_i}= -\frac{\p P}{\p \nu_i} &\quad  \mbox{on }\Gamma_i, \ms\\
\ds \psi_e =12 \zeta_0\frac{\partial \varphi}{\partial \nu_e} &\quad  \mbox{on }\Gamma_e. \
\end{cases}
\end{align}
Furthermore, there exists a constant $C=C(\zeta_0,D,\Omega)$ such that
\begin{align}
\label{stability}
\|\psi_i\|_{L^2(\Gamma_i)}+\|\psi_e\|_{L^2(\Gamma_e )}
\leq C \big(\| \nabla P \|_{L^2(\Gamma_i)}+\left\|\nabla \varphi \right\|_{L^2(\Gamma_e)}\big).
\end{align}
\end{thm}
\begin{proof}
Using the jump formula \eqref{jump-relation} for the normal derivative of the single layer potential, the boundary condition on the boundary $\Gamma_i$ satisfied by \eqref{sol-ep} reads \eqref{e-potential-density}.
Thus by Lemma \ref{I+K-invertible}, the density function $\phi$ exists uniquely. Similarly, the boundary and transmission conditions along the boundary $\Gamma_i$ and interface $\Gamma_e$ satisfied by \eqref{sol-p} read \eqref{density-equation}. The density function $\psi_e$ was already expressed by the normal derivative of the electrostatic potential. Hence the density function $\psi_e$ exists uniquely. By Lemma \ref{I+K-invertible}, the density function $\psi_i$ also exists uniquely. The estimate \eqref{stability} is a consequence of the solvability and the closed graph theorem.
The proof is complete.
\end{proof}

The following corollary is a direct consequence of Theorem \ref{well-posedness}, and thus its proof is skipped.
\begin{cor}\label{cor-stability}
Let $\varphi$ and $p$ be the classical solution to (\ref{electro-osmotic equation}). Then there exists a positive constant $C=C(\zeta_0,D,\Omega)$ such that
\begin{align*}
  \|\varphi\|_{H_{loc}^1(\R^2\setminus\overline{D})}\leq C\|H\|_{H^1(\R^2\setminus\overline{D})},
 \end{align*}
 and
\begin{align*}
  \|p\|_{H_{loc}^1(\R^2\setminus\overline{D})}\leq C\big(\|P\|_{H^1(\R^2\setminus\overline{D})} + \|H\|_{H^1(\R^2\setminus\overline{D})}\big).
 \end{align*}
\end{cor}

According to Theorem \ref{well-posedness} and Corollary \ref{cor-stability}, we obtain the representation formula and quantitative estimation of solution of the electro-osmosis model. It will be used to deduce the hydrodynamic cloaking and shielding conditions in the following sections.

\section{Hydrodynamic cloaking and shielding }\label{sec-cloaking-shilding}
This section is devoted to the proofs of Theorems \ref{main-thm-annulus} and \ref{main-thm-ellipses}, which determine the conditions for hydrodynamic cloaking and shielding.
We first consider the microscale hydrodynamic cloaking via electro-osmosis when the domains $D$ and $\Omega$ are concentric disks and confocal ellipses in Subsections \ref{subsec-cloaking-annulus} and \ref{subsec-cloaking-ellipses}, respectively. We calculate the explicit form of the solution to \eqref{electro-osmotic equation} for this two special cases. In Subsection \ref{subsec-optimal}, we then consider hydrodynamic cloaking and shielding for general shapes via an optimization method.

\subsection{ Perfect hydrodynamic cloaking and shielding on the annulus}\label{subsec-cloaking-annulus}
Throughout this subsection, we set
$D :=\{|x| < r_i\}$ and $\Omega :=\{|x| < r_e\}$, where $r_e > r_i$.
For each integer $n$ and $a=i, e$, one can easily see that (cf. \cite{Ammari2013})
\begin{align}\label{S-ra}
\mathcal{S}_{\Gamma}[e^{\i n\theta}](x) = \begin{cases}
\ds -\frac{r_a}{2n}\Big( \frac{r}{r_a}\Big)^n e^{\i n\theta},\quad & |x|=r  < r_a,\ms\\
\ds -\frac{r_a}{2n}\Big( \frac{r_a}{r}\Big)^n e^{\i n\theta},\quad & |x|=r > r_a,
\end{cases}
\end{align}
and
\begin{equation}\label{K-ra}
\mathcal{K}_{\Gamma}^*[e^{\i n\theta}](x)=0, \quad \forall \  n\neq 0.
\end{equation}

We begin with the proof of Theorem \ref{main-thm-annulus}.
\renewcommand{\proofname}{\indent Proof of  Theorem \ref{main-thm-annulus}}
\begin{proof}
Let $H(x) = r^n\e^{\i n\theta}$ for $n\geq 1$. From Theorem \ref{well-posedness}, we have
\begin{equation*}
\varphi=
H(x) + \mathcal{S}_{\Gamma_i}[\phi](x),\quad x\in\mathbb{R}^2\setminus\overline{D},
\end{equation*}
where
\begin{align*}
\Big(\frac{1}{2} I+\mathcal{K}^*_{\Gamma_i}\Big)[\phi]=-\frac{\partial H}{\partial r}\Big |_{r=r_i}.
\end{align*}
By straightforward calculations and using  (\ref{K-ra}), we can obtain
\begin{align}\label{density-annulus-ep}
\phi = -2n r_i^{n-1}e^{\i n\theta}.
\end{align}
Substituting (\ref{density-annulus-ep}) into (\ref{sol-ep}) and using (\ref{S-ra}), we have the solution to \eqref{electro-osmotic equation}
\begin{align}\label{disk-varphi}
\varphi = \Big(r^n+\frac{r_i^{2n}}{r^n}\Big)e^{\i n\theta}.
\end{align}

Let $P(x)=12 r^n\e^{\i n\theta}$ for $n\geq 1$. By (\ref{density-equation}),  (\ref{S-ra}), (\ref{K-ra}), and (\ref{disk-varphi}), if
\begin{align*}
\left[
  \begin{array}{ccc}
    \psi_{i} \ms \\
    \psi_{e}
  \end{array}
\right]
=
\left[
  \begin{array}{cc}
    \psi_{i}^n \ms \\
    \psi_{e}^n
  \end{array}
\right]
\e^{\i n\theta},
\end{align*}
then we have
\begin{align}
\label{density-annulus}
\begin{cases}
\ds \psi_i = 12n\frac{r_i^{n-1}}{r_e^{2n}}\Big((r_e^{2n}-r_i^{2n})\zeta_0-2r_e^{2n}\Big)\e^{\i n\theta}, \ms \\
\ds \psi_e = 12 n \zeta_0 \Big(r_e^{n-1}-\frac{r_i^{2n}}{r_e^{n+1}}\Big) \e^{\i n\theta}.
\end{cases}
\end{align}
 Substituting (\ref{density-annulus}) into (\ref{sol-p}) and using (\ref{S-ra}), we can find that the solution to \eqref{electro-osmotic equation} is
 \begin{align}\label{annulus-p}
 p=
  \begin{cases}
\ds -\frac{6}{r_e^{2n}}\Big((r_e^{2n}-r_i^{2n})\zeta_0-2r_e^{2n}\Big)\Big(r^n+\frac{r_i^{2 n}}{r^n}\Big)\e^{\i n\theta},\quad r_i<r<r_e,
\vspace{1em}\\
\ds 12 r^n\e^{\i n\theta}-\frac{6}{r_e^{2n}}\Big((r_e^{4 n} - r_i^{4 n})\zeta_0 - 2r_i^{2n} r_e^{2n}\Big)\frac{1}{r^n}\e^{\i n\theta},\quad r>r_e.
\end{cases}
 \end{align}

From the second equation in (\ref{annulus-p}), it follows that the outer flow and pressure satisfy the cloaking conditions (\ref{cond-cloaking-p}) and (\ref{cond-cloaking}) provided
\begin{align*}
\zeta_0=\frac{2r_i^{2n}r_e^{2n}}{r_e^{4n}-r_i^{4n}}.
\end{align*}
From the first equation in (\ref{annulus-p}), it follows that the inner pressure vanishes for $r_i<r<r_e$  provided
\begin{align*}
\zeta_0=\frac{2 r_e^{2n}}{r_e^{2n}-r_i^{2n}}.
\end{align*}
Hence the depth-averaged hydrodynamic force $\bm F$ vanishes.
The proof is complete.
\end{proof}

\subsection{Perfect hydrodynamic cloaking and shielding on the confocal ellipses}\label{subsec-cloaking-ellipses}
To consider the perfect hydrodynamic cloaking and shielding on the confocal ellipses we introduce the elliptic coordinates $(\xi, \eta)$ so that $x=(x_1,x_2)$ in Cartesian coordinates are defined by
\begin{align*}
  x_1=l \cosh \xi \cdot \cos \eta, \quad x_2=l \sinh \xi \cdot \sin \eta,\quad \xi \geq 0, \quad 0\leq \eta \leq 2\pi,
\end{align*}
where $2l$ is the focal distance.
Suppose that $\p D=\Gamma_i$ and $\p \Omega=\Gamma_e$ are given by
$$
\Gamma_i = \{ (\xi, \eta) : \xi = \xi_i\} \quad \mbox{and} \quad \Gamma_e = \{ (\xi, \eta) : \xi = \xi_e\},
$$
where the number $\xi_i$ and $\xi_e$ are called the elliptic radius $\Gamma_i$ and $\Gamma_e$, respectively.

Let $\Gamma = \{ (\xi, \eta) : \xi = \xi_a\}$ for $a=i, e$. One can see easily that the length element $\d s$ and the outward normal derivative $\frac{\p}{\p \nu}$ on $\Gamma$ are given in terms of the elliptic coordinates by
\begin{equation*}
  \d s = \gamma \d \eta \quad \mbox{and} \quad \frac{\p}{\p \nu} = \gamma^{-1}\frac{\p}{\p \xi},
\end{equation*}
where
\begin{equation*}
\gamma =  \gamma (\xi_a, \eta) = l \sqrt{\sinh^2\xi_a+\sin^2\eta}.
\end{equation*}

To proceed, it is convenient to use the following notations: for $a=i, e$ and $n=1,2,\dots$,
\begin{align*}
  \beta_n^{c,a} := \gamma (\xi_a, \eta)^{-1} \cos (n\eta) \quad \mbox{and} \quad \beta_n^{s,a} := \gamma (\xi_a, \eta)^{-1} \sin (n\eta).
\end{align*}
For a nonnegative integer $n$ and $a=i, e$, it is proved in \cite{Chung2014, Ando2016} that
\begin{align}\label{S-ellipse-cos}
\mathcal{S}_{\Gamma}[ \beta_n^{c,a}](x) = \begin{cases}
\ds -\frac{\cosh (n\xi)}{n\e^{n\xi_a}}\cos (n\eta),\quad & \xi  < \xi_a,\ms\\
\ds -\frac{\cosh (n\xi_a)}{n\e^{n\xi}}\cos (n\eta),\quad & \xi  > \xi_a,
\end{cases}
\end{align}
and
\begin{align*}
\mathcal{S}_{\Gamma}[\beta_n^{s,a}](x) = \begin{cases}
\ds -\frac{\sinh (n\xi)}{n\e^{n\xi_a}}\sin (n\eta),\quad & \xi  < \xi_a,\ms\\
\ds -\frac{\sinh (n\xi_a)}{n\e^{n\xi}}\sin (n\eta),\quad & \xi  > \xi_a.
\end{cases}
\end{align*}
Moreover, we also have
\begin{align}\label{K-ellipse}
\mathcal{K}^*_\Gamma [\beta_n^{c,a}] = \frac{1}{2 \e^{2n\xi_a}}\beta_n^{c,a}\quad \mbox{and} \quad \mathcal{K}^*_\Gamma [\beta_n^{s,a}] = -\frac{1}{2 \e^{2n\xi_a}}\beta_n^{s,a}.
\end{align}

We are ready to present the proof of Theorem \ref{main-thm-ellipses}.
\renewcommand{\proofname}{\indent Proof of  Theorem \ref{main-thm-ellipses}}
\begin{proof}
Let $H(x)=\cosh (n\xi) \cos (n\eta)$ for $n\geq 1$. From the Theorem \ref{well-posedness} in Section \ref{sec-layer-potentials} and (\ref{K-ellipse}), if $\phi = \phi_n \beta_n^{c,i}$, we have
\begin{equation}\label{ep-density-ellipse}
  \Big(\frac{1}{2} + \frac{1}{2 \e^{2n\xi_i}}\Big) \phi = - n \sinh (n\xi_i) \;\beta_n^{c,i}.
\end{equation}
It is readily seen that the solution to (\ref{ep-density-ellipse}) is given by
\begin{equation}\label{density-ellipse-ep}
  \phi = - n \e^{n\xi_i} \tanh (n\xi_i) \;\beta_n^{c,i}.
\end{equation}
 Substituting (\ref{density-ellipse-ep}) into (\ref{sol-ep}) and using (\ref{S-ellipse-cos}), we have the solution to \eqref{electro-osmotic equation}
\begin{equation}\label{ellipse-varphi}
  \varphi =\big(\cosh (n\xi) + \e^{n\xi_i} \sinh (n\xi_i) \;\e^{-n\xi}\big)\cos (n\eta).
\end{equation}

Let $P(x)=12 \cosh (n\xi) \cos (n\eta)$ for $n\geq 1$. From (\ref{density-equation}) and (\ref{ellipse-varphi}) we first have
\begin{equation}\label{ellipse-phi-e}
     \psi_{e} = 12 n \zeta_0 \big(\sinh (n\xi_e) - \e^{n(\xi_i-\xi_e)} \sinh (n\xi_i) \big)\beta_n^{c,e}.
\end{equation}
If $ \psi_{i} =  \psi_{i}^n \beta_n^{c,i}$, then the first equation in (\ref{density-equation}) is equivalent to
\begin{equation*}
  \Big(\frac{1}{2} + \frac{1}{2 \e^{2n\xi_i}}\Big) \psi_i = 12 n\Big( \big(\sinh (n\xi_e) - \e^{n(\xi_i-\xi_e)}\sinh (n\xi_i)\big)\frac{\sinh (n\xi_i)}{\e^{n\xi_e}}\zeta_0- \sinh (n\xi_i) \Big)\beta_n^{c,i}.
\end{equation*}
By the simple calculation, we can obtain
\begin{equation}\label{ellipse-phi-i}
\psi_i = 12 n\Big( \big(\sinh (n\xi_e) - \e^{n(\xi_i-\xi_e)}\sinh (n\xi_i)\big) \e^{n(\xi_i-\xi_e)} \tanh (n\xi_i) \;\zeta_0- \e^{n\xi_i} \tanh (n\xi_i) \Big)\beta_n^{c,i}.
\end{equation}
Substituting (\ref{ellipse-phi-e}) and (\ref{ellipse-phi-i}) into (\ref{sol-ep}) and using (\ref{S-ellipse-cos}), we can obtain the solution to \eqref{electro-osmotic equation}
{\small
 \begin{align}\label{ellipse-p-dir-x}
 p=
  \begin{cases}
\ds -\frac{12}{\e^{n\xi_e}}\Big(\big(\sinh (n\xi_e) - \e^{n(\xi_i-\xi_e)}\sinh (n\xi_i)\big)\zeta_0 -\e^{n\xi_e}\Big)\Big(\cosh (n\xi) + \e^{n\xi_i} \sinh (n\xi_i) \, \e^{-n\xi}\Big)\cos (n\eta),\ &\xi_i<\xi<\xi_e,
\vspace{1em}\\
\ds \Big( 12 \cosh (n\xi) - 12\Big(\zeta_0\big(\sinh (n\xi_e) - \e^{n(\xi_i-\xi_e)}\sinh (n\xi_i) \big)\cosh (n(\xi_e -\xi_i) - \sinh (n\xi_i)) \Big)\frac{\e^{n\xi_i}}{\e^{n\xi}}\Big)\cos (n\eta),\ &\xi>\xi_e.
\end{cases}
 \end{align}}
 If $H(x) = \sinh (n\xi) \sin (n\eta)$ and $P = 12 \sinh (n\xi) \sin (n\eta)$ for $n\geq 1$, in a similar way we can find the solution to \eqref{electro-osmotic equation}
 \begin{equation*}
  \varphi =\big(\sinh (n\xi) + \e^{n\xi_i} \cosh (n\xi_i) \;\e^{-n\xi}\big)\sin (n\eta),
\end{equation*}
 and the solution to \eqref{electro-osmotic equation}
 {\small
 \begin{align}\label{ellipse-p-dir-y}
 p=
  \begin{cases}
\ds -\frac{12}{\e^{n\xi_e}}\Big(\big(\cosh (n\xi_e) - \e^{n(\xi_i-\xi_e)}\cosh (n\xi_i)\big)\zeta_0 -\e^{n\xi_e}\Big)\Big(\sinh (n\xi) + \e^{n\xi_i} \cosh (n\xi_i) \, \e^{-n\xi}\Big)\sin (n\eta),\quad & \xi_i<\xi<\xi_e,
\vspace{1em}\\
\ds \Big( 12 \sinh (n\xi) - 12\Big(\zeta_0\big(\cosh (n\xi_e) - \e^{n(\xi_i-\xi_e)}\cosh (n\xi_i)\big)\cosh(\xi_e -\xi_i) - \cosh (n\xi_i) \Big)\frac{\e^{n\xi_i}}{\e^{n\xi}}\Big)\sin (n\eta),\quad &\xi>\xi_e.
\end{cases}
 \end{align}}
The cloaking and shielding conditions  immediately follow from the equations (\ref{ellipse-p-dir-x}) and (\ref{ellipse-p-dir-y}).

The proof is complete.
\end{proof}

\subsection{Approximate hydrodynamic cloaking using optimization method}\label{subsec-optimal}
In this subsection, we discuss a general framework for hydrodynamic cloaking and shielding. According to Definition \ref{def-cloaking} and equation (\ref{cond-cloaking-p}) and in order to ensure the cloaking occurs, we only need to solve the following overdetermined boundary value problem with $\{D,\Omega;\zeta_0\}$:
\begin{align}\label{press-cloaking}
\begin{cases}
\ds \Delta p= 0 & \mbox{in }  \Omega\setminus\overline{D}, \ms\\
\ds \frac{\partial p}{\partial \nu} = 0  & \mbox{on } \partial D, \ms\\
\ds p=P & \mbox{on }  \partial \Omega,\ms\\
\ds \frac{\partial P}{\partial \nu} - \frac{\partial p}{\partial \nu}  = 12 \zeta_0\frac{\partial \varphi}{\partial \nu} & \mbox{on } \partial \Omega.
\end{cases}
\end{align}
and $\varphi$ satisfies:
\begin{align}\label{electro equation}
\begin{cases}
\ds \Delta \varphi = 0  & \mbox{in } \mathbb{R}^2\setminus\overline{D}, \ms \\
\ds \frac{\partial \varphi}{\partial \nu} = 0  & \mbox{on } \partial D, \ms \\
\ds  \varphi = H(x) + O(|x|^{-1}) & \mbox{as } |x|\rightarrow + \infty.
\end{cases}
\end{align}

In fact, from the uniqueness of exterior Dirichlet problem for the Laplace equation with $p|_{\partial \Omega}=P$, we obtain $p=P$, in $\mathbb{R}^2\setminus\overline{\Omega}$, i.e., the perfect hydrodynamic cloaking occurs.

Similarly, according to Definition \ref{def-shielding} and equation (\ref{cond-shileding-p}), and in order to ensure the shielding occurs, the following overdetermined exterior boundary value problem is needed to solve:
\begin{align}\label{press-shielding}
\begin{cases}
\ds \Delta p= 0 & \mbox{in }  \mathbb{R}^2\setminus\overline{\Omega}, \ms\\
\ds p=0 & \mbox{on }  \partial \Omega,\ms\\
\ds \frac{\partial p}{\partial \nu}  = 12 \zeta_0\frac{\partial \varphi}{\partial \nu} & \mbox{on } \partial \Omega,\ms \\
\ds  p = P(x) + O(|x|^{-1}) & \mbox{as } |x|\rightarrow +\infty.
\end{cases}
\end{align}

In Sections \ref{subsec-cloaking-annulus} and \ref{subsec-cloaking-ellipses}, we know that perfect hydrodynamic cloaking and shielding occur in two special cases by using explicit expressions. One can easily verify that these cloaking and shielding structures $\{D,\Omega;\zeta_0\}$ satisfy equations (\ref{press-cloaking}) and (\ref{press-shielding}), respectively. However, for the arbitrary shape, we can not demonstrate the cloaking and shielding structures $\{D,\Omega;\zeta_0\}$ exists since the analytical solutions do not exist for the general shape. In particular, the overdetermined problem (\ref{press-cloaking}) and (\ref{press-shielding}) with general domain have no solution.  Fortunately, the numerical experiments show that the approximate hydrodynamic cloaking and shielding exist for the general shape.

To find the approximate hydrodynamic cloaking structure $\{D,\Omega;\zeta_0\}$ for general shape, we first solve the electrostatic potential equation in \eqref{electro-osmotic equation} and then the following interior mixed boundary value problem
\begin{align}\label{p-cloaking-Dirichlet}
\begin{cases}
\ds \Delta p= 0 & \mbox{in }  \Omega\setminus\overline{D}, \ms\\
\ds \frac{\partial p}{\partial \nu} = 0  & \mbox{on } \partial D, \ms\\
\ds p=P & \mbox{on }  \partial \Omega.
\end{cases}
\end{align}
Hence we can obtain $\zeta_0$ that satisfies the cloaking condition (\ref{cond-cloaking-p}) by solving the following equation
\begin{equation}\label{p-cloaking-zeta}
  \ds \frac{\partial P}{\partial \nu} - \frac{\partial p}{\partial \nu} = 12 \zeta_0\frac{\partial \varphi}{\partial \nu} \quad \mbox{on } \partial \Omega.
\end{equation}
In order to solve equation (\ref{p-cloaking-Dirichlet}), one can use the standard Nystr\"{o}m method or finite element method. It is worth noting that the solution to (\ref{p-cloaking-zeta}) does not exist for the general shape. Hence we need to choose optimal $\zeta_0$ by optimization method.
Define a cost functional to be
\begin{equation}\label{cost-functional-cloaking}
  \mathcal{F}(\zeta_0)=\Big\|\frac{\partial P}{\partial \nu} - \frac{\partial p}{\partial \nu}
-12 \zeta_0 \frac{\partial \varphi}{\partial \nu}\Big\|^2_{L^2(\partial \Omega)}, \quad \zeta_0\in[a,b]\subset\R.
\end{equation}
Then the optimal zeta potential is defined by
\begin{equation}\label{zeta-opt-cloaking}
  \zeta_{0,opt} :=\mathop{\arg\min}_{\zeta_0\in [a,b]} \,\mathcal{F}(\zeta_0).
\end{equation}
That is, to design the approximate hydrodynamic cloaking, we can solve the PDE-constrained optimization problem (\ref{p-cloaking-Dirichlet}), (\ref{electro equation}) and (\ref{zeta-opt-cloaking}).

Similarly, for shielding, we first solve electrostatic potential equation in \eqref{electro-osmotic equation} and the following exterior Dirichlet boundary value problem
\begin{align}
\label{press-shielding-Dirichlet}
\begin{cases}
\ds \Delta p= 0 & \mbox{in }  \mathbb{R}^2\setminus\overline{\Omega}, \ms\\
\ds p=0 & \mbox{on }  \partial \Omega,\ms\\
\ds  p = P(x) + O(|x|^{-1}) & \mbox{as } |x|\rightarrow +\infty.
\end{cases}
\end{align}
Let the cost functional be
\begin{equation}\label{cost-functional-shielding}
  \mathcal{G}(\zeta_0)=\Big\|\frac{\partial p}{\partial \nu}
-12 \zeta_0 \frac{\partial \varphi}{\partial \nu}\Big\|^2_{L^2(\partial \Omega)}, \quad \zeta_0\in[c,d],
\end{equation}
then the optimal zeta potential is defined by
\begin{equation}\label{zeta-opt-shielding}
  \zeta_{0,opt} :=\mathop{\arg\min}_{\zeta_0\in [c,d]} \,\mathcal{G}(\zeta_0).
\end{equation}

The existence, uniqueness and stability of a minimizer for the constrained optimization problem $\eqref{cost-functional-cloaking}$ and $\eqref{cost-functional-shielding}$ are given in the following theorems. We only give the proof of the optimal problem $\eqref{cost-functional-cloaking}$. The proof of $\eqref{cost-functional-shielding}$ is similar
to that of the optimization problem $\eqref{cost-functional-cloaking}$  and therefore is skipped.

\begin{thm}
There exists a unique optimal zeta potential  $\zeta_{0,opt} \in [a,b]$, which minimizes the cost functional $\mathcal{F}(\zeta_0)$ with the PDE-constraint (\ref{p-cloaking-Dirichlet}) and (\ref{electro equation}) (or $\mathcal{G}(\zeta_0)$  with the PDE-constraint (\ref{electro equation}) and (\ref{press-shielding-Dirichlet})) over all $\zeta_0\in [a,b]$.
\end{thm}
\renewcommand{\proofname}{Proof}
\begin{proof}
For any $\zeta_0^{(1)}$, $\zeta_0^{(2)}\in [a,b]$, by the $L^2$-estimation of solution for (\ref{p-cloaking-Dirichlet}),  we have
\begin{align*}
&|\mathcal{F}\big(\zeta_0^{(1)}\big) - \mathcal{F}\big(\zeta_0^{(1)}\big)| \\
&\leq 24 \big|\zeta_0^{(1)}-  \zeta_0^{(2)}\big|\cdot \Big\|\frac{\p (P-p)}{\p \nu}\Big\|_{L^2(\partial \Omega)} \cdot \Big\|\frac{\p \varphi}{\p\nu}\Big\|_{L^2(\partial \Omega)} + 12^2 \big|\zeta_0^{(1)} -  \zeta_0^{(2)}\big|\cdot \big|\zeta_0^{(1)} +  \zeta_0^{(2)}\big|\cdot\Big\| \frac{\p \varphi}{\p \nu}\Big\|^2_{L^2(\partial \Omega)}\\
&\leq C \big|\zeta_0^{(1)} -  \zeta_0^{(2)}\big|.
\end{align*}
Hence $\mathcal{F}(\zeta_0)$ is Lipschitz continuous in $[a,b]$. Furthermore, supposing $\lambda_1$, $\lambda_2 \in (0, 1)$ and $\lambda_1+\lambda_2=1$, we obtain
\begin{align*}
&\mathcal{F}\Big(\lambda_1\zeta_0^{(1)}+\lambda_2 \zeta_0^{(2)}\Big)\\
&= \Big\|\frac{\partial P}{\partial \nu} - \frac{\partial p}{\partial \nu}
-12 \big(\lambda_1\zeta_0^{(1)}+\lambda_2 \zeta_0^{(2)}\big) \frac{\partial \varphi}{\partial \nu}\Big\|^2_{L^2(\partial \Omega)}\\
&< \lambda_1\Big\|\frac{\partial P}{\partial \nu} - \frac{\partial p}{\partial \nu}
-12 \zeta_0^{(1)} \frac{\partial \varphi}{\partial \nu}\Big\|^2_{L^2(\partial \Omega)} + \lambda_2\Big\|\frac{\partial P}{\partial \nu} - \frac{\partial p}{\partial \nu}
-12 \zeta_0^{(2)} \frac{\partial \varphi}{\partial \nu}\Big\|^2_{L^2(\partial \Omega)}\\
&= \lambda_1\mathcal{F}\big(\zeta_0^{(1)}\big)+ \lambda_2 \mathcal{F}\big(\zeta_0^{(2)}\big).
\end{align*}
Then $\mathcal{F}(\zeta_0)$ is strictly convex in $[a,b]$.  Therefore the cost functional \eqref{cost-functional-cloaking} has a unique minimizer.

The proof is complete.
\end{proof}

\begin{thm}
Let $\{p_k\}$ and $\{\varphi_k\}$ be sequences such that $\frac{\partial p_k}{\partial \nu} \rightarrow p$ and $\frac{\partial \varphi_k}{\partial \nu} \rightarrow \varphi$, as $k \rightarrow \infty$ in $L^2(\partial \Omega)$ and let $\{\zeta_{0,opt}^{(k)}\}$ be a minimizer of \eqref{cost-functional-cloaking} (or \eqref{cost-functional-shielding}) with $\varphi$ and $p$ replaced by
$\varphi_k$ and $p_k$. Then the optimal zeta potential is stable with respect to the electrostatic potential $\varphi$ and pressure $p$, i.e., ${\zeta_{0,opt}^{(k)}}\rightarrow\zeta_{0,opt}$.
\end{thm}
\begin{proof}
  From the definition of ${\zeta_{0,opt}^{(k)}}$, we find
  \begin{align}\label{ineq}
  \Big\|\frac{\partial P}{\partial \nu} - \frac{\partial p_k}{\partial \nu}-12\zeta_{0,opt}^{(k)} \frac{\partial \varphi_k}{\partial \nu}\Big\|^2_{L^2(\partial \Omega)}
  \leq\Big\|\frac{\partial P}{\partial \nu} - \frac{\partial p_k}{\partial \nu} -12 \zeta_0 \frac{\partial \varphi_k}{\partial \nu}\Big\|^2_{L^2(\partial \Omega)}, \quad \forall \;  \zeta_0\in [a,b].
  \end{align}
 Since $\zeta_{0,opt}^{(k)}\in [a,b]$, there exists a subsequence, still denoted $\{\zeta_{0,opt}^{(k)}\}$, which implies $\zeta_{0,opt}^{(k)}\rightarrow\zeta_{0,opt}^{(0)}$  and $\zeta_{0,opt}\in [a,b]$.  Based on the continuity of $L^2$-norm, by \eqref{ineq} it follows that
  \begin{align*}
  \left\| \frac{\partial P}{\partial \nu} - \frac{\partial p}{\partial \nu}-12\zeta_{0,opt}^{(0)} \frac{\partial \varphi}{\partial \nu}\right\|^2_{L^2(\partial \Omega)}
  &= \lim_{k\rightarrow\infty}\left\|\frac{\partial P}{\partial \nu} - \frac{\partial p_k}{\partial \nu}-12\zeta_{0,opt}^{(k)} \frac{\partial \varphi_k}{\partial \nu}\right\|^2_{L^2(\partial \Omega)}\\
  &\leq \lim_{k\rightarrow\infty}\left\|\frac{\partial P}{\partial \nu} - \frac{\partial p_k}{\partial \nu} -12 \zeta_0 \frac{\partial \varphi_k}{\partial \nu}\right\|^2_{L^2(\partial \Omega)}\\
  &=  \left\| \frac{\partial P}{\partial \nu} - \frac{\partial p}{\partial \nu} -12 \zeta_0 \frac{\partial \varphi}{\partial \nu}\right\|^2_{L^2(\partial \Omega)},
  \end{align*}
  for  all $\zeta_0\in [a,b]$. This deduces that $\zeta_{0,opt}^{(0)}$ is a minimizer of $\mathcal{F}(\zeta_0)$. By the uniqueness of the minimizer, we have $\zeta_{0,opt}^{(0)}=\zeta_{0,opt}$.
  The proof is complete.
\end{proof}
\begin{rem}
In fact, from the continuous dependence on the background fields $H$ and $P$ of the solutions of (\ref{p-cloaking-Dirichlet}) and (\ref{electro equation}), we can also deduce the stability of the minimizer with respect to the background fields.
\end{rem}

Next, we give the proof of Theorem \ref{thm-estimation}.
\renewcommand{\proofname}{\indent Proof of  Theorem \ref{thm-estimation}}
\begin{proof}
We first prove the case of cloaking.
From the boundary condition $\frac{\partial p}{\partial \nu} \Big|_{+} - \frac{\partial p}{\partial \nu} \Big|_{-} = 12 \zeta_0\frac{\partial \varphi}{\partial \nu}$  on $\partial \Omega$ defined in
equation \eqref{electro-osmotic equation}, we have
\begin{align}\label{eq-boundary}
\frac{\partial P}{\partial \nu}-\frac{\partial p}{\partial \nu} \Big|_{+} = \frac{\partial P}{\partial \nu}- \frac{\partial p}{\partial \nu} \Big|_{-} - 12 \zeta_0\frac{\partial \varphi}{\partial \nu} \quad \mbox{on } \p \Omega.
\end{align}
Applying the $L^2(\partial \Omega)$ norm to both sides of the equation \eqref{eq-boundary} gives
\begin{align*}
 \Big\|\frac{\partial P}{\partial \nu}- \frac{\partial p}{\partial \nu} \Big|_{+} \Big\|_{L^2(\partial \Omega)} =\Big\| \frac{\partial P}{\partial \nu} - \frac{\partial p}{\partial \nu}\Big|_{-} -12 \zeta_0 \frac{\partial \varphi}{\partial \nu}\Big\|_{L^2(\partial \Omega)}.
\end{align*}
Setting $\zeta_0=\zeta_{0,opt}$ and using the condition $\mathcal{F}(\zeta_{0,opt})<\epsilon^2$, we obtain $\Big\|\frac{\partial P}{\partial \nu}- \frac{\partial p}{\partial \nu} \Big|_{+} \Big\|_{L^2(\partial \Omega)} < \epsilon$.
Applying Green's formula \cite{Kress1989} to $P$ and $p$ in $\R^2 \setminus \overline{\Omega}$, we have
\begin{alignat}{2}\label{Gf-exterior1}
\ds P(x)&=P_{\infty} + \int_{\p \Omega} G(x,y) \frac{\p P}{\p \nu}(y)\d s(y) - \int_{\p \Omega} \frac{\p G(x,y)}{\p \nu(y)}P(y)\d s(y), \quad &&x\in \R^2 \setminus \overline{\Omega},\\
\ds p(x)&=p_{\infty} +  \int_{\p \Omega} G(x,y) \frac{\p p}{\p \nu}(y)\d s(y) - \int_{\p \Omega} \frac{\p G(x,y)}{\p \nu(y)}p(y)\d s(y), \quad &&x\in \R^2\setminus \overline{\Omega}.\label{Gf-exterior2}
\end{alignat}
where the mean value properties at infinity
\begin{align*}
P_{\infty} = \frac{1}{2 \pi r} \int_{|y|=r} P(y) \d s(y) \quad \mbox{and} \quad p_{\infty} = \frac{1}{2 \pi r} \int_{|y|=r} p(y) \d s(y)
\end{align*}
for sufficiently large $r$ are satisfied.

In the case of cloaking, we require  $p|_{+}=P$ on $\p \Omega$. Moreover, from \eqref{Gf-exterior1}, \eqref{Gf-exterior2} and the boundary condition $\Big\|\frac{\partial P}{\partial \nu}- \frac{\partial p}{\partial \nu} \Big|_{+} \Big\|_{L^2(\partial \Omega)} < \epsilon$, the following inequalities hold:
\begin{align*}
|p(x) - P(x)|&\leq |p_{\infty} - P_{\infty}| + \Big|\int_{\p \Omega} G(x,y) \Big(\frac{\p p}{\p \nu}(y)-\frac{\p P}{\p \nu}(y)\Big)\d s(y) - \int_{\p \Omega} \frac{\p G(x,y)}{\p \nu(y)}\big(p(y) - P(y)\big)\d s(y)\Big|\\
&\leq \frac{1}{2 \pi r} \int_{|y|=r} |p(y)- P(y)| \d s(y) + \Big(\int_{\p \Omega} G^2(x,y)\d s(y)\Big)^{\frac{1}{2}} \cdot \Big(\int_{\p \Omega}\Big(\frac{\p p}{\p \nu}(y)-\frac{\p P}{\p \nu}(y)\Big)^2\d s(y)\Big)^{\frac{1}{2}}\\
&\leq C\Big\|\frac{\partial P}{\partial \nu}- \frac{\partial p}{\partial \nu} \Big|_{+} \Big\|_{L^2(\partial \Omega)}\\
&<\epsilon, \quad x\in \R^2 \setminus \overline{\Omega}.
\end{align*}
We next prove the case of shielding.
Setting $\zeta_0=\zeta_{0,opt}$ and using the condition $\mathcal{G}(\zeta_{0,opt})<\epsilon^2$ and the boundary condition $\frac{\partial p}{\partial \nu} \Big|_{+} - \frac{\partial p}{\partial \nu} \Big|_{-} = 12 \zeta_0\frac{\partial \varphi}{\partial \nu}$  on $\partial \Omega$, we obtain $\Big\|\frac{\partial p}{\partial \nu}\Big|_{-}\Big\|_{L^2(\partial \Omega)} < \epsilon$.
Applying Green's formula \cite{Kress1989} to $p$ in $\Omega \setminus \overline{D}$, we have
\begin{align}\label{Gf-interior}
\ds p(x)&=\int_{\p \Omega} \frac{\p G(x,y)}{\p \nu(y)}p(y)\d s(y) - \int_{\p \Omega} G(x,y) \frac{\p p}{\p \nu}(y)\d s(y), \quad x\in \Omega \setminus \overline{D}.
\end{align}
In the case of shielding, we require  $p|_{-}=0$ on $\p \Omega$. Moreover, from \eqref{Gf-interior} and  the boundary condition $\Big\|\frac{\partial p}{\partial \nu}\Big|_{-}\Big\|_{L^2(\partial \Omega)} < \epsilon$, the following inequalities hold:
\begin{align*}
|p(x)|&=\Big|\int_{\p \Omega} \frac{\p G(x,y)}{\p \nu(y)}p(y)\d s(y) - \int_{\p \Omega} G(x,y) \frac{\p p}{\p \nu}(y)\d s(y)\Big|\\
&= \Big|\int_{\p \Omega} G(x,y) \frac{\p p}{\p \nu}(y)\d s(y)\Big|\\
&\leq \Big(\int_{\p \Omega} G^2(x,y)\d s(y)\Big)^{\frac{1}{2}}\cdot \Big(\int_{\p \Omega}\Big(\frac{\p p}{\p \nu}(y)\Big)^2\d s(y)\Big)^{\frac{1}{2}}\\
&\leq C\Big\|\frac{\partial p}{\partial \nu}\Big|_{-}\Big\|_{L^2(\partial \Omega)}<\epsilon, \quad x\in \Omega \setminus \overline{D}.
\end{align*}

The proof is complete.
\end{proof}

In the proof of  Theorem \ref{thm-estimation}, we find the following corollary.
\begin{cor}\label{cor} Let $p$ be the solution to (\ref{electro-osmotic equation}) with $p|_{+}=P$ on $\p \Omega$. Then there exists a positive constant $C$ such that
\begin{align*}
|p(x) - P(x)|\leq C\Big\| \frac{\partial P}{\partial \nu} - \frac{\partial p}{\partial \nu}\Big|_{-} -12 \zeta_0 \frac{\partial \varphi}{\partial \nu}\Big\|_{L^2(\partial \Omega)}, \quad x\in \R^2 \setminus \overline{\Omega}.
\end{align*}
 Let $p$ be the solution to (\ref{electro-osmotic equation}) with $p|_{-}=0$ on $\p \Omega$. Then there exists a positive constant $C$ such that
\begin{align*}
|p(x)|\leq C\Big\|\frac{\partial p}{\partial \nu}\Big|_{+}
-12 \zeta_0 \frac{\partial \varphi}{\partial \nu}\Big\|_{L^2(\partial \Omega)}, \quad x\in \Omega \setminus \overline{D}.
\end{align*}
\end{cor}
Corollary \ref{cor} shows that the error is minimized when the cost function $\mathcal{F}(\zeta_0)$ (or $\mathcal{G}(\zeta_0)$) is minimized. It will be further corroborated by our numerical experiments in Section \ref{sec-NumSim} in what follows.

\section{Numerical experiments}\label{sec-NumSim}

In this section, we validate the theoretical results based on a two-dimensional model by performing three-dimensional finite-element simulations, which shows nice
agreement. We perform the finite-element numerical simulations using the commercial software COMSOL Multiphysics.

Before showing the numerical results, we give the values of the physical and geometrical parameters used in the finite-element simulations, which are summarized in the following table.
\begin{table}[H]
  \centering
  \begin{tabular}{cccc}
\centering
  Physical property & Notation & Value &Units \\
  \hline
  Gap between the plates & $\tilde{h}$ & 15 & $\mu \mathrm{m}$ \\
  Length of computational domain & $\tilde{L}$ & 2 &$\mathrm{mm}$\\
  Width of computational domain& $\tilde{W}$  & 2&$\mathrm{mm}$\\
  Density of fluid & $\rho$ & $10^{3}$ &$\mathrm{Kg}/\mathrm{m}^{3}$\\
  Viscosity of fluid & $\mu$ & $10^{-3}$ &$\mathrm{Pa} \cdot \mathrm{s}$\\
  Permittivity of fluid  & $\varepsilon$ & $7.08 \times 10^{-10}$ & $\mathrm{F}/\mathrm{m}$\\
  Electric field far from object& $\tilde{E}$ & $3\times 10^{2}  $ & $\mathrm{V}/\mathrm{m}$\\
  External velocity & $\tilde{u}_{ext}$ & $51$ & $\mu \mathrm{m} / \mathrm{s}$\\
  \hline
\end{tabular}
  \caption{Values of the physical and geometrical parameters used in the finite-element simulations.}
\end{table}

We first perform finite-element simulations of the flow around a circular cylinder of radius $100\  \mu \mathrm{m}$ in a Hele-Shaw cell. The radius of the cloaking and shielding region is $200\  \mu \mathrm{m}$. We derive the important quantity, that is, the characteristic value of $\tilde{\zeta}_0$.
Scaling (\ref{annulus-p}) by the characteristic dimensions, we know that the pressure distribution is \cite{Boyko2021}
 \begin{align}\label{annulus-p-tilde}
 \tilde{p}=
  \begin{cases}
\ds \frac{6\mu}{\tilde{h}^2 \tilde{r}_e^{2}}\Big((\tilde{r}_e^{2}-\tilde{r}_i^{2})\tilde{u}_{EOF}-2\tilde{r}_e^{2}\tilde{u}_{ext}\Big)\Big(\tilde{r}+\frac{\tilde{r}_i^{2}}{\tilde{r}}\Big)\cos\theta,\quad \tilde{r}_i<\tilde{r}<\tilde{r}_e,
\vspace{1em}\\
\ds -\frac{12 \mu \tilde{u}_{ext}}{\tilde{h}^2} \tilde{r}\cos\theta+\frac{6\mu}{\tilde{h}^2 \tilde{r}_e^{2}}\Big((\tilde{r}_e^{4}-\tilde{r}_i^{4})\tilde{u}_{EOF}- 2\tilde{r}_i^{2} \tilde{r}_e^{2}\tilde{u}_{ext}\Big)\frac{1}{\tilde{r}}\cos\theta,\quad \tilde{r}>\tilde{r}_e.
\end{cases}
 \end{align}
 Where $\tilde{u}_{EOF}= -\varepsilon \tilde{\zeta}_0 \tilde{E}/\mu$, $\tilde{r}_i = 100\  \mu \mathrm{m}$ and $\tilde{r}_e = 200\  \mu \mathrm{m}$.  From (\ref{annulus-p-tilde}), we can obtain shielding and cloaking conditions
 \begin{align*}
\tilde{u}_{EOF} =\frac{2\tilde{r}_e^{2}}{\tilde{r}_e^{2}-\tilde{r}_i^{2}}\tilde{u}_{ext} \quad \mbox{and} \quad \tilde{u}_{EOF}= \frac{2\tilde{r}_i^{2} \tilde{r}_e^{2}}{\tilde{r}_e^{4}-\tilde{r}_i^{4}}\tilde{u}_{ext}.
 \end{align*}
 Furthermore, in terms of  $\tilde{\zeta}_0$, we have shielding and cloaking conditions
 \begin{align}\label{shielding-cloaking-condition}
\tilde{\zeta}_0 =-\frac{\mu}{\varepsilon\tilde{E} }\frac{2\tilde{r}_e^{2}}{\tilde{r}_e^{2}-\tilde{r}_i^{2}}\tilde{u}_{ext} \quad \mbox{and} \quad \tilde{\zeta}_0= -\frac{\mu}{\varepsilon\tilde{E} } \frac{2\tilde{r}_i^{2} \tilde{r}_e^{2}}{\tilde{r}_e^{4}-\tilde{r}_i^{4}}\tilde{u}_{ext}.
 \end{align}
 The equations in (\ref{shielding-cloaking-condition}) are equivalent to
 \begin{align}\label{shielding-cloaking-condition-equivalence}
\tilde{\zeta}_0 =-\frac{2r_e^{2}}{r_e^{2}-r_i^{2}}\frac{\mu}{\varepsilon\tilde{E}}\tilde{u}_{ext} \quad \mbox{and} \quad \tilde{\zeta}_0= - \frac{2r_i^{2} r_e^{2}}{r_e^{4}-r_i^{4}}\frac{\mu}{\varepsilon\tilde{E}}\tilde{u}_{ext}.
 \end{align}
Comparing the dimensionless shielding and cloaking conditions (\ref{annulus-cloaking-zeta}, \ref{annulus-shielding-zeta}) and (\ref{shielding-cloaking-condition-equivalence}), we find that the difference is a factor. In fact, the factor is precisely the characteristic value of $\tilde{\zeta}_0$.

Figure \ref{fig-circle} presents a comparison of the finite-element simulation results corresponding to pressure-driven flow (a), shielding (b), and cloaking (c) conditions. Figures \ref{fig-circle}(a)--\ref{fig-circle}(c) present the resulting pressure distribution (colormap) and streamlines (white lines), showing excellent shielding and cloaking in the finite-element simulations for the cylinder object.  Under shielding conditions (\ref{annulus-shielding-zeta}), the pressure inside the region surrounding the cylinder becomes uniform, and the force on the object vanishes. Under cloaking conditions (\ref{annulus-cloaking-zeta}), the streamlines outside of the control region are straight, unmodified relative to the uniform far field, and undisturbed by the object.
\begin{figure}[H]
	\centering  
	\subfigcapskip=-10pt 
	\subfigure[]{
		\includegraphics[width=0.32\linewidth]{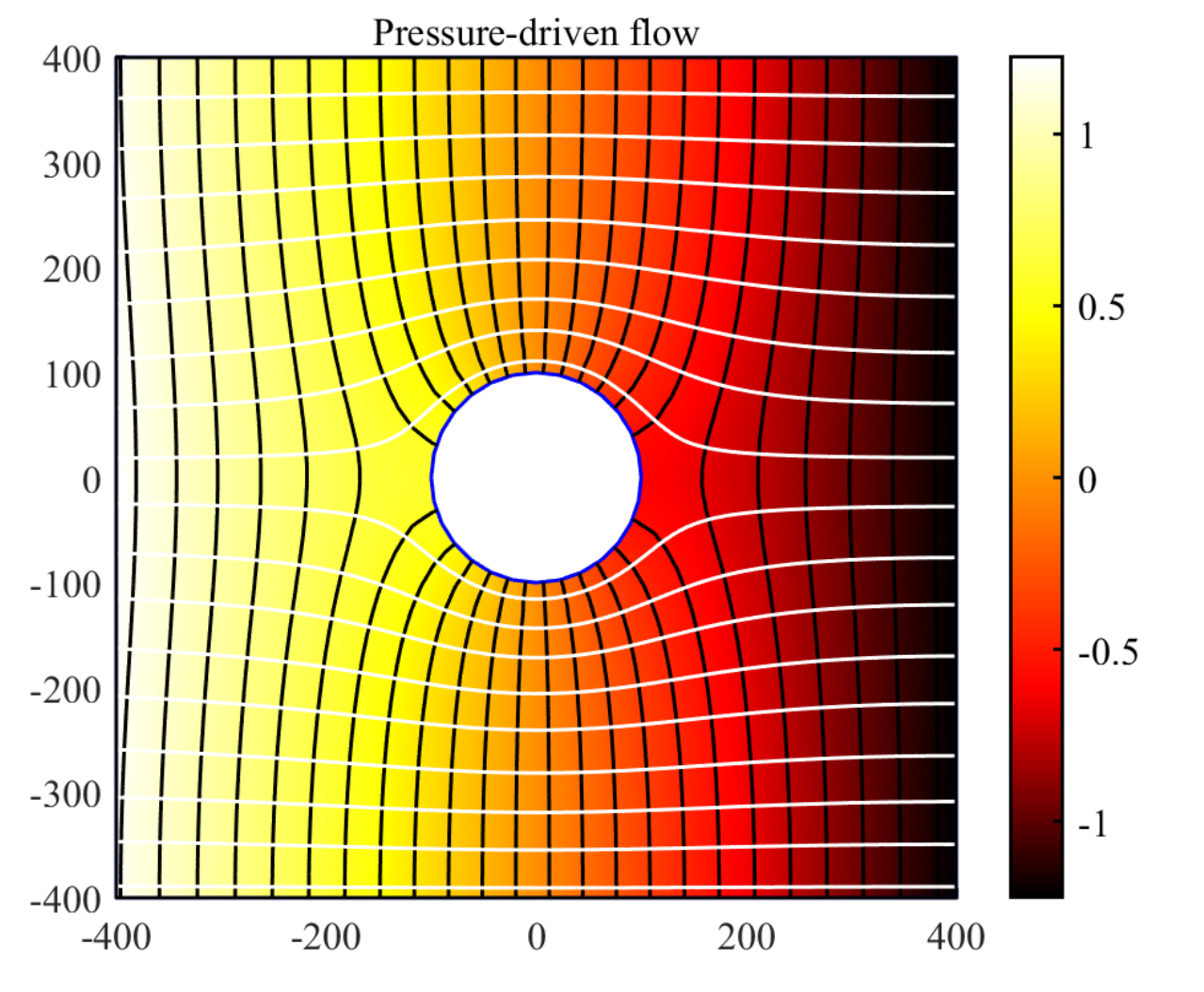}}
	\subfigure[]{
		\includegraphics[width=0.32\linewidth]{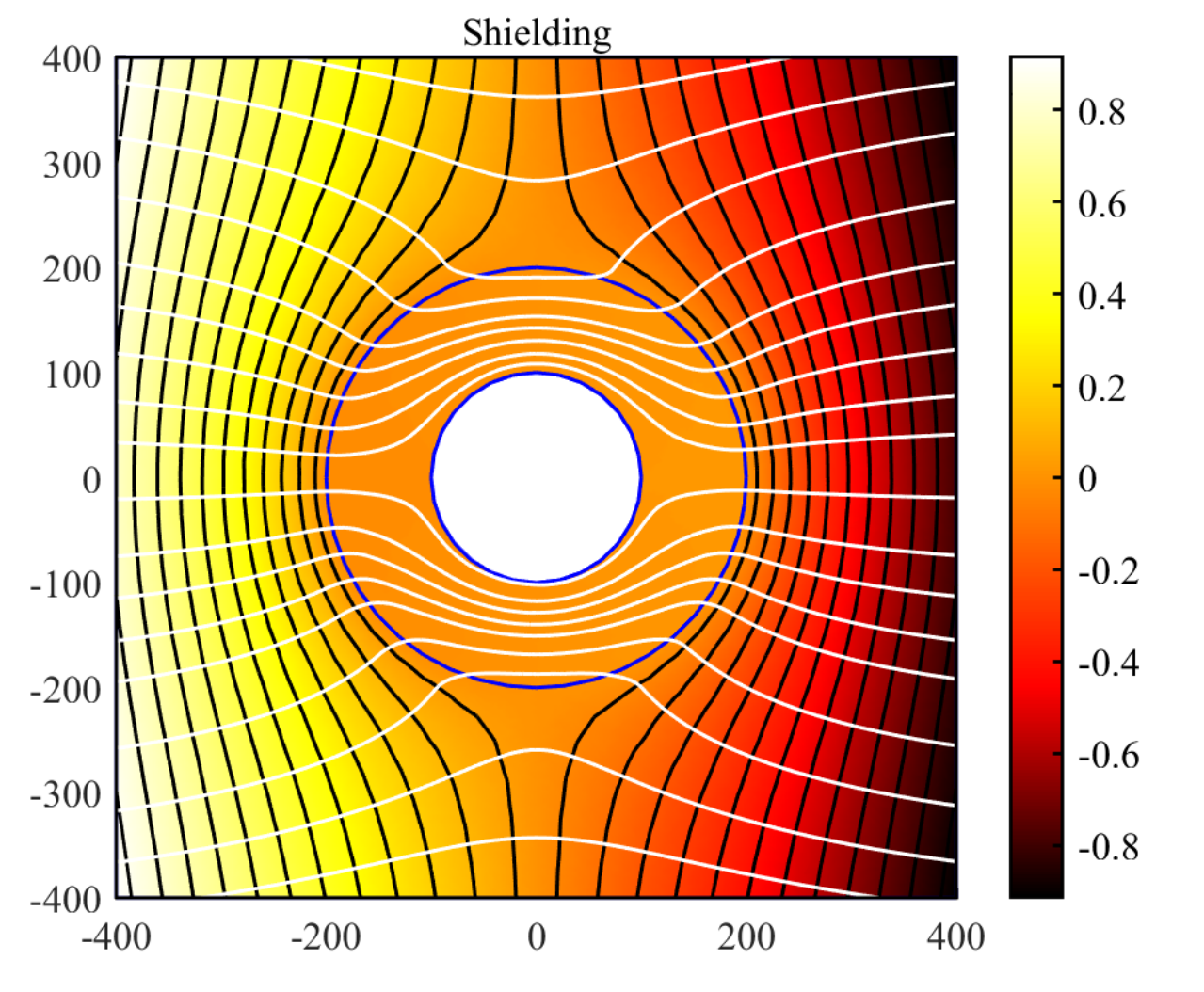}}
	\subfigure[]{
		\includegraphics[width=0.32\linewidth]{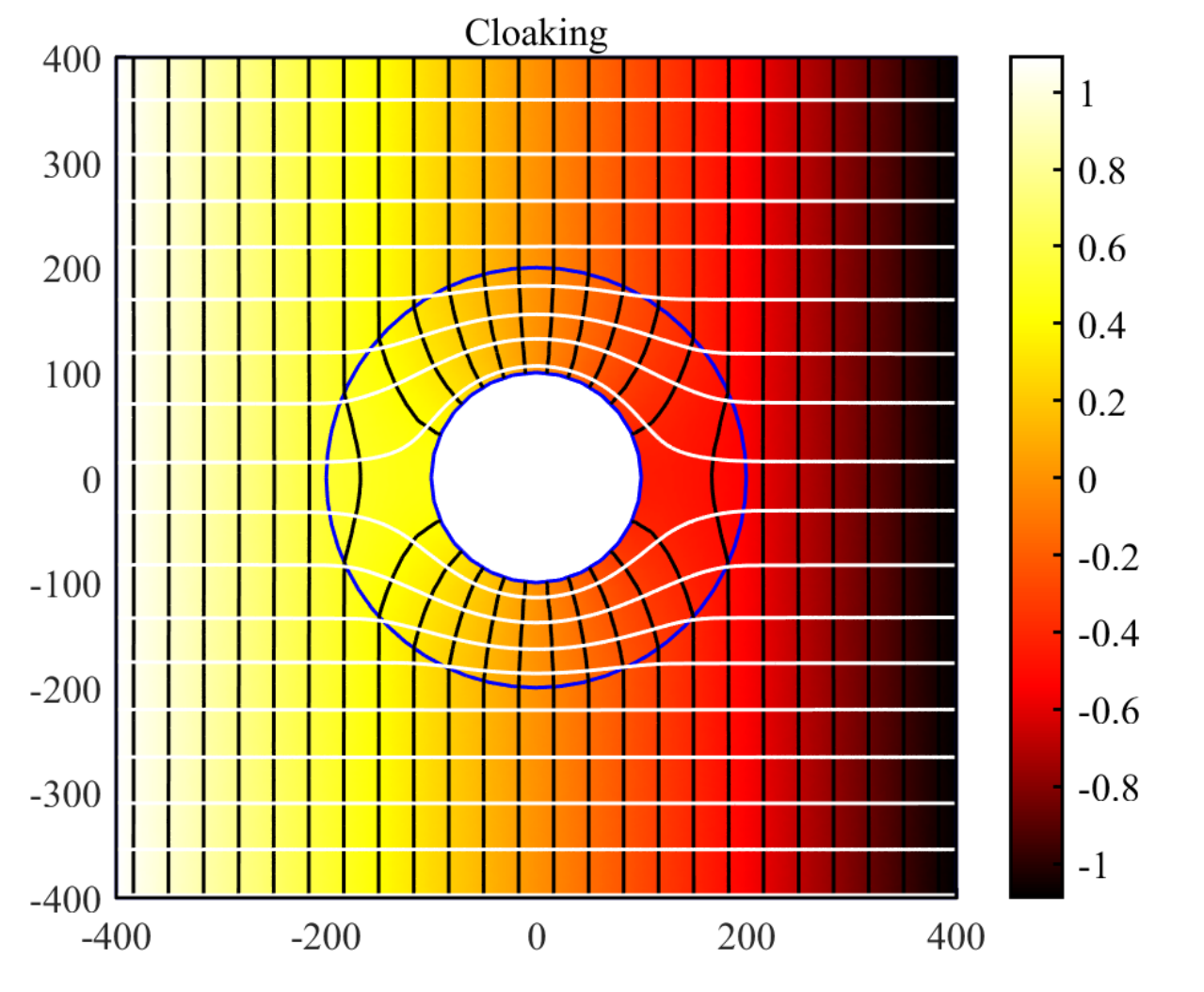}}
	\caption{Comparison of finite-element simulation results on the annulus. Numerical (a,b,c) results for the pressure distribution (colormap) and streamlines (white lines),
corresponding to pressure-driven flow (a), shielding (b), and cloaking (c) conditions. Here in the case of shielding, the zeta potential is $\tilde{\zeta}_0=-0.64\, \mathrm{V}$, and in the case of cloaking the zeta potential is $\tilde{\zeta}_0=-0.128 \,\mathrm{V}$. }\label{fig-circle}
\end{figure}

We perform finite-element simulations of the flow around an elliptic cylinder in a Hele-Shaw cell.  We consider an elliptic cylinder of boundary $\tilde{\xi}_i$ wrapped by a region of interior and exterior boundaries $\tilde{\xi}_i$ and $\tilde{\xi}_e$, choosing $\tilde{\xi}_i= 50 \  \mu \mathrm{m}$ and $\tilde{\xi}_e= 100 \  \mu \mathrm{m}$.  It is remarked that the inner and outer ellipses are of the same focus. We first consider a uniform electric filed and velocity externally applied in the $\tilde{x}_1$ direction. The zeta potentials $\tilde{\zeta}_0$ of cloaking and shielding region depend on the equations (\ref{ellipse-cloaking-zeta-x}, \ref{ellipse-shielding-zeta-x}) and the characteristic value of $\tilde{\zeta}_0$. Figures \ref{fig-ellipse}(a)--\ref{fig-ellipse}(c) present the resulting pressure distribution (colormap) and streamlines (white lines) with the background electric field and velocity in the $\tilde{x}_1$ direction. Analogously, when the background electric field and velocity are externally applied in the $\tilde{x}_2$ direction, under conditions  (\ref{ellipse-cloaking-zeta-y}) and (\ref{ellipse-shielding-zeta-y}) we can obtain numerical results illustrated in Figures \ref{fig-ellipse}(d)--\ref{fig-ellipse}(f).  These results also show excellent agreement like the circular cylinder case. The performance of the proposed cloaking and shielding conditions have been numerically confirmed. Moreover, it also is verified  that the confocal ellipses have the same shielding condition, comparing Figure \ref{fig-ellipse}(b) and Figure \ref{fig-ellipse}(e).
\begin{figure}[H]
	\centering  
	\subfigcapskip=-10pt 
	\subfigure[]{
		\includegraphics[width=0.32\linewidth]{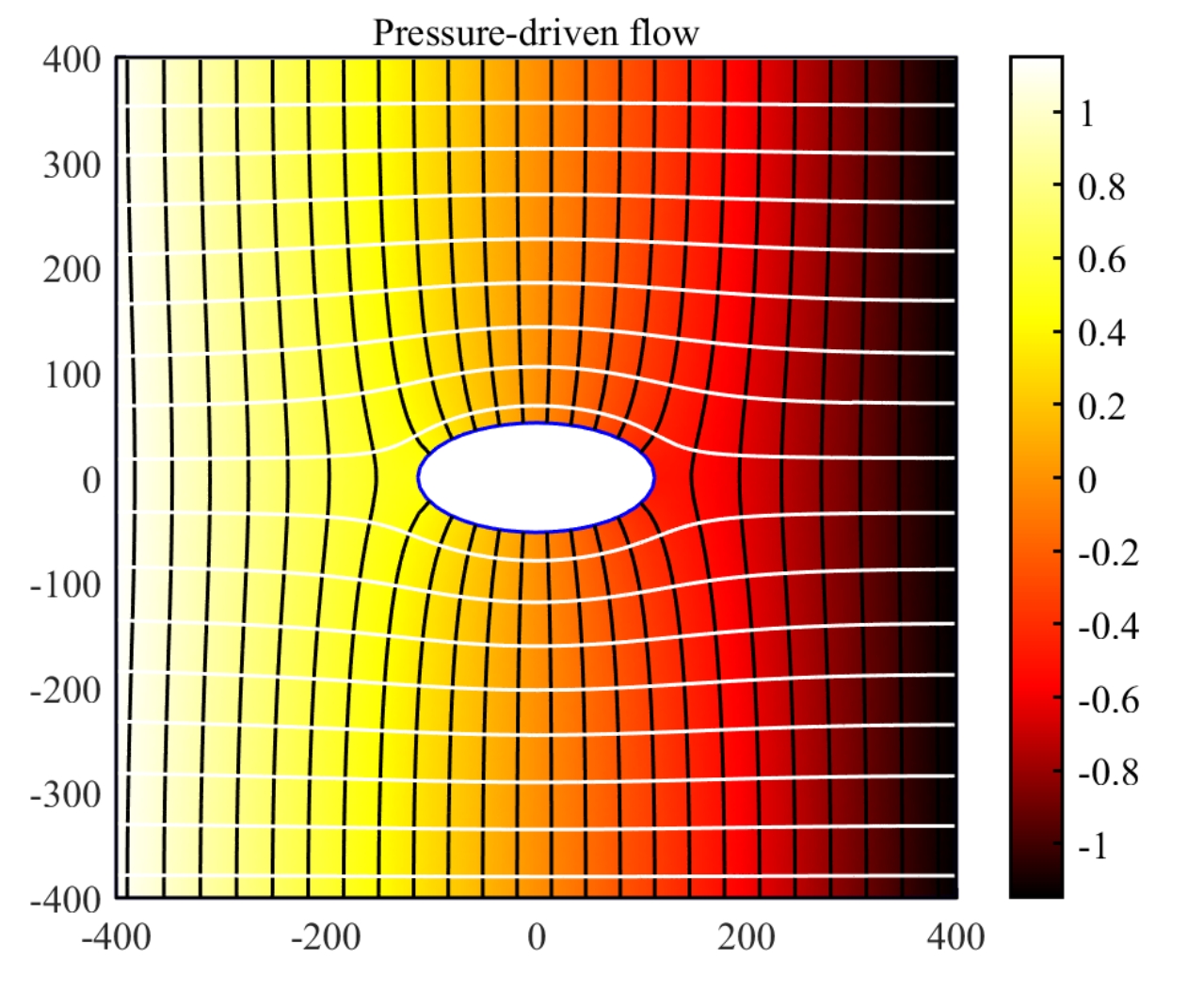}}
	\subfigure[]{
		\includegraphics[width=0.32\linewidth]{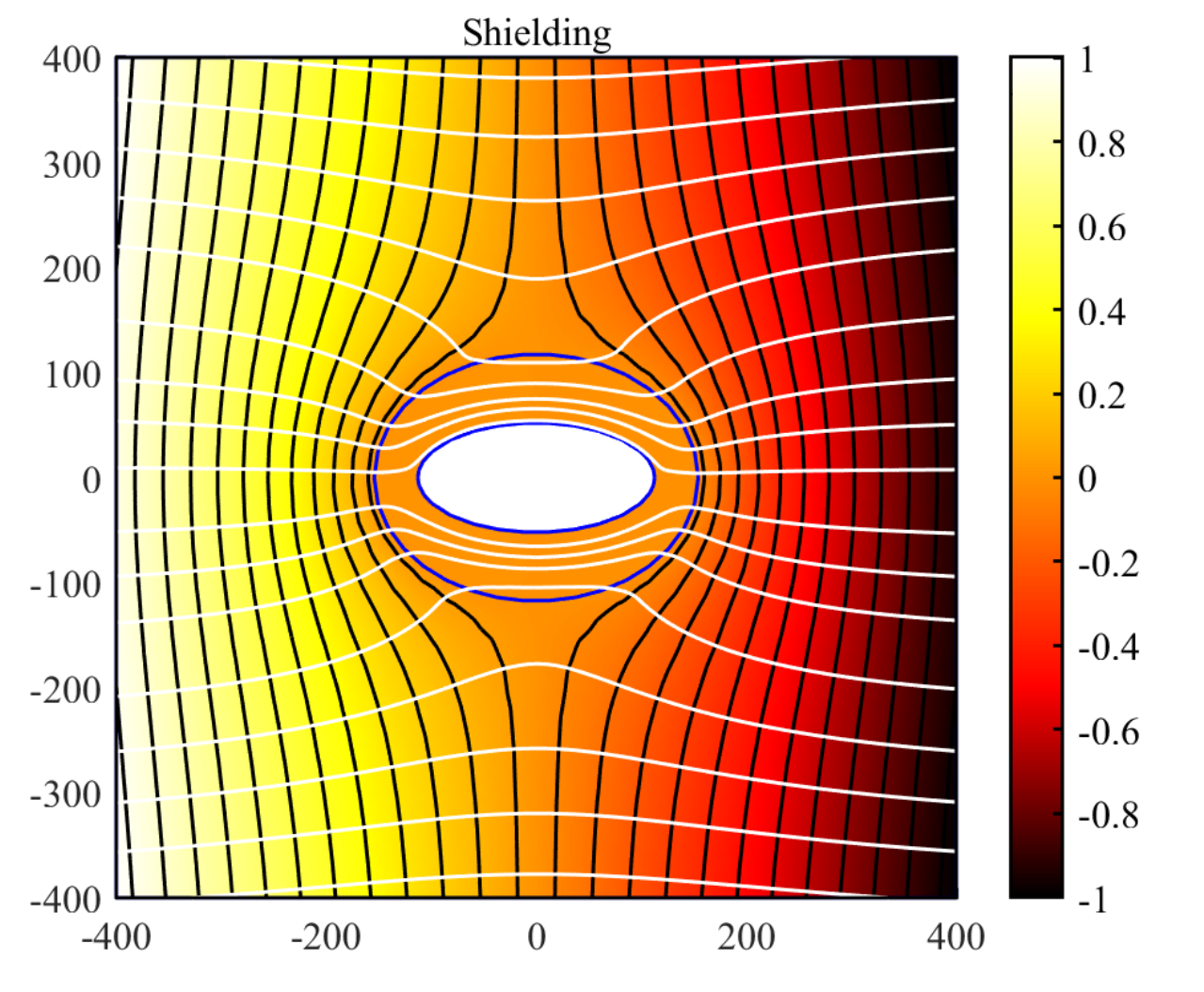}}
	\subfigure[]{
		\includegraphics[width=0.32\linewidth]{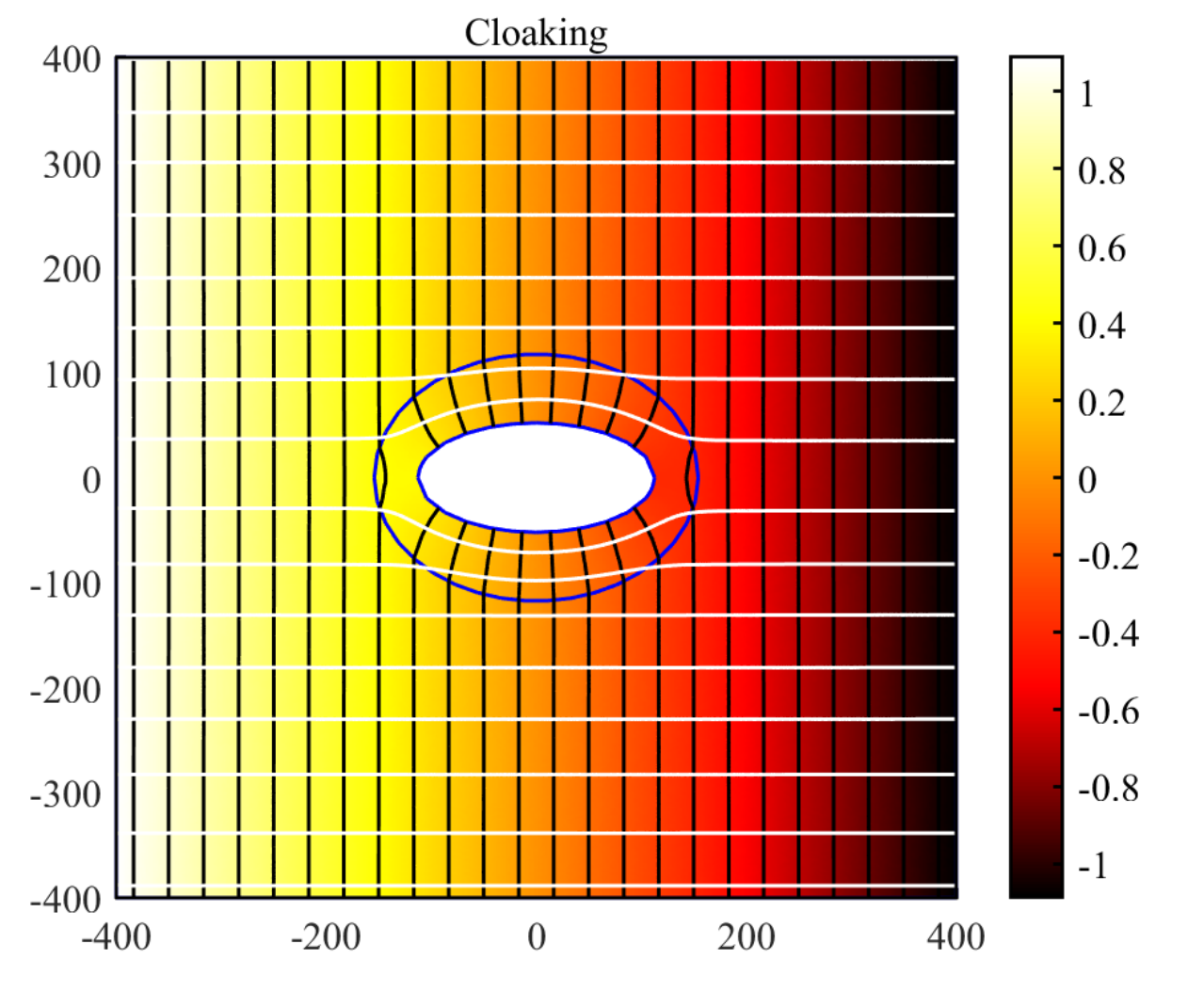}}\\
	\subfigure[]{
		\includegraphics[width=0.32\linewidth]{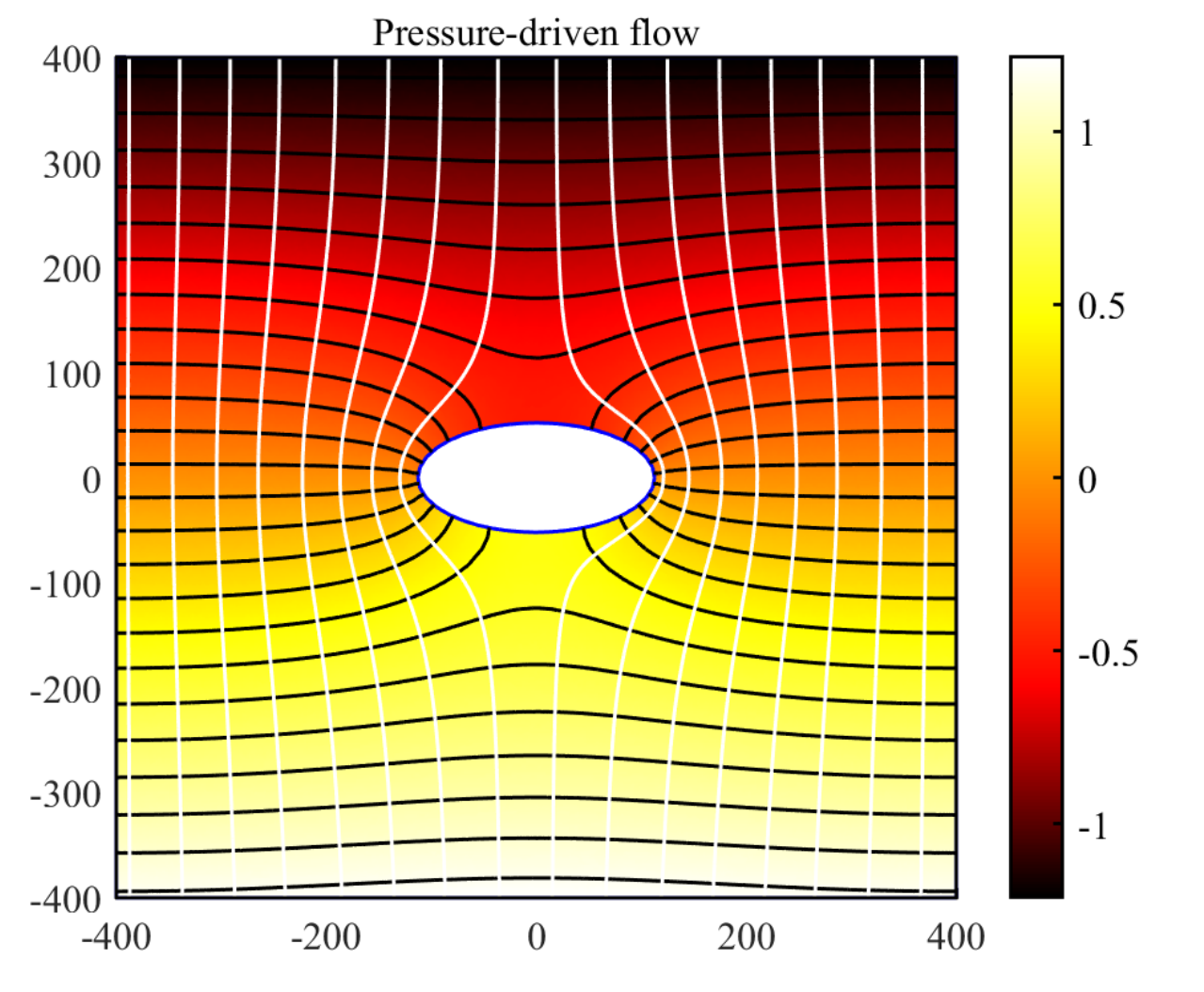}}
	\subfigure[]{
		\includegraphics[width=0.32\linewidth]{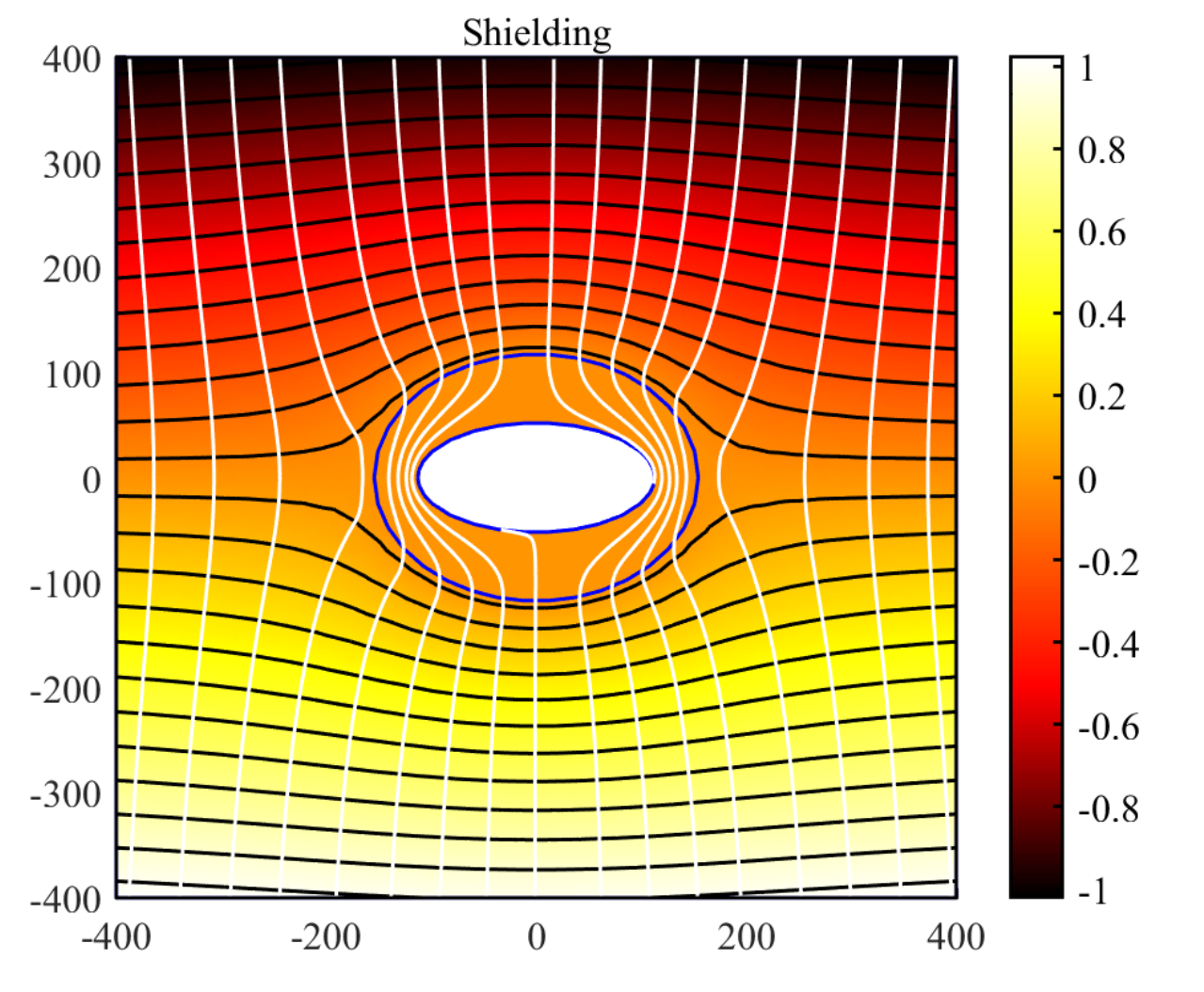}}
	\subfigure[]{
		\includegraphics[width=0.32\linewidth]{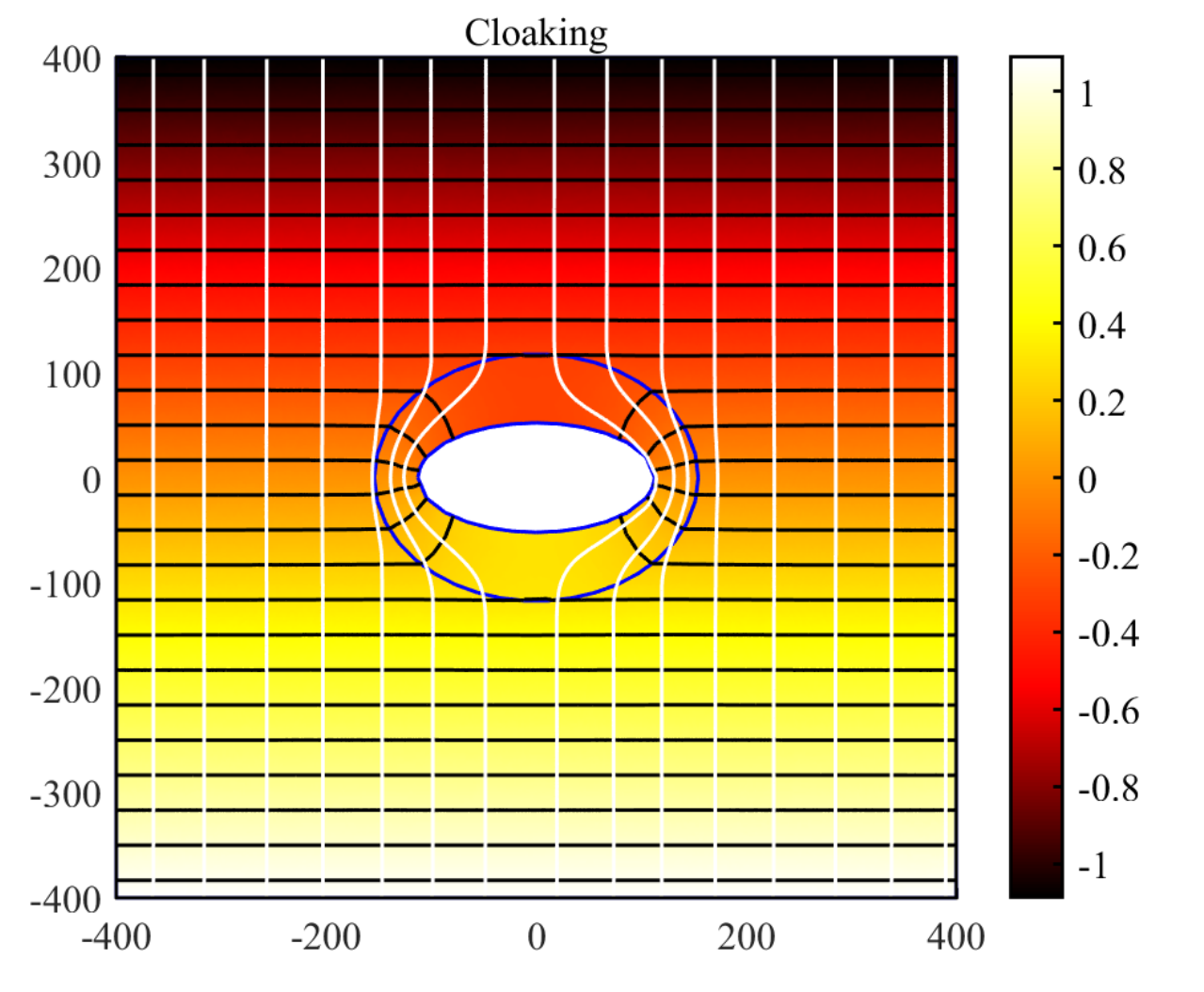}}
	\caption{Comparison of finite-element simulation results on the confocal ellipses with background field in the $\tilde{x}_1$ (Top: a,b,c) and $\tilde{x}_2$ direction (Bottom: d,e,f). The zeta potentials are $\tilde{\zeta}_0 = -0.7593  \ \mathrm{V} (b), -0.1291  \ \mathrm{V} (c), -0.7593  \ \mathrm{V} (e)$ and $-0.2793  \ \mathrm{V} (f)$, respectively.}\label{fig-ellipse}
\end{figure}

In Figure \ref{fig-thin}, we consider thin cloaking region for annulus and confocal ellipses with $\tilde{r}_i =100  \  \mu \mathrm{m}$,   $\tilde{r}_e =110 \  \mu \mathrm{m}$, $\tilde{\xi}_i =50 \  \mu \mathrm{m}$ and   $\tilde{\xi}_e =70 \  \mu \mathrm{m}$. These results show that thin region can also have excellent cloaking. However, these regions are not so thin that meta-surfaces occur because the cloaking conditions are singular.
\begin{figure}[H]
	\centering  
	\subfigcapskip=-10pt 
	\subfigure[]{
		\includegraphics[width=0.32\linewidth]{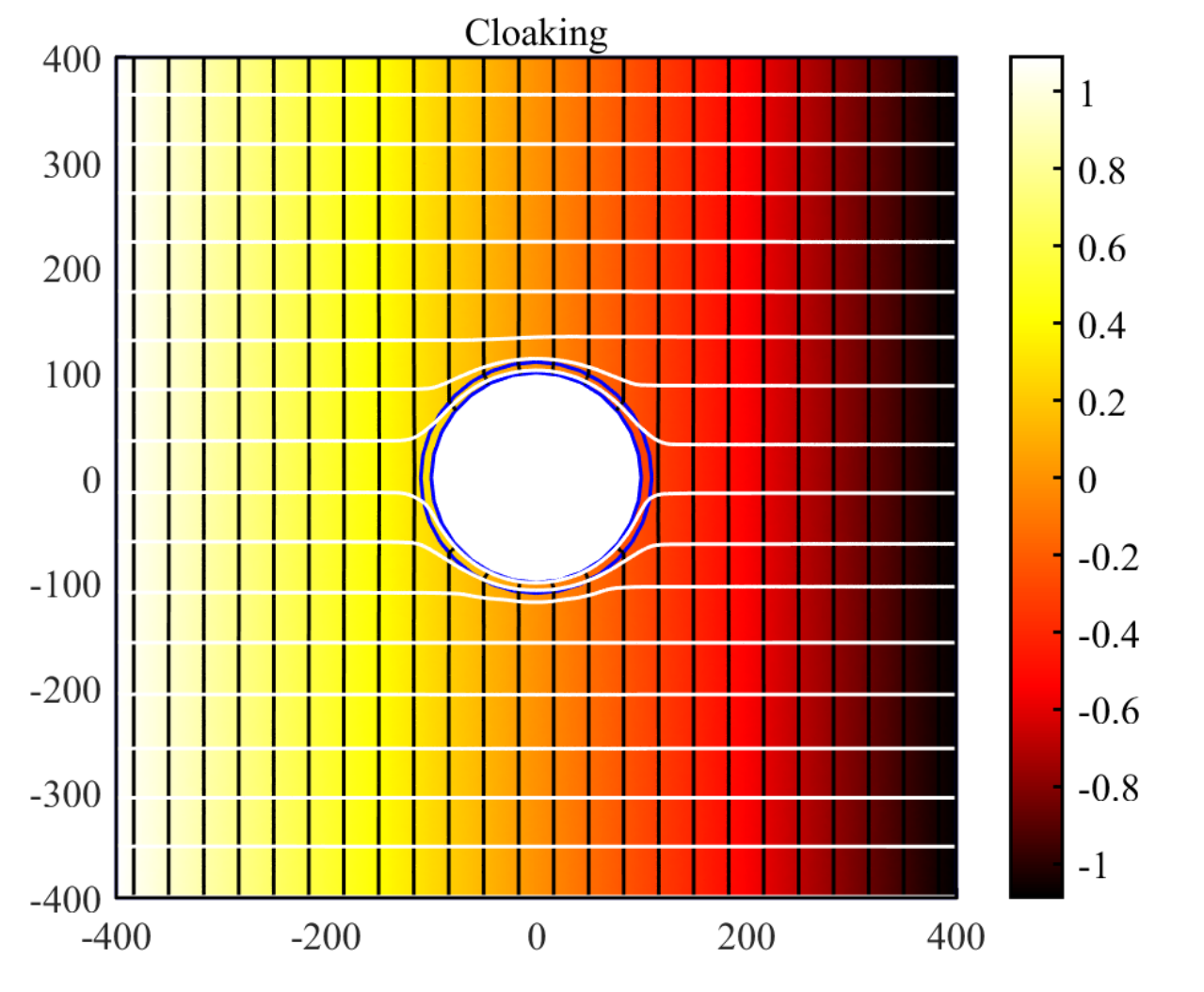}}
	\subfigure[]{
		\includegraphics[width=0.32\linewidth]{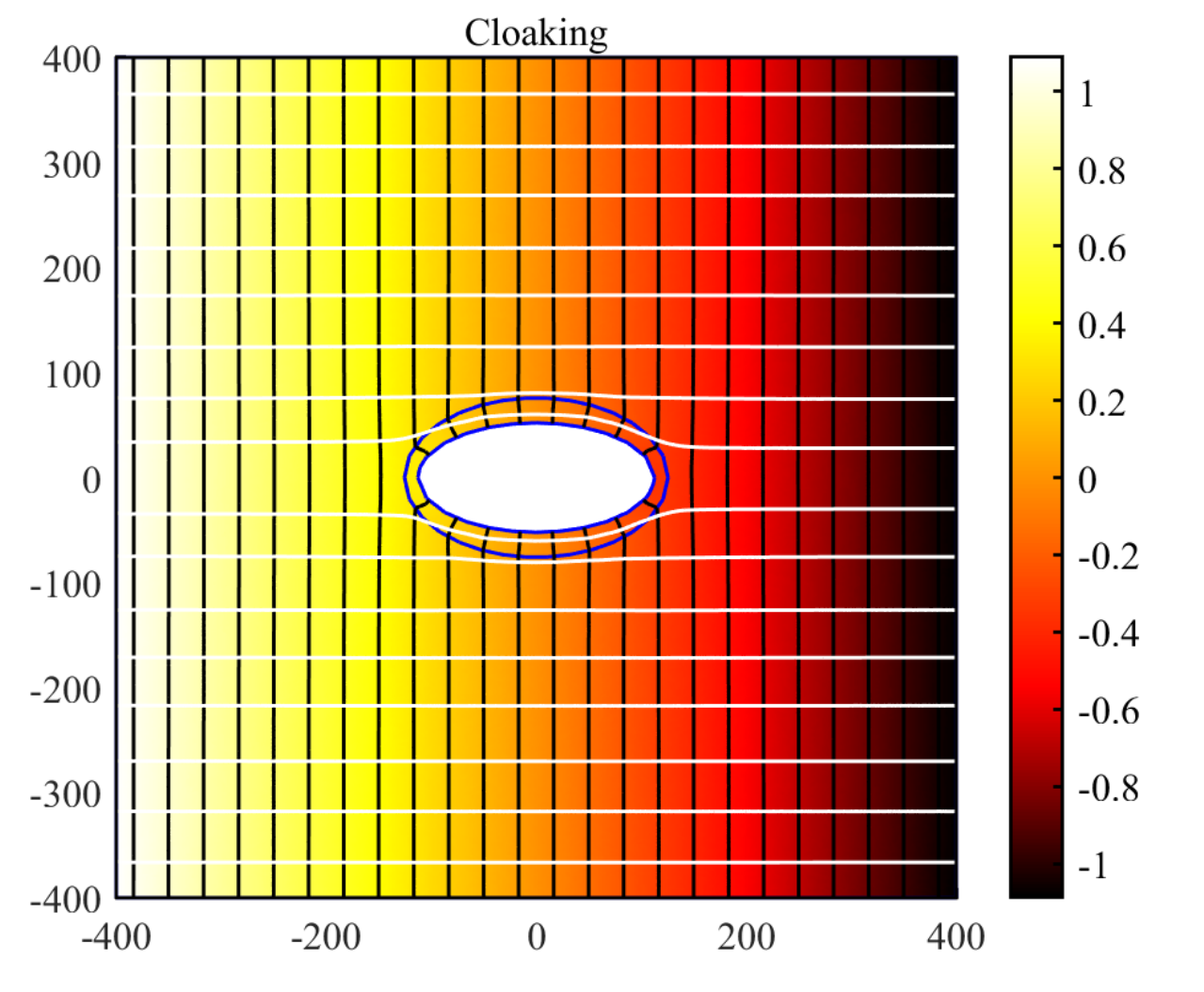}}
	\subfigure[]{
		\includegraphics[width=0.32\linewidth]{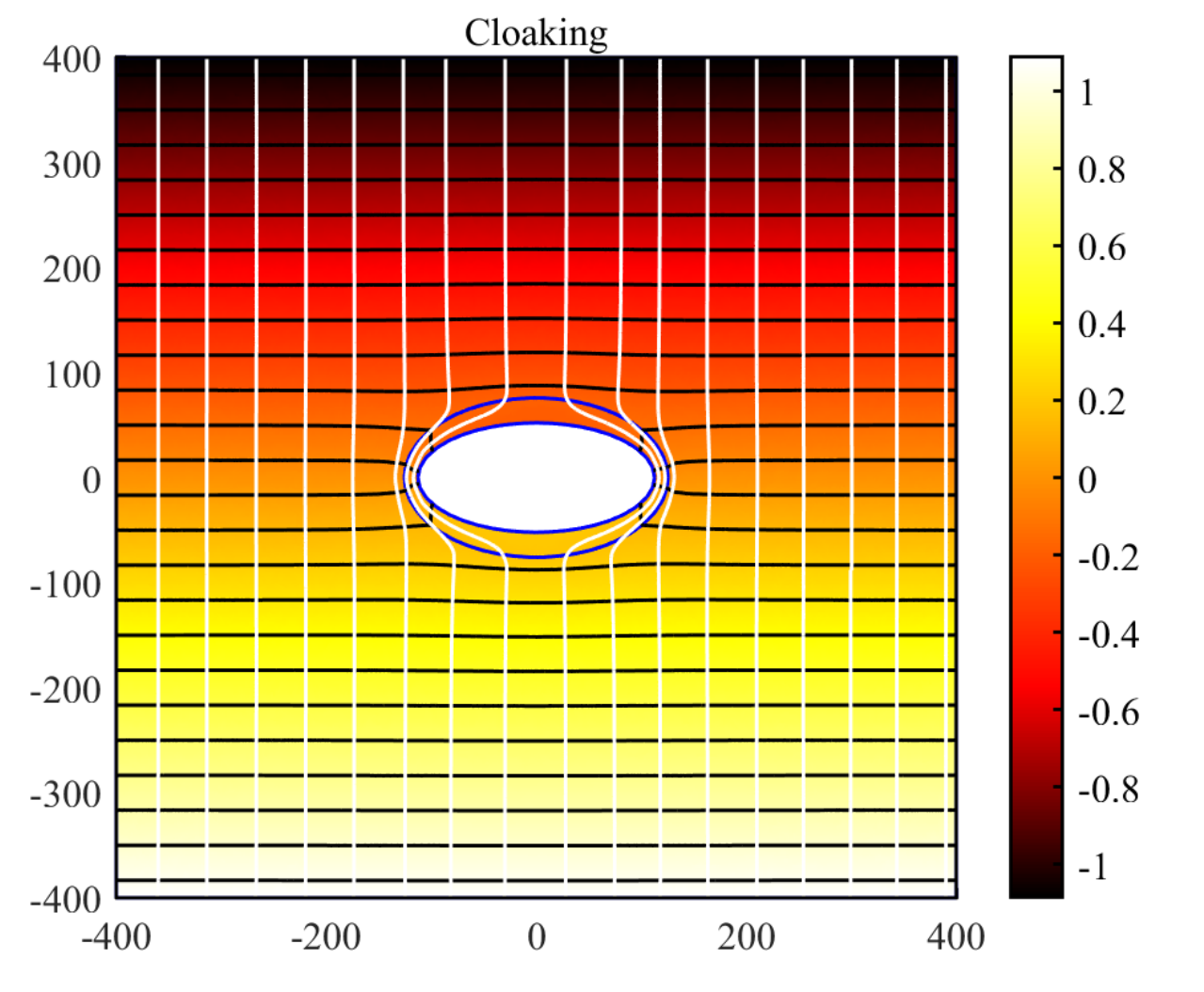}}
	\caption{Comparison of finite-element simulation results with thin cloaking region. The zeta potentials are $\tilde{\zeta}_0 = -1.2515\ \mathrm{V} (a), -0.3693  \ \mathrm{V} (b)$ and $-0.7992  \ \mathrm{V} (c)$, respectively.}\label{fig-thin}
\end{figure}

\begin{figure}[H]
	\centering  
	\subfigcapskip=-10pt 
	\subfigure[]{
		\includegraphics[width=0.32\linewidth]{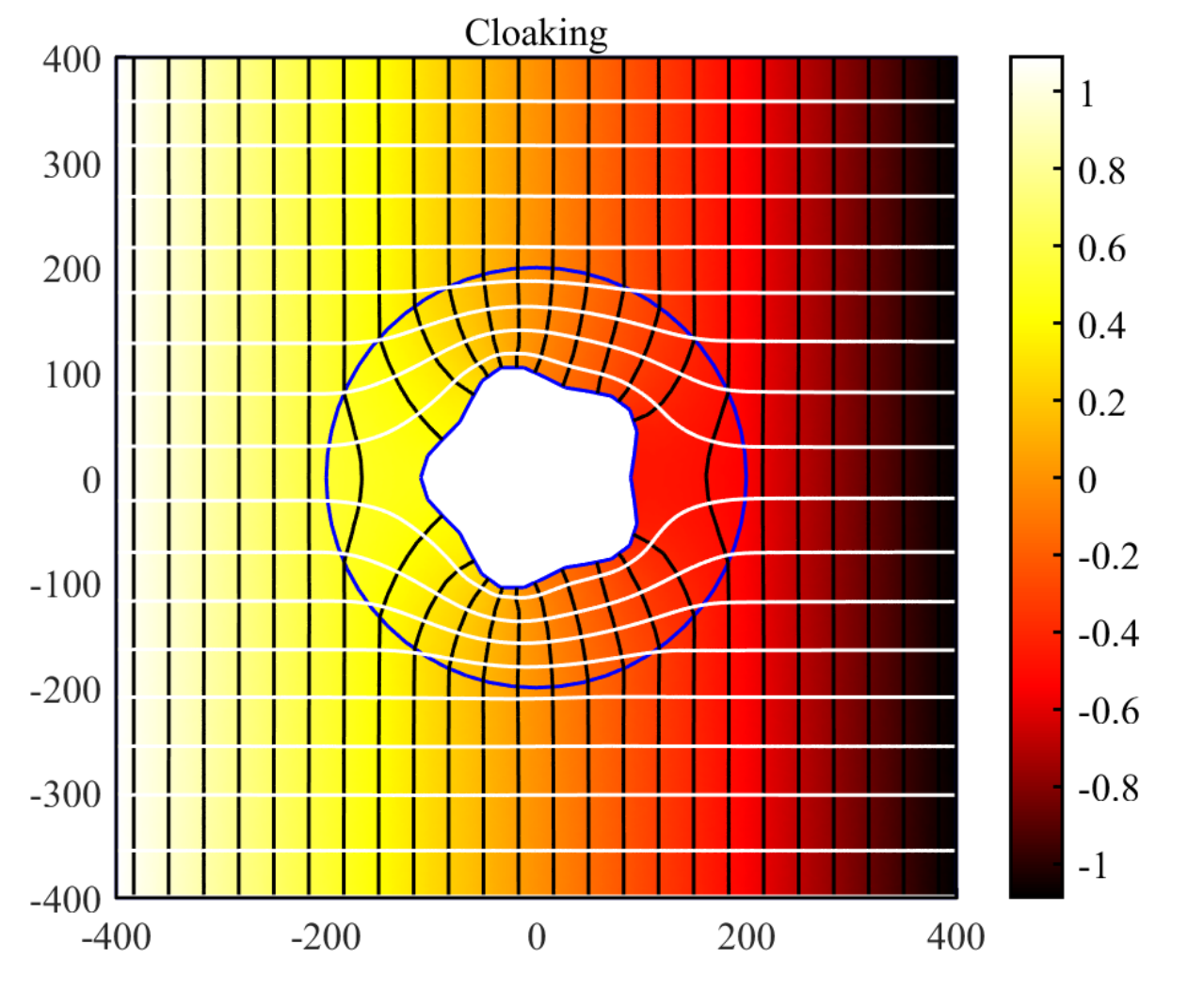}}
	\subfigure[]{
		\includegraphics[width=0.32\linewidth]{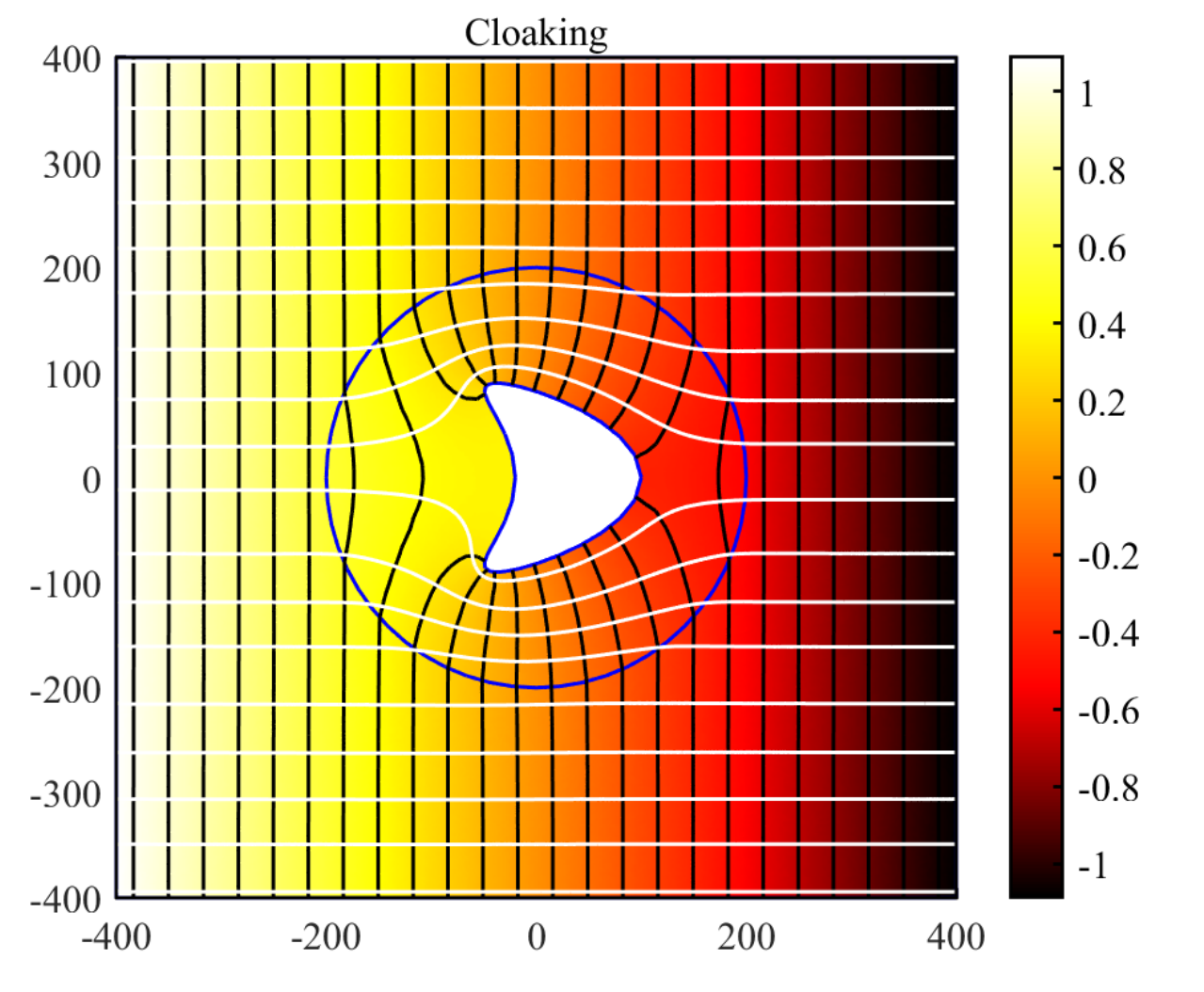}}
	\subfigure[]{
		\includegraphics[width=0.32\linewidth]{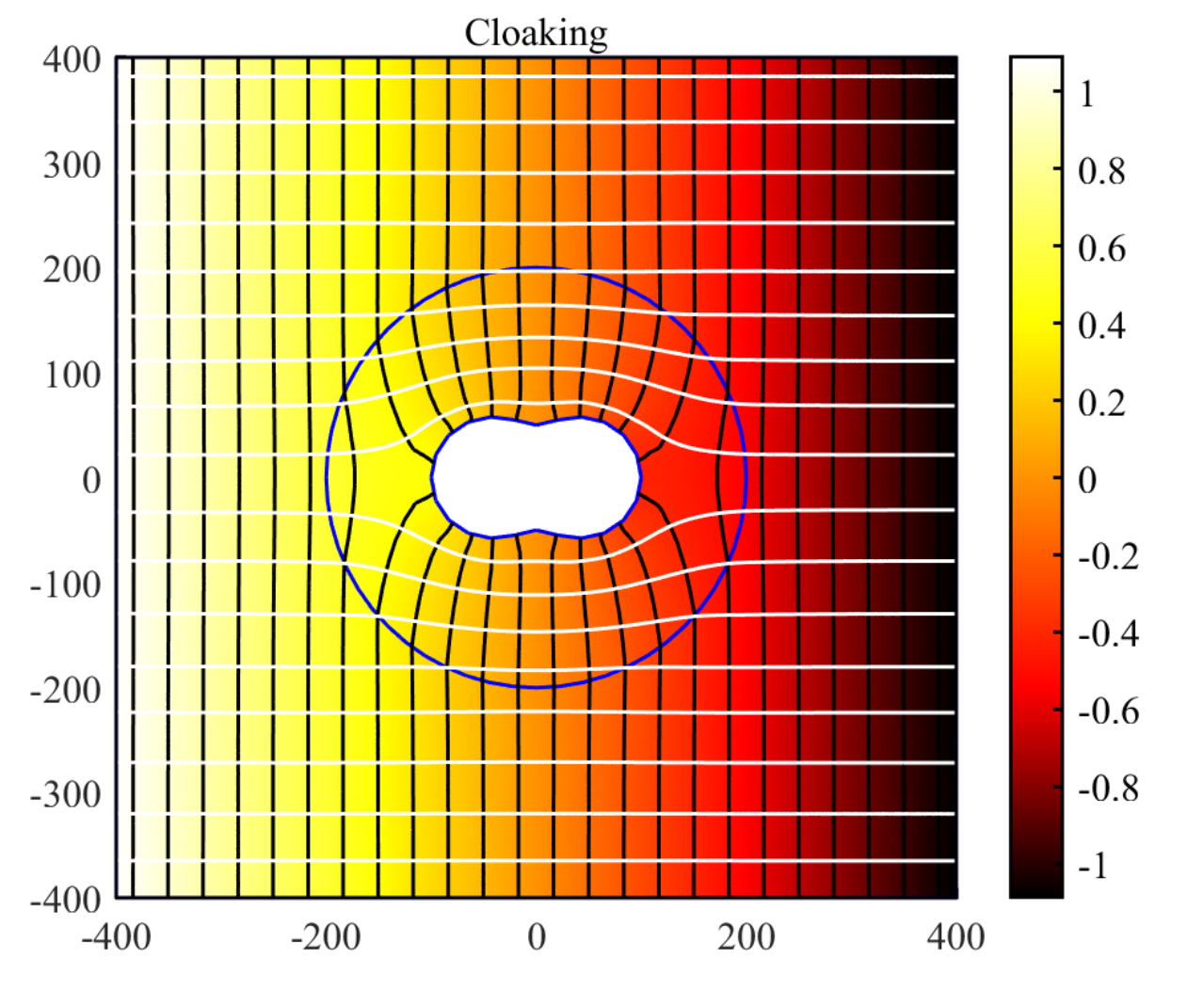}}\\
	\subfigure[]{
		\includegraphics[width=0.32\linewidth]{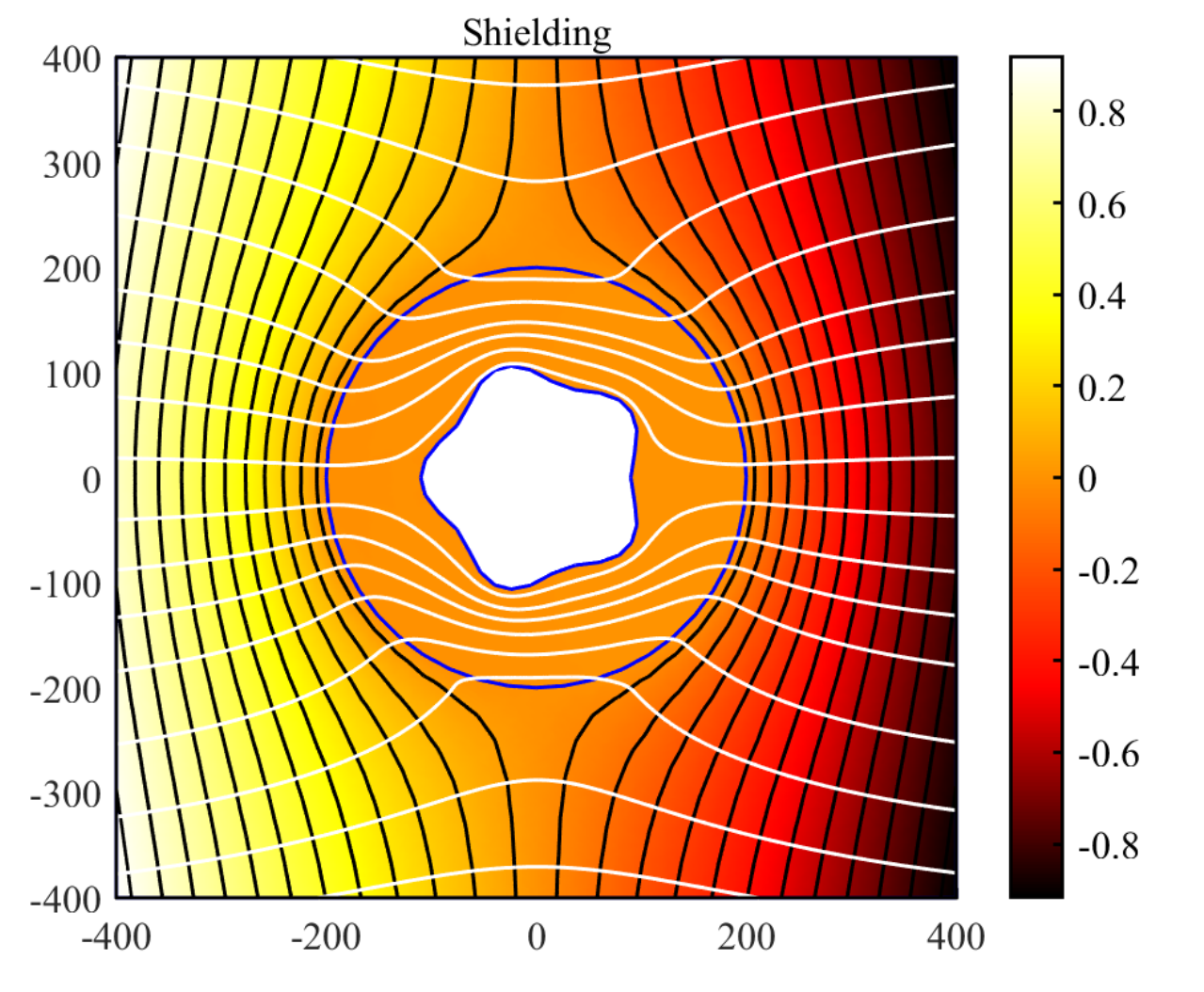}}
	\subfigure[]{
		\includegraphics[width=0.32\linewidth]{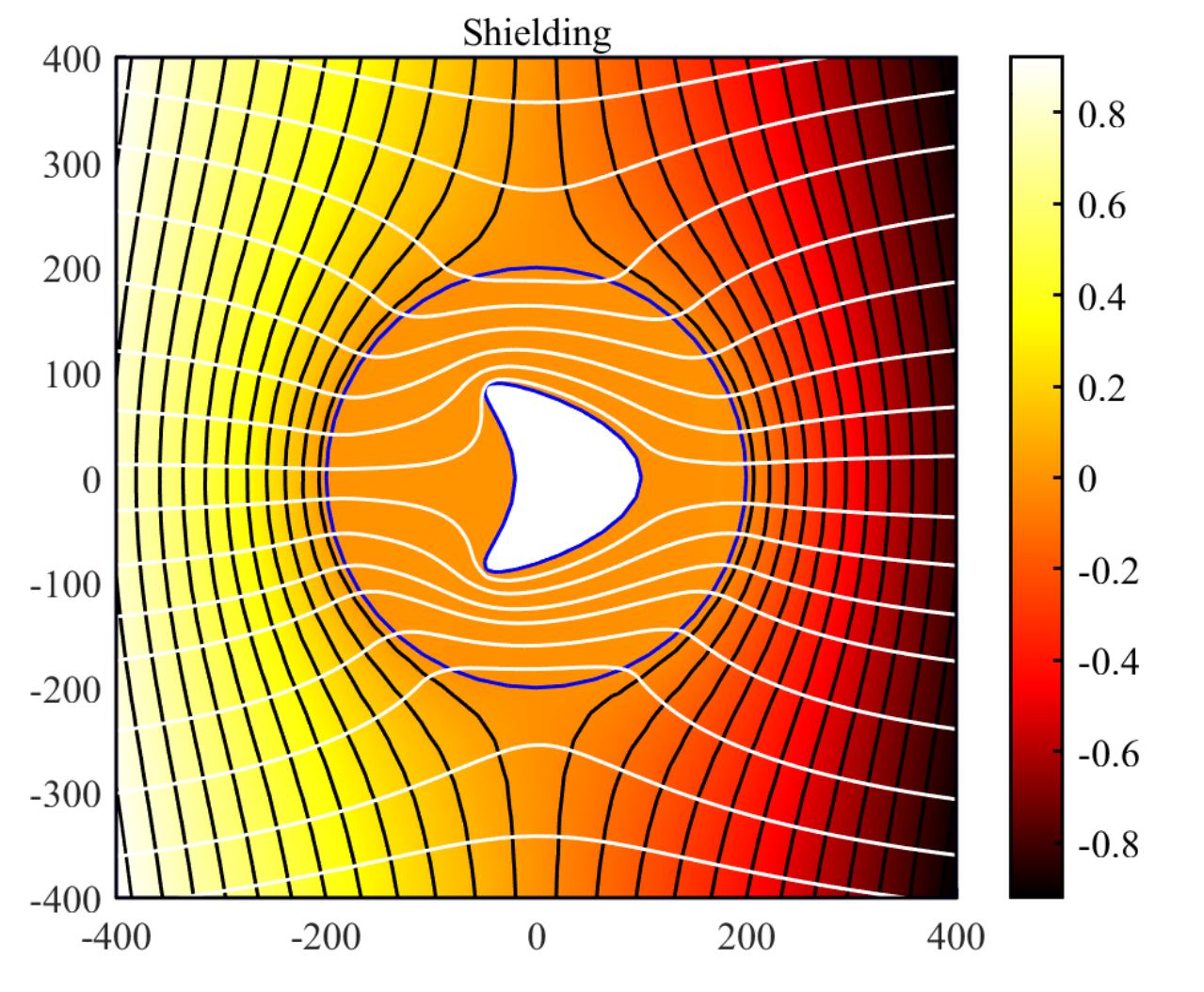}}
	\subfigure[]{
		\includegraphics[width=0.32\linewidth]{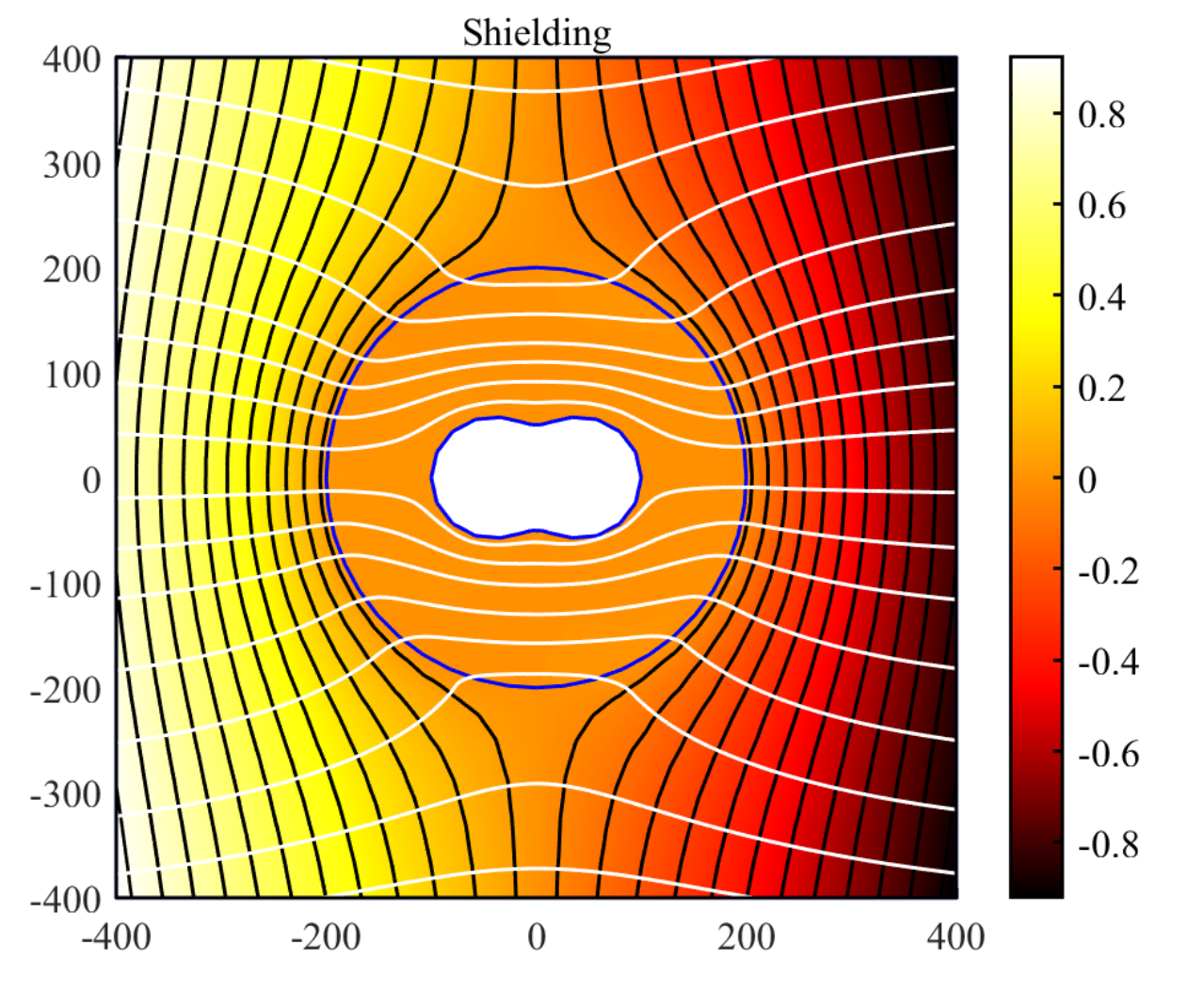}}
	\caption{Comparison of finite-element simulation results for some objects with regular boundaries. Numerical results for the pressure distribution (colormap) and streamlines (white lines)
corresponding to cloaking (a-c) and shielding (d-f). The zeta potentials are $\tilde{\zeta}_{0, opt} = -0.1288\ \mathrm{V} (a), -0.0864  \ \mathrm{V} (b), -0.0751  \ \mathrm{V} (c), -0.6456\ \mathrm{V} (d), -0.5832  \ \mathrm{V} (e)$ and $-0.5664  \ \mathrm{V} (f)$, respectively. The cloaking regions are circles with a radius of $200 \ \mu \mathrm{m}$.}\label{fig-regular-shape}
\end{figure}

For a benchmarking experiment, we consider the flow of a cylinder with a flower-shaped cross-section with boundary described by the parametric representation
\begin{equation*}
  x(t) = 1-0.1\cos 5t, \quad 0\leq t \leq 2\pi.
\end{equation*}
Extending our analysis to more complex shapes, we also considered the case of a non-convex kite-shaped object parameterized by
\begin{equation*}
  x(t) = (0.6\cos t + 0.39\cos 2t + 0.01, 0.9\sin t), \quad 0\leq t \leq 2\pi,
\end{equation*}
and a peanut-shaped object parameterized by
\begin{equation*}
  x(t) = \sqrt{\cos^2 t + 0.25 \sin^2 t}, \quad 0\leq t \leq 2\pi.
\end{equation*}
When these objects are surrounded by an appropriate circle, we observe that good cloaking and shielding occur in Figure \ref{fig-regular-shape}.

Without loss of generality, we confine our presentation to objects with boundary curve $\p D$ with some corners.  We consider some special shapes, for instance, triangles, squares and pentagons  inscribed in the circle of radius $100 \ \mu \mathrm{m}$. Here the cloaking and shielding region is a circle of radius $200 \ \mu \mathrm{m}$. Figure \ref{fig-corner-shape} shows good cloaking and shielding. These zeta potentials used in finite-element numerical simulations are calculated by the optimization method in Subsection \ref{subsec-optimal}. They are scaled by the characteristic value of $\tilde{\zeta}_0$.

\begin{figure}[H]
	\centering  
	\subfigcapskip=-10pt 
	\subfigure[]{
		\includegraphics[width=0.32\linewidth]{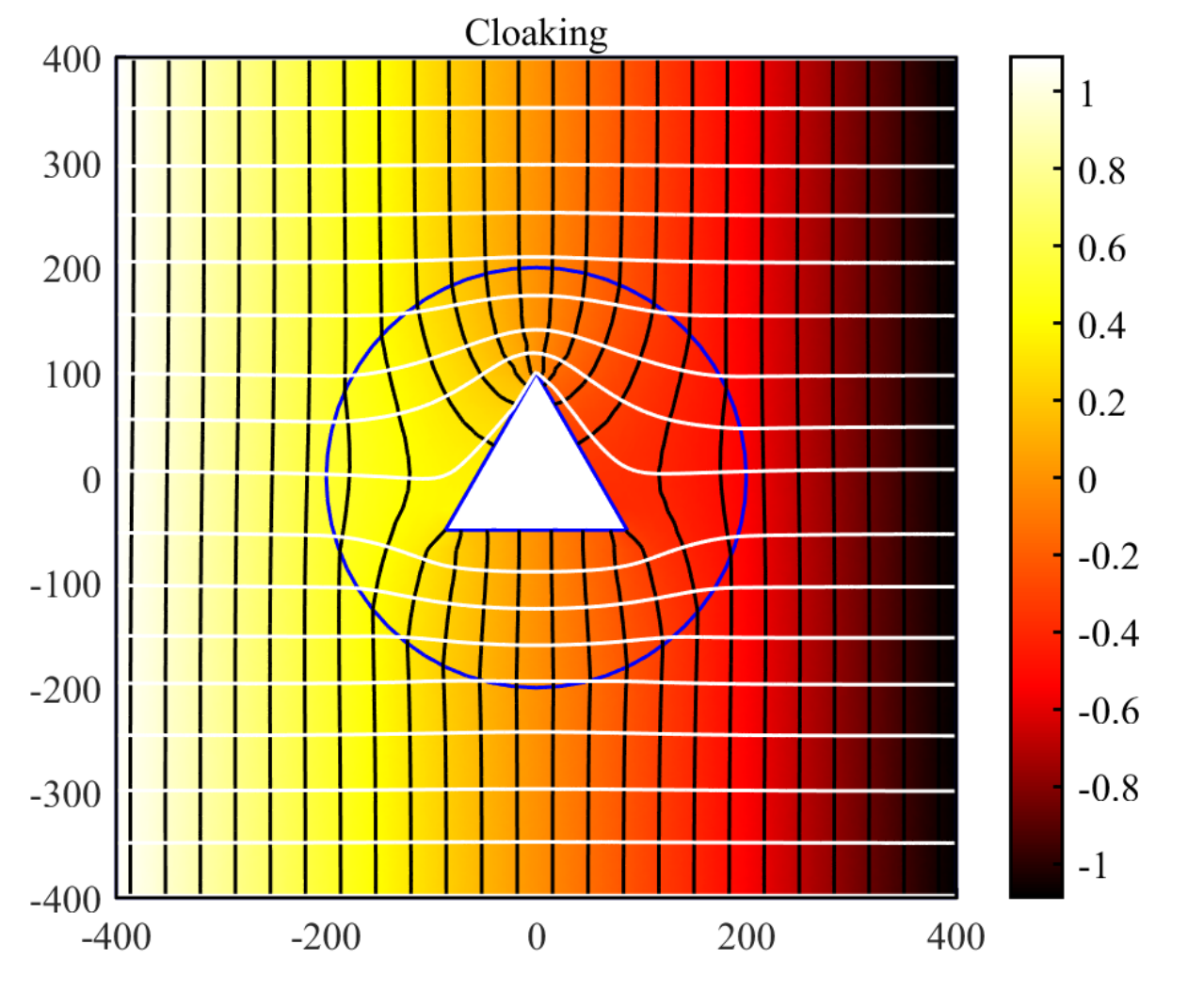}}
	\subfigure[]{
		\includegraphics[width=0.32\linewidth]{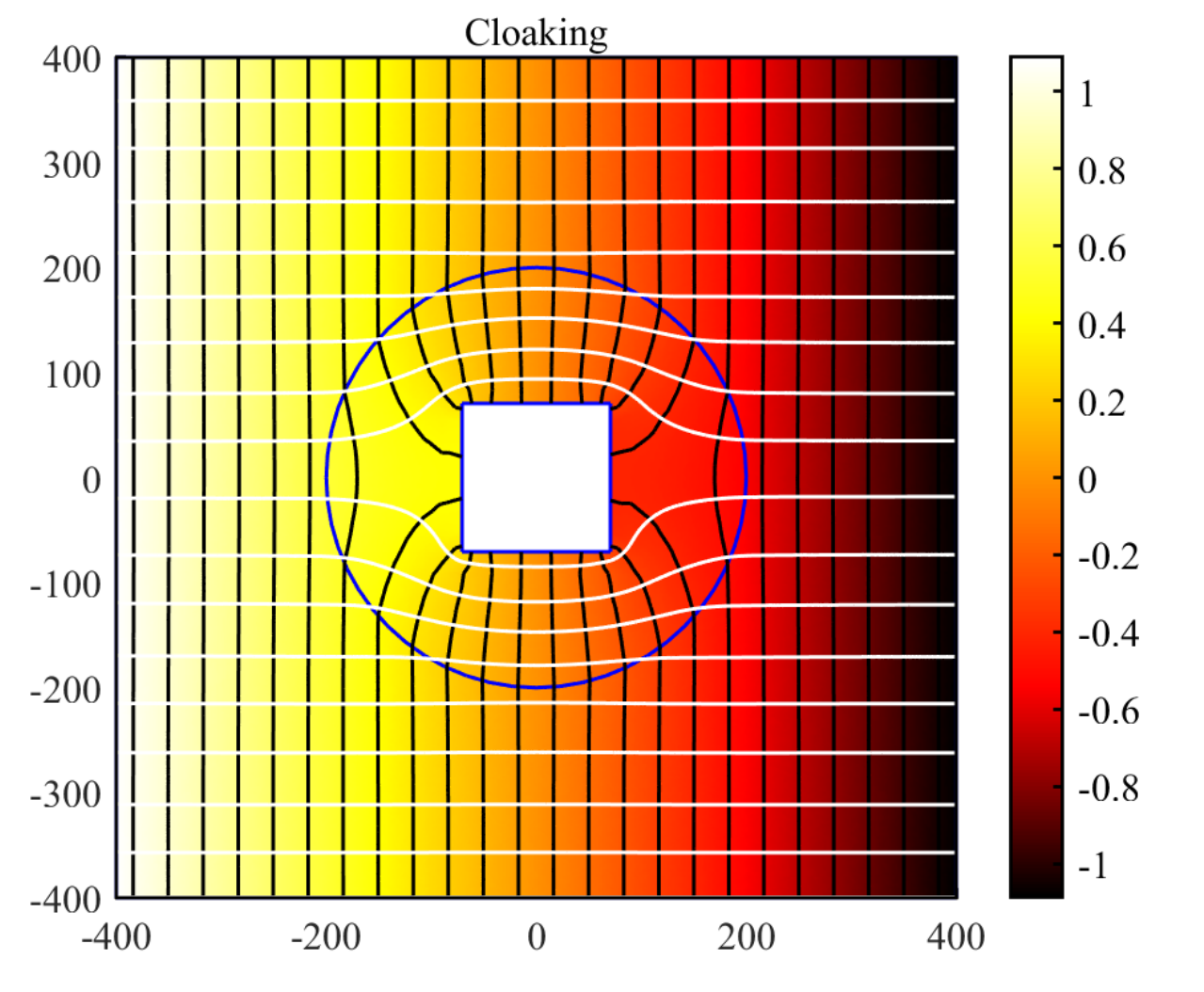}}
	\subfigure[]{
		\includegraphics[width=0.32\linewidth]{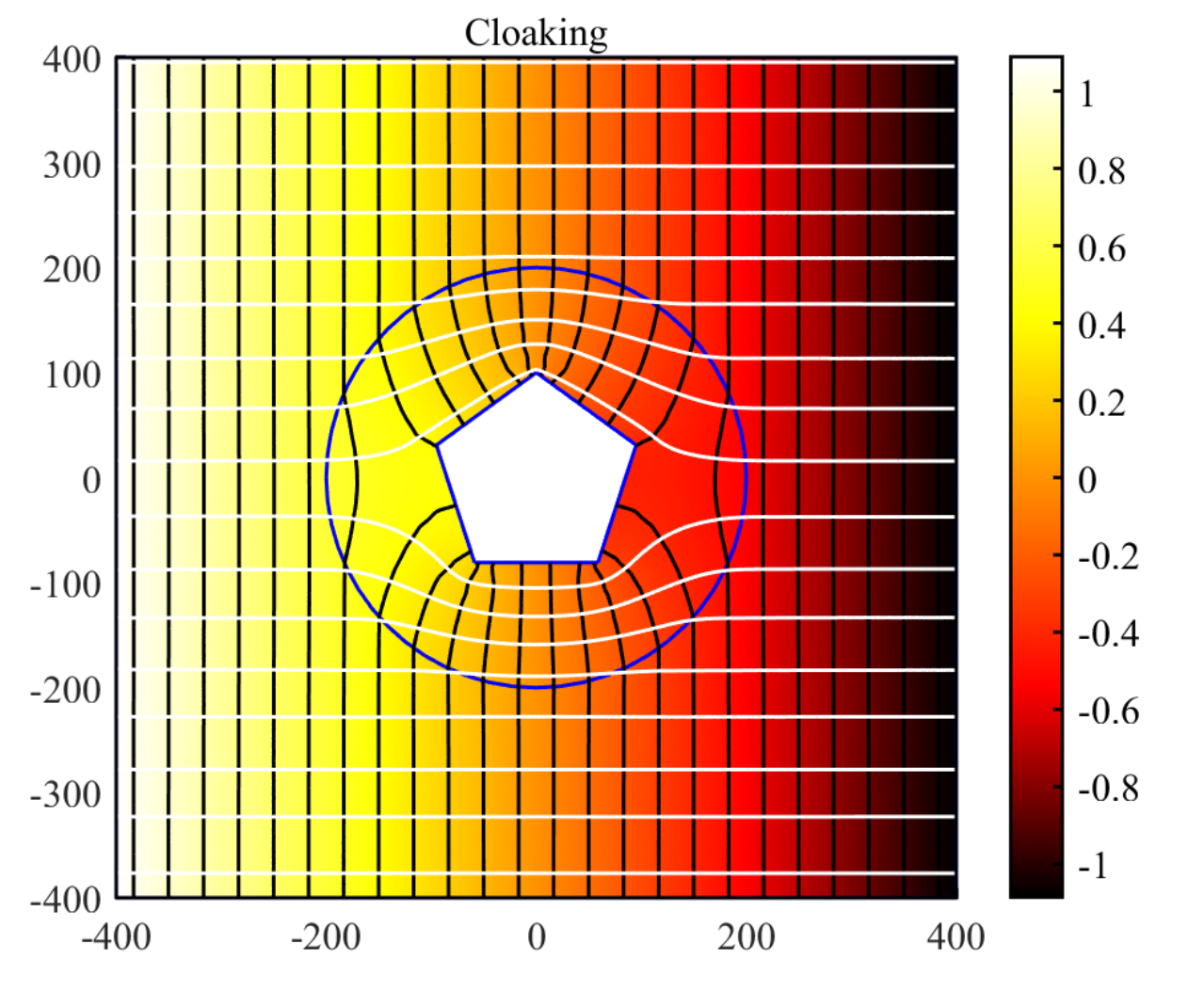}}\\
	\subfigure[]{
		\includegraphics[width=0.32\linewidth]{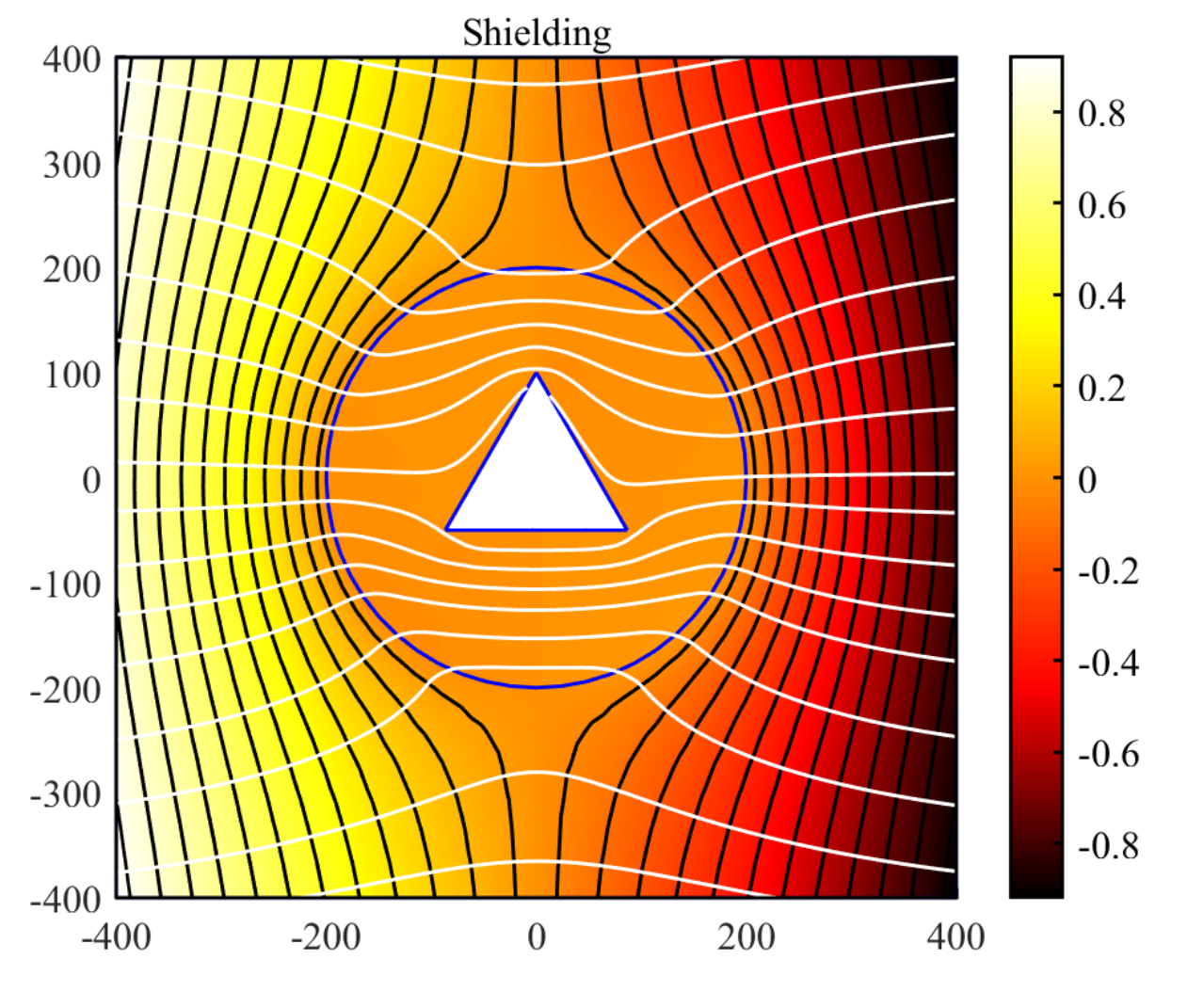}}
	\subfigure[]{
		\includegraphics[width=0.32\linewidth]{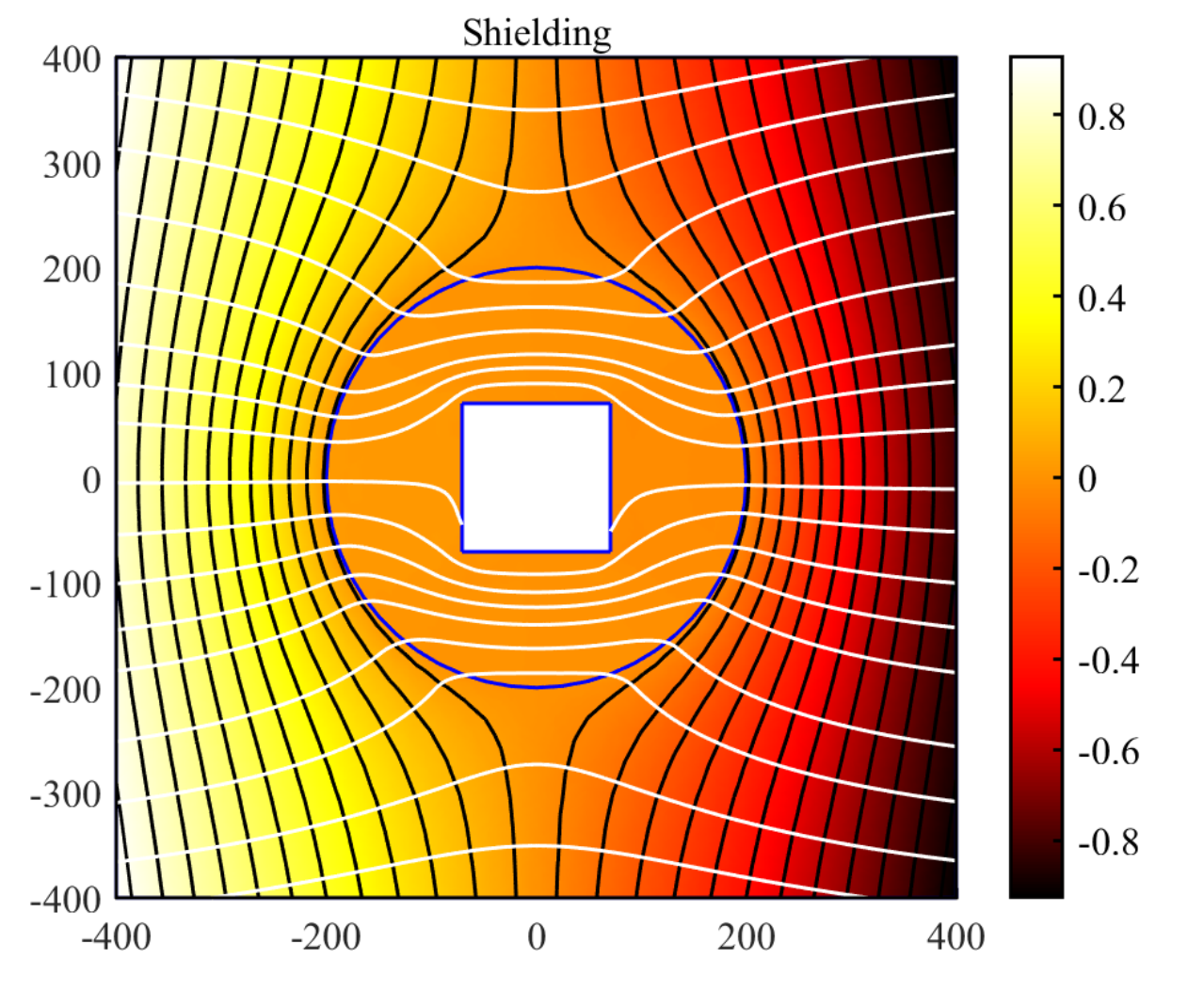}}
	\subfigure[]{
		\includegraphics[width=0.32\linewidth]{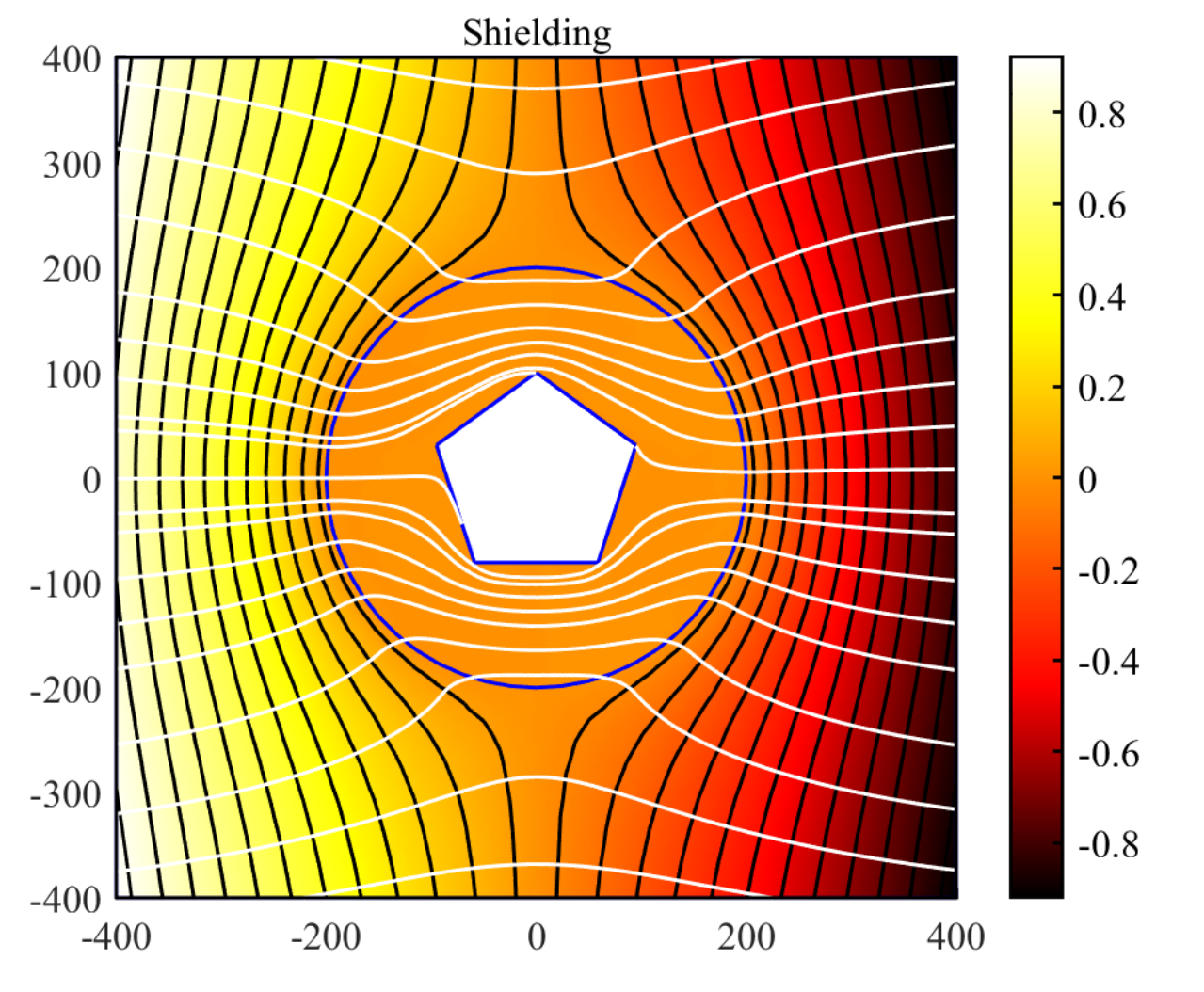}}
	\caption{Comparison of finite-element simulation results for some objects with corners. Numerical (a-f) results for the pressure distribution (colormap) and streamlines (white lines)
   corresponding to cloaking (a-c) and shielding (d-f). The zeta potentials are $\tilde{\zeta}_{0, opt} = -0.0636\ \mathrm{V} (a), -0.0912  \ \mathrm{V} (b), -0.0984  \ \mathrm{V} (c), -0.5568\ \mathrm{V} (d), -0.5904  \ \mathrm{V} (e)$ and $-0.5997  \ \mathrm{V} (f)$, respectively. }\label{fig-corner-shape}
\end{figure}

 We extend numerical simulations by investigating the possibility of cloaking multiple objects placed in close proximity to each other. Figures \ref{fig-cloaking-multi-circle}, \ref{fig-cloaking-multi-ellipse} and \ref{fig-cloaking-multi-circle-ellipse} show good cloaking for multiple circular, elliptical objects  and the combination of them. Here excellent cloaking remains exist. However, the shielding does not exist for multiple objects.
\begin{figure}[H]
	\centering  
	\subfigcapskip=-10pt 
	\subfigure[]{
		\includegraphics[width=0.32\linewidth]{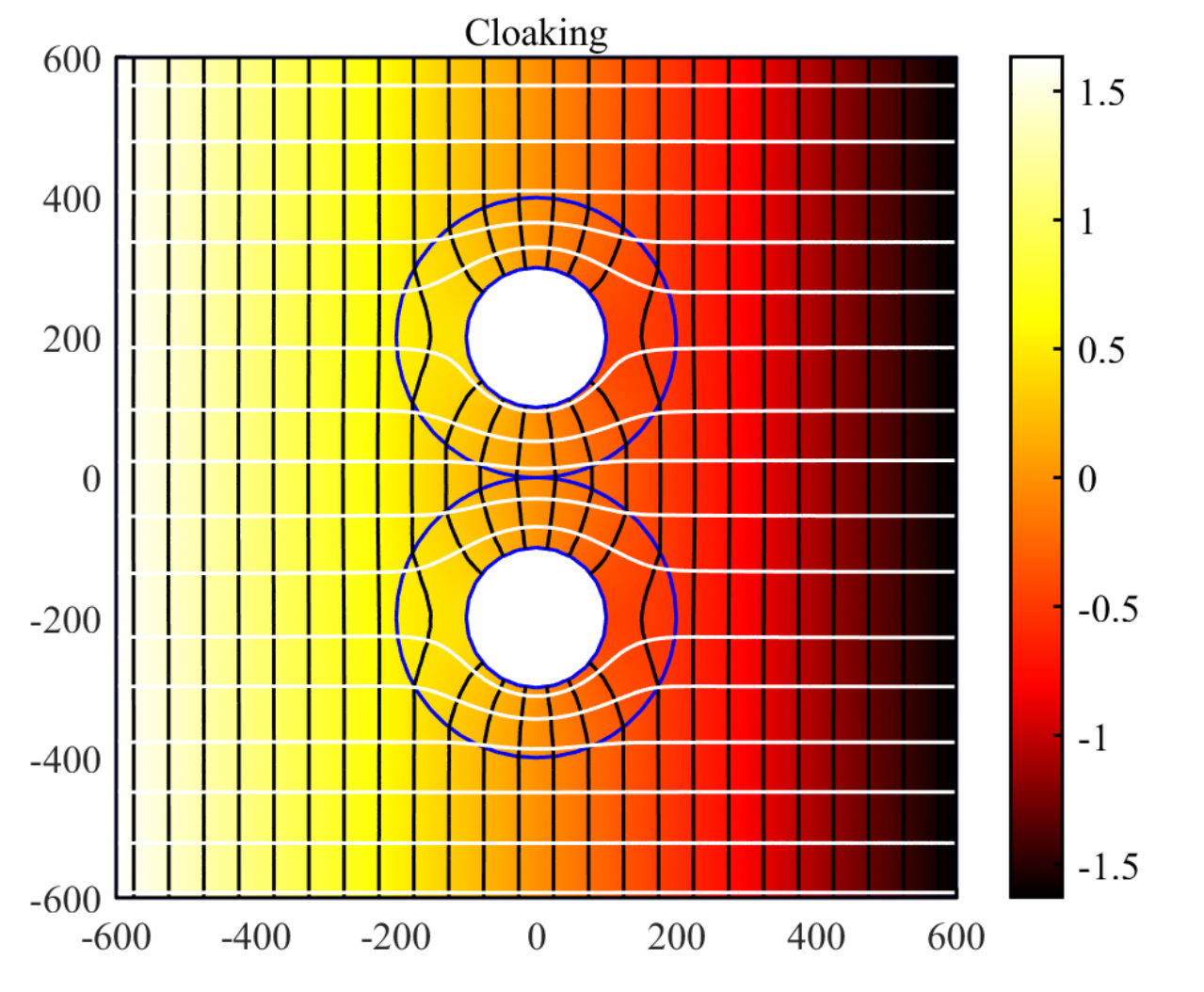}}
	\subfigure[]{
		\includegraphics[width=0.32\linewidth]{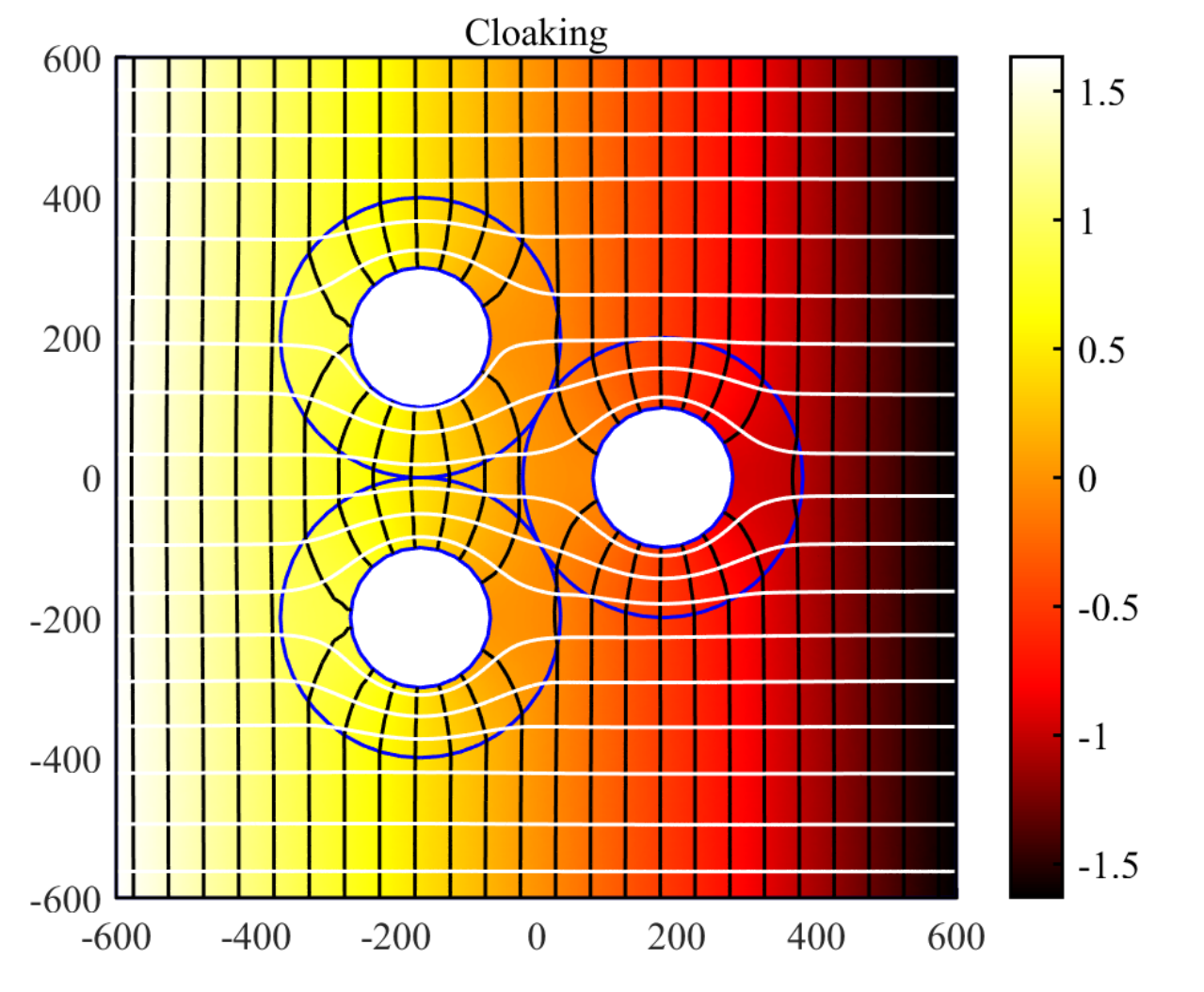}}
	\subfigure[]{
		\includegraphics[width=0.32\linewidth]{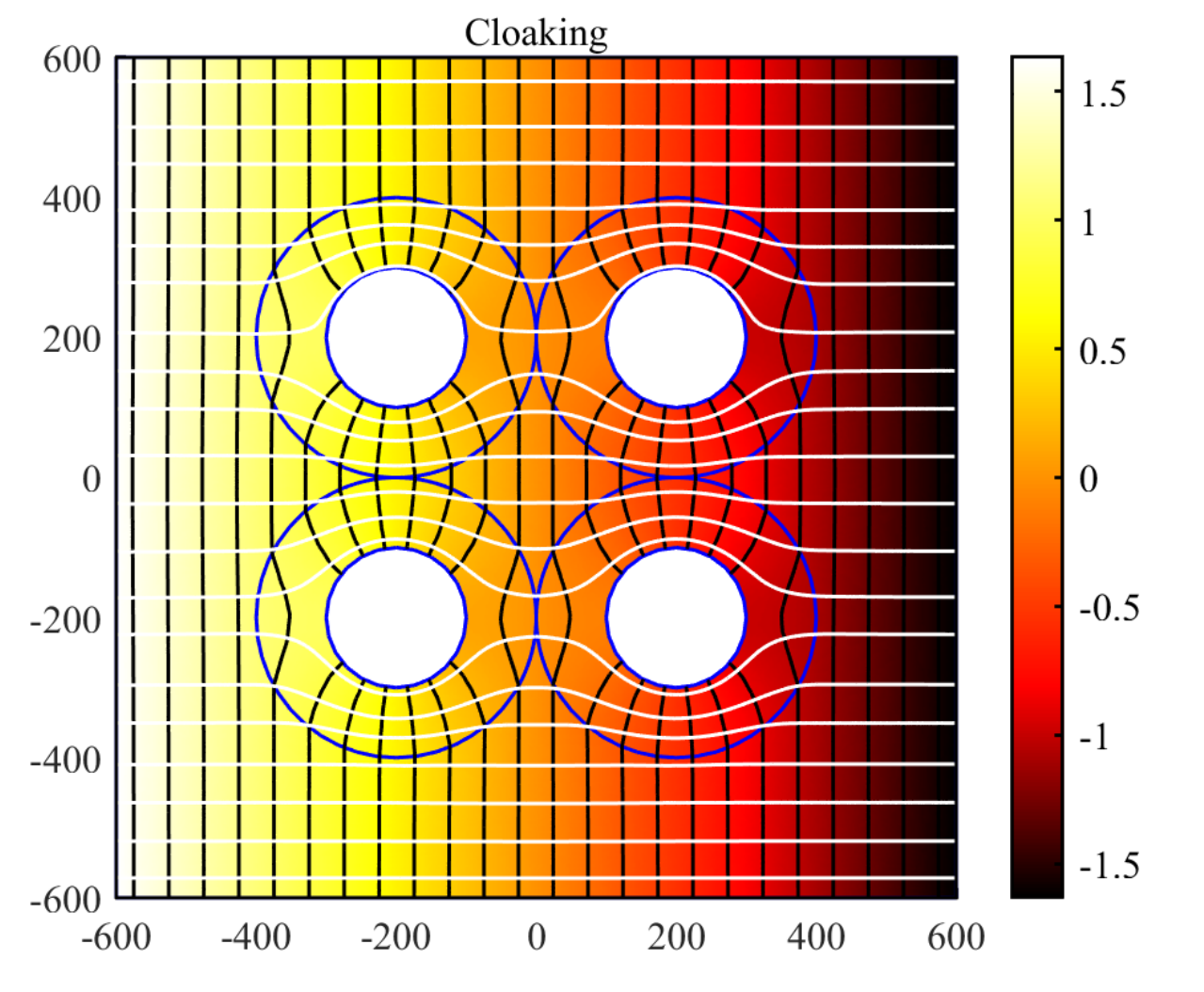}}
	\caption{Cloaking for multiple circular cylinder objects. Here $\tilde{r}_i=100 \ \mu \mathrm{m}$, $\tilde{r}_e=200 \ \mu \mathrm{m}$ and $\tilde{\zeta}_0=-0.128 \,\mathrm{V}$.}\label{fig-cloaking-multi-circle}
\end{figure}

\begin{figure}[H]
	\centering  
	\subfigcapskip=-10pt 
	\subfigure[]{
		\includegraphics[width=0.32\linewidth]{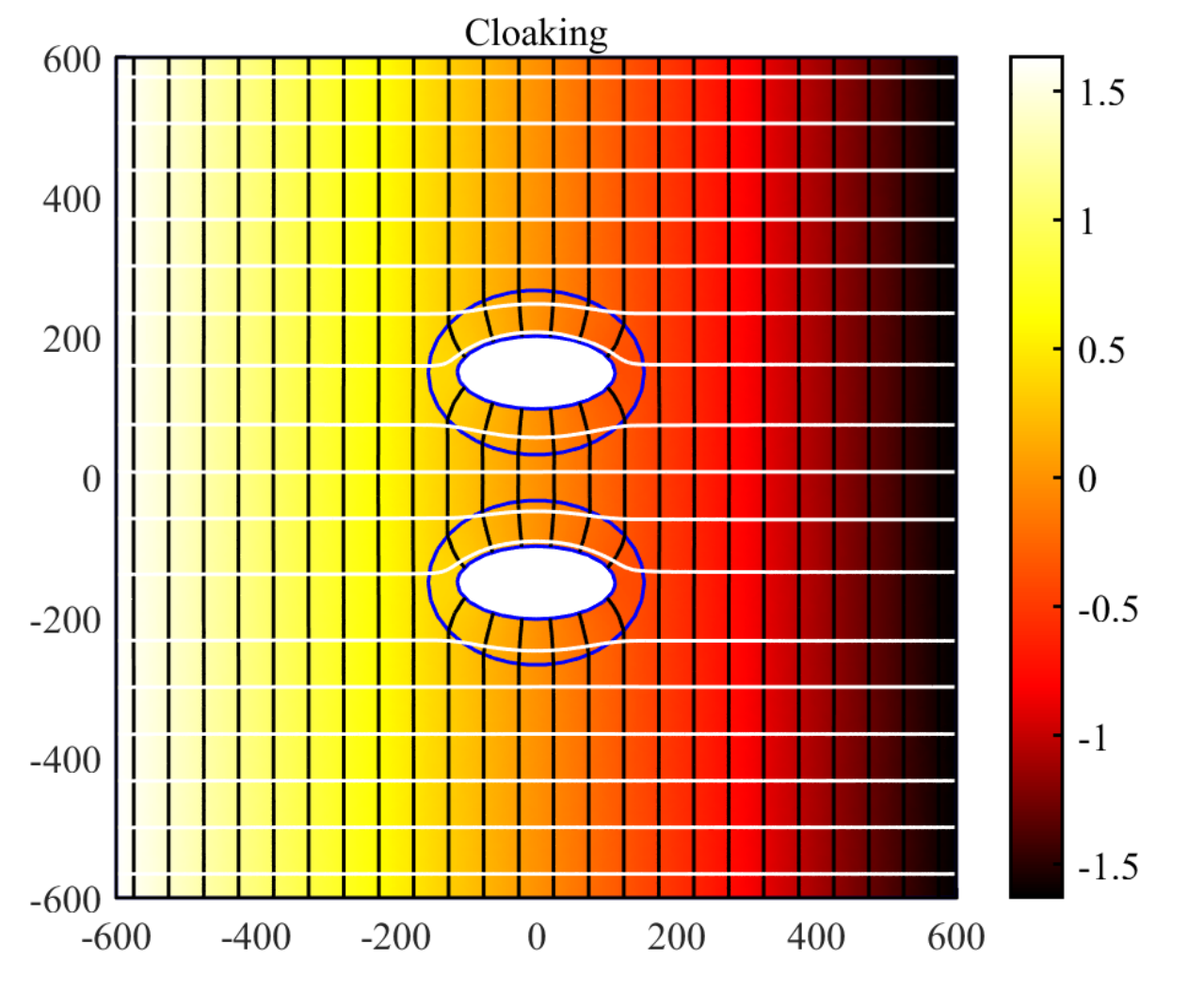}}
	\subfigure[]{
		\includegraphics[width=0.32\linewidth]{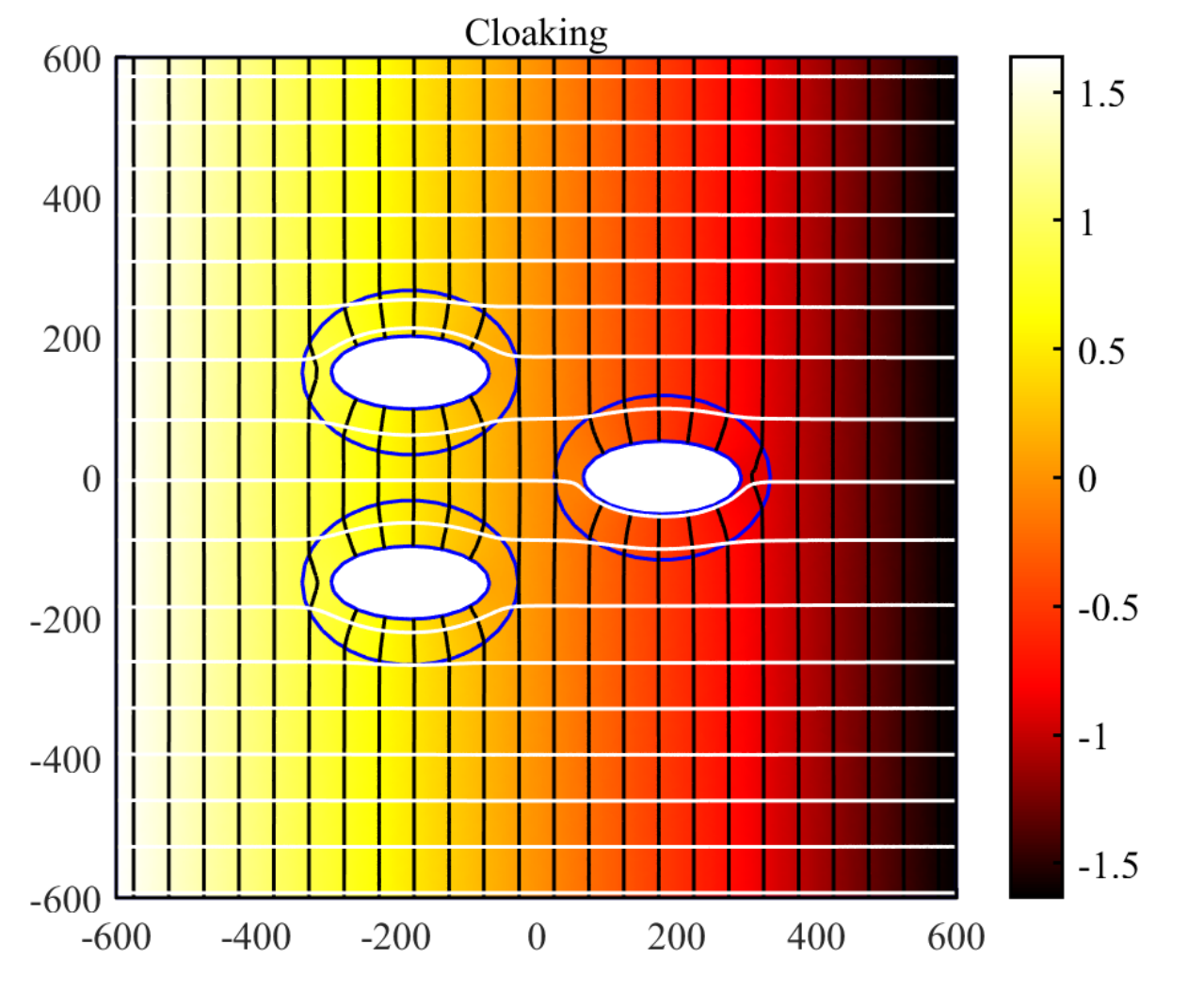}}
	\subfigure[]{
		\includegraphics[width=0.32\linewidth]{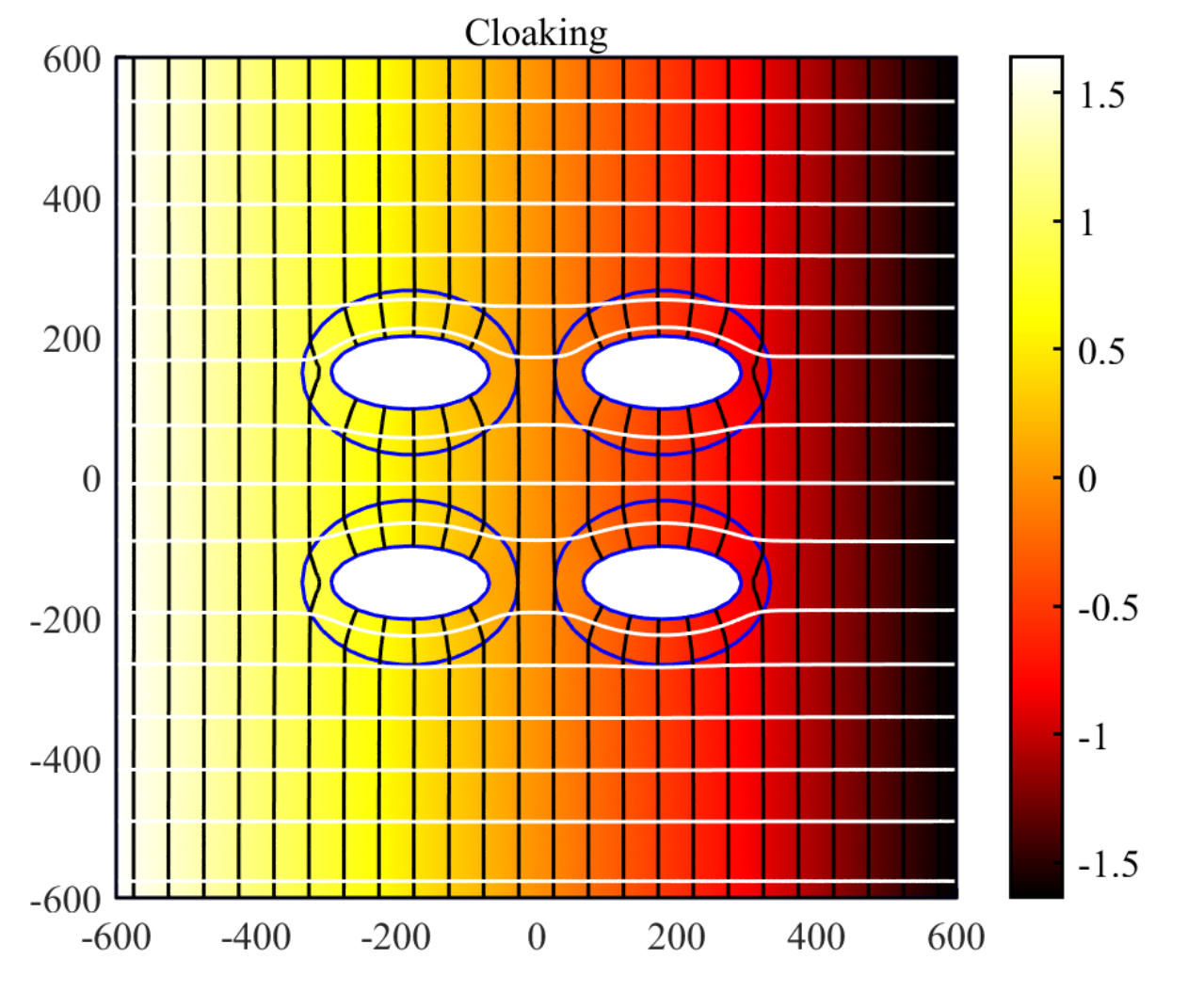}}
	\caption{Cloaking for multiple elliptic cylinder objects. Here $\tilde{\xi}_i=50 \ \mu \mathrm{m}$, $\tilde{\xi}_e=100 \ \mu \mathrm{m}$ and $\tilde{\zeta}_0=-0.1291 \,\mathrm{V}$.}\label{fig-cloaking-multi-ellipse}
\end{figure}

\begin{figure}[H]
	\centering  
	\subfigcapskip=-10pt 
	\subfigure[]{
		\includegraphics[width=0.32\linewidth]{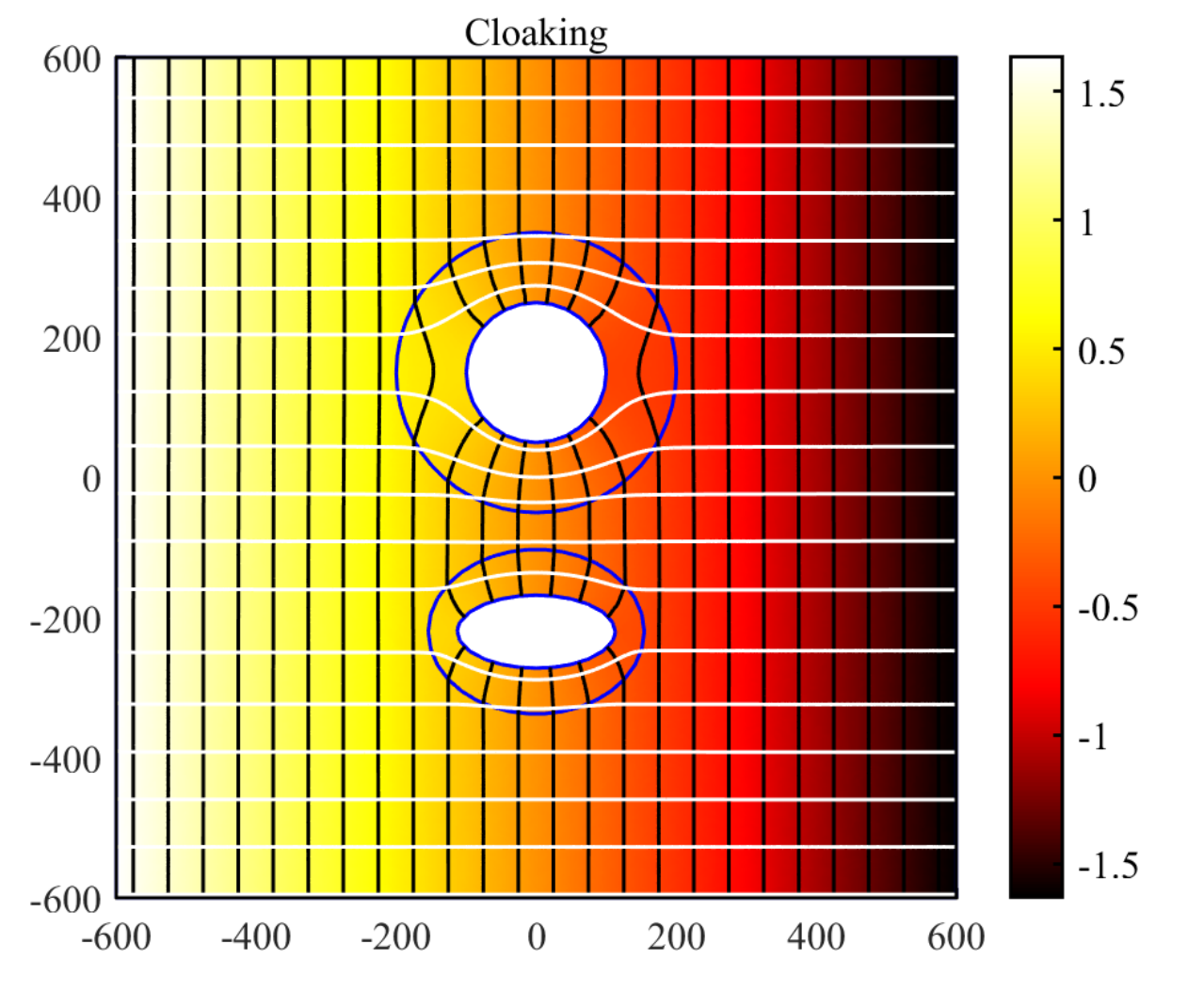}}
	\subfigure[]{
		\includegraphics[width=0.32\linewidth]{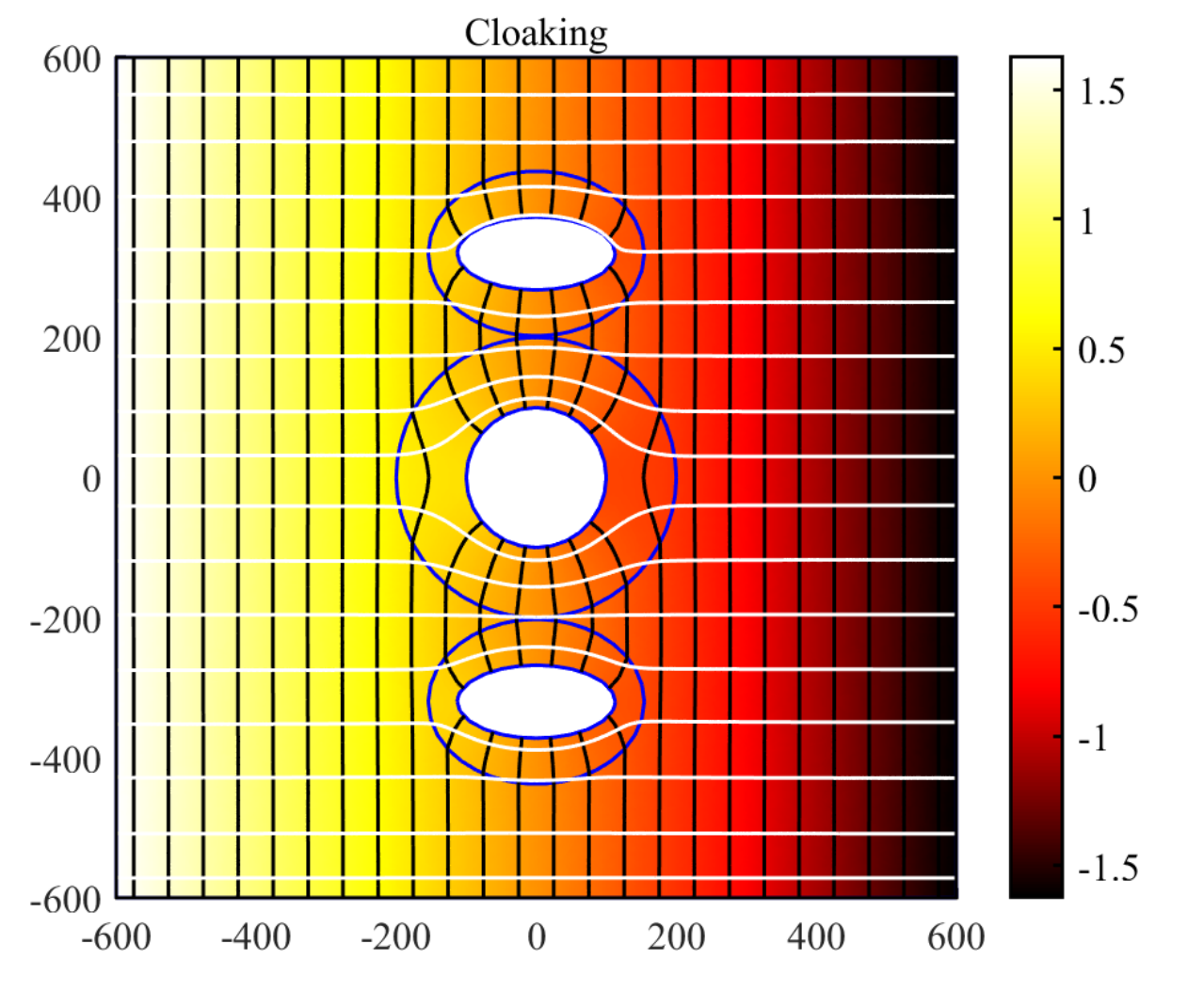}}
	\subfigure[]{
		\includegraphics[width=0.32\linewidth]{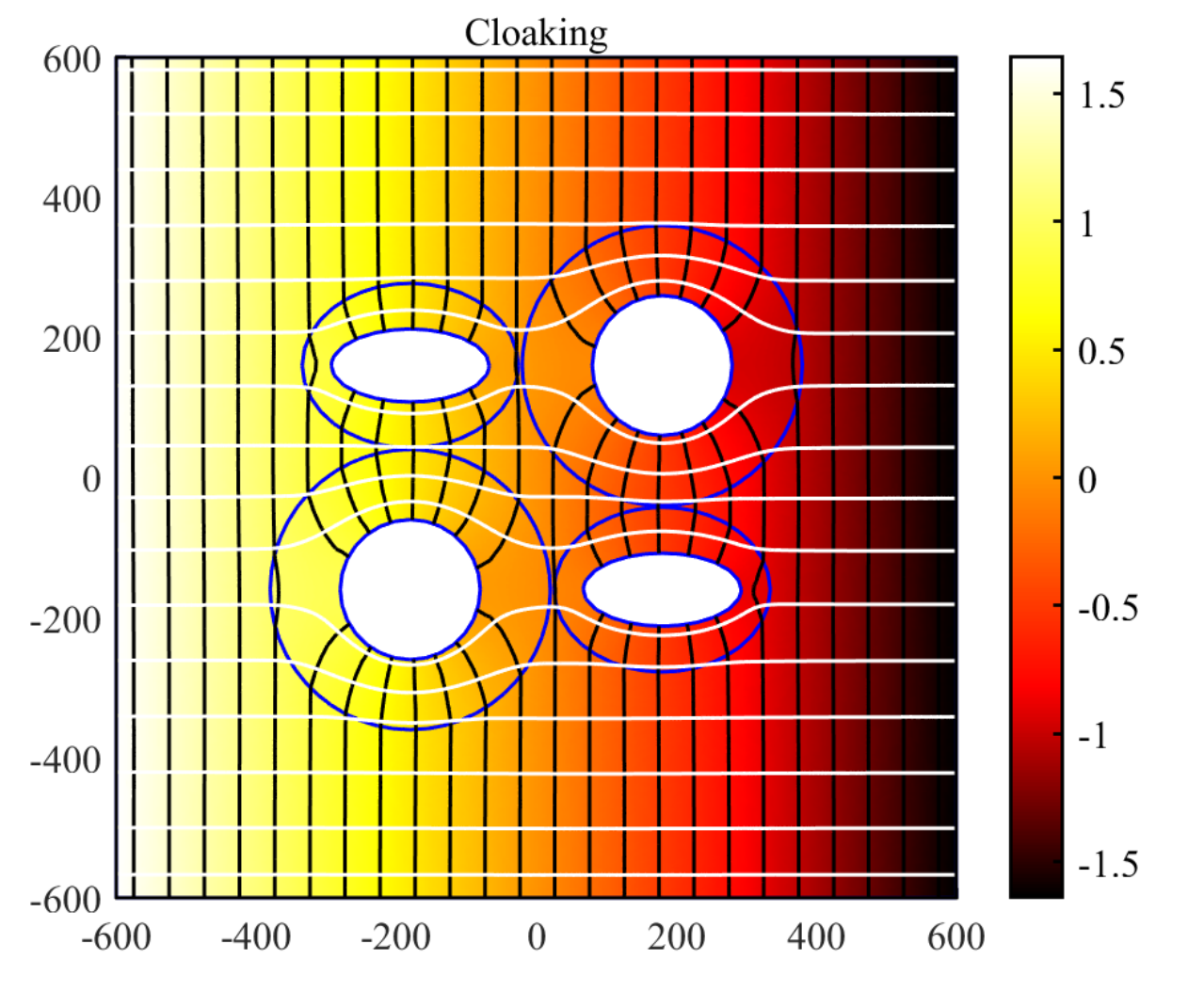}}
	\caption{Cloaking for the combination of circular and elliptic cylinder objects. The parameters are the same as that of Figures \ref{fig-cloaking-multi-circle} and \ref{fig-cloaking-multi-ellipse}.  }\label{fig-cloaking-multi-circle-ellipse}
\end{figure}

\section{Conclusion}\label{sect:6}
In this paper, we consider the hydrodynamic cloaking and shielding of objects in microscale flows. We established a systematic and comprehensive mathematical framework that allows us to derive the shielding and cloaking conditions via electroosmosis for complex geometries. In addition to the theoretical results, extensive numerical experiments were conducted to corroborate the theoretical findings. We would like to emphasize that the above approach is metamaterial-less, namely without using any ``exotic" materials. The corresponding “cloaks” and “shields” presented here have simpler structures than their metamaterial equivalents. Moreover, these structures are easy to control by adapting the gate electrode in physics. The idea can be extended in several directions:
(i) to investigate the enhancement of approximate hydrodynamic cloaking and shielding based on the perfect hydrodynamic cloaking and shielding on the annuls and confocal ellipses using the perturbation theory;
(ii) to study the case when the electric filed is cloaked simultaneously;
(iii) to consider the bigger scale flow.
These new developments will be reported in our forthcoming works.

\section*{Acknowledgment}
The research of H Liu was supported by NSFC/RGC Joint Research Scheme, N\_CityU101/21, ANR/RGC Joint Research Scheme, A-CityU203/19, and the Hong Kong RGC General Research Funds (projects 12302919, 12301420 and 11300821).

\end{document}